\documentclass[11pt,reqno,twoside]{article}

\usepackage[hang]{footmisc}
\usepackage{lipsum}

\usepackage{fixltx2e} 

\usepackage{cmap} 

\usepackage[T1]{fontenc}
\usepackage[utf8]{inputenc}
\usepackage{graphicx}
\usepackage{placeins}
\usepackage{enumerate}
\usepackage{algcompatible}
\usepackage{subcaption}
\usepackage{verbatim}
\newcommand{\comments}[1]{}

\usepackage{soul}

\usepackage{setspace}

\let\counterwithin\relax  
\usepackage{lmodern} 
\usepackage[scale=0.88]{tgheros} 

\usepackage{bm} 

\usepackage{bbold}

\usepackage{pgfplots}

\usepackage{amsmath,amsbsy,amsgen,amscd,amsthm,amsfonts,amssymb} 

\usepackage[centering,top=1.1in,bottom=1.3in,left=0.9in,right=0.9in]{geometry}

\usepackage{titling}
\usepackage{musicography}
\graphicspath{ {./images/} }

\usepackage[sf,bf,compact]{titlesec}

\usepackage{booktabs,longtable,tabu} 
\usepackage[font=small,margin=30pt,labelfont={sf,bf},labelsep={space}]{caption}

\definecolor{dark-gray}{gray}{0.3}
\definecolor{dkgray}{rgb}{.4,.4,.4}
\definecolor{dkblue}{rgb}{0,0,.5}
\definecolor{medblue}{rgb}{0,0,.75}
\definecolor{rust}{rgb}{0.5,0.1,0.1}

\usepackage{url}
\usepackage[colorlinks=true]{hyperref}
\hypersetup{linkcolor=dkblue}    
\hypersetup{citecolor=rust}      
\hypersetup{urlcolor=rust}     

\usepackage[final]{microtype} 

\newtheoremstyle{myThm} 
    {\topsep}                    
    {\topsep}                    
    {\itshape}                   
    {}                           
    {\sffamily\bfseries}                   
    {.}                          
    {.5em}                       
    {}  

\newtheoremstyle{myRem} 
    {\topsep}                    
    {\topsep}                    
    {}                   
    {}                           
    {\sffamily}                   
    {.}                          
    {.5em}                       
    {}  

\newtheoremstyle{myDef} 
    {\topsep}                    
    {\topsep}                    
    {}                   
    {}                           
    {\sffamily\bfseries}                   
    {.}                          
    {.5em}                       
    {}  

\theoremstyle{myThm}
\newtheorem{theorem}{Theorem}[section]
\newtheorem{lemma}[theorem]{Lemma}

\newtheorem{assumption}[theorem]{Assumption}

\theoremstyle{myRem}
\newtheorem{remarkx}[theorem]{Remark}
 \newenvironment{remark}
  {\pushQED{\qed}\remarkx}
  {\popQED\endremarkx}

\theoremstyle{myDef}
\newtheorem{definition}[theorem]{Definition}

\usepackage{fancyhdr}
\let\originalleft\left
\let\originalright\right
\renewcommand{\left}{\mathopen{}\mathclose\bgroup\originalleft}
\renewcommand{\right}{\aftergroup\egroup\originalright}

\usepackage{mathtools}
\mathtoolsset{centercolon}  

\definecolor{mygreen}{rgb}{0.1,0.75,0.2}

\providecommand{\mathbbm}{\mathbb} 

\newcommand{\R}{\mathbbm{R}}

\newcommand{\N}{\mathbbm{N}}

\newcommand{\G}{\mathcal{G}}

\renewcommand{\phi}{\varphi}

\newcommand{\tildeu}{\tilde{u}}
\newcommand{\tildev}{\tilde{v}}

\usepackage[font = small, margin=30pt]{caption}

\usepackage[]{algorithm}
\usepackage{algpseudocode}
\usepackage{enumerate}

\usepackage{graphicx}

\usepackage{authblk}
\usepackage[square,numbers]{natbib}
\usepackage{chngcntr}
\usepackage{mathrsfs}
\counterwithin{table}{section}
\counterwithin{algorithm}{section}

\newcommand{\sfd}{\mathsf{d}}
\newcommand{\tcov}{\mathsf{Cov}}

\newcommand{\tdiam}{\Delta}
\newcommand{\tvol}{\mathsf{vol}}
\newcommand{\ttrace}{\mathsf{tr}}

\newcommand{\indicator}{{\bf{1}}}

\newcommand{\iid}{\stackrel{\mathsf{i.i.d.}}{\sim}}

\newcommand{\E}{\mathbb{E}}

\renewcommand{\P}{\mathbb{P}}

\newcommand{\tildeT}{\widetilde{T}}

\newcommand{\tildek}{\tilde{k}}

\newcommand{\tildeh}{\widetilde{h}}

\newcommand{\tildeH}{\widetilde{H}}

\newcommand{\hatSigma}{\widehat{\Sigma}}

\newcommand{\bfu}{\mathbf{u}}

\newcommand{\mfs}{\mathfrak{s}}

\newcommand{\hatk}{\hat{k}}

\newcommand{\hatt}{\hat{t}}

\newcommand{\hatC}{\widehat{C}}

\newcommand{\hatT}{\hat{T}}

\newcommand{\hatV}{\hat{V}}

\newcommand{\hattheta}{\hat{\theta}}

\newcommand{\mcD}{\mathcal{D}}
\newcommand{\mcE}{\mathcal{E}}
\newcommand{\mcF}{\mathcal{F}}
\newcommand{\mcG}{\mathcal{G}}
\newcommand{\mcH}{\mathcal{H}}

\newcommand{\mcK}{\mathcal{K}}
\newcommand{\mcL}{\mathcal{L}}

\newcommand{\mcS}{\mathcal{S}}

\newcommand{\mcU}{\mathcal{U}}

\usepackage{multirow}
\usepackage{enumitem}

\newcommand{\inparen}[1]{\left(#1\right)}             
\newcommand{\inbraces}[1]{\left\{#1\right\}}           
\newcommand{\insquare}[1]{\left[#1\right]}             

\newcommand{\normn}[1]{\ensuremath{\lVert #1 \rVert}}

\newcommand{\abs}[1]{\ensuremath{\left\lvert #1 \right\rvert}}
\newcommand{\inp}[1]{\ensuremath{\left\langle #1 \right\rangle}}

\newcommand{\tvar}{\text{var}}

\usepackage[scr = esstix, cal = cm, frak=euler]{mathalfa}

\definecolor{mygreen}{rgb}{0.1,0.75,0.2}

\usepackage[scr = esstix, cal = cm, frak=euler]{mathalfa}

\title{Covariance Operator Estimation via Adaptive Thresholding} 

\author{Omar Al-Ghattas and Daniel Sanz-Alonso}

\date{University of Chicago}

\vspace{.25in}

\makeatletter\@addtoreset{section}{part}\makeatother%
\numberwithin{equation}{section}

\newcommand{\upperRomannumeral}[1]{\uppercase\expandafter{\romannumeral#1}}

\renewcommand{\hat}{\widehat}

\begin{document}
\maketitle 



\abstract{
This paper studies sparse covariance operator estimation for nonstationary processes with sharply varying marginal variance and small correlation lengthscale. We introduce a covariance operator estimator that adaptively thresholds the sample covariance function using an estimate of the variance component. Building on recent results from empirical process theory, we derive an operator norm bound on the estimation error in terms of the sparsity level of the covariance and the expected supremum of a normalized process. Our theory and numerical simulations demonstrate the advantage of adaptive threshold estimators over universal threshold and sample covariance estimators in nonstationary settings.


\section{Introduction}
This paper investigates  sparse covariance operator estimation in an infinite-dimensional function space setting.
 Covariance estimation is a fundamental task that arises in numerous scientific applications and data-driven algorithms \cite{anderson1958introduction, fan2008high, hardoon2004canonical, tharwat2017linear, ghattas2022non, al2024ensemble}. The sample covariance is arguably the most natural estimator, and its error in both finite and infinite dimension can be controlled by a notion of  \emph{effective dimension} that accounts  for spectrum decay \cite{koltchinskii2017concentration,lounici2014high}. 
 However, a rich literature has identified sparsity assumptions under which other estimators drastically outperform the sample covariance in finite high-dimensional settings \cite{bickel2008covariance, bickel2008regularized, el2008operator, cai2012optimal, cai2016estimating, wainwright2019high}. This work contributes to the largely unexplored subject of \emph{sparse} covariance operator estimation in infinite dimension.
  Through rigorous theory and complementary numerical simulations, we demonstrate the benefit of adaptively thresholding the sample covariance.  
 In doing so, this paper contributes to the emerging literature on 
 operator estimation and learning \cite{kovachki2024operator, de2023convergence, mollenhauer2022learning}, emphasizing the importance of exploiting structural assumptions in the design and analysis of estimators. 

 In this work, we investigate approximate sparsity structure that arises in the nonstationary regime where the marginal variance varies widely in the domain and the correlation lengthscale is small relative to the size of the domain. Covariance estimation for such nonstationary processes is crucial, for instance, in numerical weather forecasting, where local and highly nonstationary phenomena such as cloud formation can significantly impact global forecasts. }
To study the sparse highly nonstationary regime where the marginal variance varies widely in the domain and the correlation lengthscale is small, we consider a novel class of covariance operators that satisfy a \emph{weighted} $L_q$-sparsity condition. 
For covariance operators in this class, we establish a bound on the operator norm error of the \textit{adaptive threshold} estimator in terms of two dimension-free quantities: the sparsity level and the expected supremum of  a normalized process. Unlike existing theory that considered unweighted $L_q$-sparsity (see Section \ref{sec:LitReviewInfiniteDims} for a review) our theory allows for covariance models with unbounded marginal variance functions. We then compare our adaptive threshold estimator with other estimators of interest, namely the \textit{universal threshold} and \textit{sample covariance} estimators. 
For universal thresholding, we prove a lower bound that is larger than our upper bound for adaptive thresholding. 
In addition, we numerically investigate adaptive thresholding for highly nonstationary covariance models defined through a scalar parameter that controls both the correlation lengthscale and the range of the marginal variance function. In the challenging case where the lengthscale is small and the range of marginal variances is large, we show an exponential improvement in sample complexity of the adaptive threshold estimator compared to the sample covariance. Our numerical simulations clearly demonstrate that universal threshold and sample covariance estimators fail in this regime. 

  By focusing on the infinite-dimensional setting, our theory reveals the key dimension-free quantities that control the estimation error, and 
 further explains how the correlation lengthscale and the marginal variance function affect the estimation problem. While our infinite-dimensional analysis helps uncover such a connection between interpretable model assumptions and complexity of the estimation task, it poses new challenges that require novel technical tools. In this work, we leverage recent results from empirical process theory \cite{al2025sharp}
 to control the error in the estimation of the variance component used to adaptively choose the thresholding radius. 
 Our infinite-dimensional perspective agrees with recent work in operator learning that advocates for the development of theory in infinite dimension, as opposed to the traditional approach in functional data analysis, where it is common to 
 study estimators constructed by first discretizing the data \cite{ramsay2002functional, zhang2016sparse}. We will discuss the differences between the two approaches and compare our theory with the existing infinite-dimensional covariance estimation literature \cite{al2023covariance, fang2023adaptive}. More broadly, infinite-dimensional analyses that delay introducing discretization have led to numerous theoretical insights and computational advances in mathematical statistics \cite{gine2021mathematical}, Bayesian inverse problems \cite{stuart2010inverse}, Markov chain Monte Carlo \cite{cotter2013mcmc}, importance sampling and particle filters \cite{agapiou2017importance}, ensemble Kalman algorithms \cite{sanz2023analysis}, graph-based learning \cite{garcia2018continuum}, stochastic gradient descent \cite{latz2021analysis}, and numerical analysis and control \cite{zuazua2005propagation}, among many others.

\subsection{Related Work} \label{ssec:RelatedWork}
For later reference and discussion, here we summarize unweighted and weighted approximate sparsity assumptions in the finite-dimensional thresholded covariance estimation literature, as well as the main sparsity assumptions that have been considered in the infinite-dimensional setting. 
\subsubsection{Finite Dimension}\label{sec:LitReviewFiniteDims}
Thresholding estimators in the finite high-dimensional setting were introduced in the seminal work \cite{bickel2008covariance} and further studied in \cite{rothman2009generalized, cai2012optimal, cai2016estimating, cai2011adaptive}. Given $d_X$-dimensional i.i.d. samples $X_1,\dots, X_N$ from a centered sub-Gaussian  distribution with covariance $\Sigma$, the authors demonstrated that thresholding the sample covariance matrix, i.e. $\hatSigma^{\mathsf{U}}_{\rho_N}= (\hatSigma_{ij} \indicator \{|\hatSigma_{ij}| \ge \rho_N\})$ --- where the superscript $\mathsf{U}$ denotes \textit{universal} --- performed well over the class of covariance matrices with bounded marginal variances and satisfying an $\ell_q$-sparsity condition whenever the thresholding parameter $\rho_N$ is chosen appropriately. Specifically, whenever $\Sigma$ belongs to the class $\mcU_q(R_q, M)$ for $q\in [0,1)$, $R_q >0$ and $M > 0$ where
\begin{align}\label{eq:SparseMatrixClass}
    \mcU_q := 
    \mcU_q(R_q, M) = \inbraces{\Sigma \in \R^{d_X\times d_X}: \Sigma\succ 0,~ \max_{i \le d_X} \Sigma_{ii} \le M, ~ \max_{i \le d_X} \sum_{j=1}^{d_X} |\Sigma_{ij}|^q \le R_q^q},
\end{align}
then the operator norm error of universal thresholding estimators is bounded above (up to universal constants) by $R_q^q (M \log d_X / N)^{(1-q)/2}.$ The bounded marginal variance assumption is crucial to this theory as it was shown that $\rho_N$ must scale with $M$ in order for the high probability guarantees on $\hatSigma^{\mathsf{U}}_{\rho_N}$ to hold. In \cite{cai2011adaptive}, the authors argued that such a bounded variance assumption effectively converted a heteroscedastic problem of covariance estimation into a worst-case homoscedastic one in which $\Sigma_{ii} = M$ for all $i$ for the purposes of choosing a universal thresholding radius. This is problematic whenever (i) no natural upper bound on the marginal variances is known and (ii) the marginal variances vary over a large range.   
They instead considered the adaptively thresholded covariance estimator $\hatSigma^A_{\rho_N}= (\hatSigma_{ij} \indicator \{|\hatSigma_{ij}| \ge \rho_N \hatV_{ij}^{1/2} \})$, where $\hatV_{ij}$ is the sample version 
of the variance component $V_{ij}:= \tvar(X_i X_j)$. This was shown to be optimal over the larger weighted $\ell_q$-sparsity covariance matrix class
\begin{align}\label{eq:SparseWeightedMatrixClass}
    \mcU^*_q:=
    \mcU^*_q(R_q) = \inbraces{\Sigma \in \R^{d_X \times d_X}: 
    \Sigma\succ 0,~ \max_{i \le d_X} \sum_{j=1}^{d_X} (\Sigma_{ii} \Sigma_{jj})^{\frac{1-q}{2}}
    |\Sigma_{ij}|^q \le R_q^q},
\end{align}
 with operator norm error bounded above (up to universal constants) by $R_q^q (\log d_X / N)^{(1-q)/2}.$  It was further shown that universal thresholding was sub-optimal over the same class. Minimax lower bounds proving the optimality of universal thresholding were studied in \cite{cai2012optimal}. Distributional assumptions were significantly relaxed by allowing for dependence  in \cite{chen2013covariance}, which analyzed thresholding in the high-dimensional time series setting.

\subsubsection{Infinite Dimension}\label{sec:LitReviewInfiniteDims}
In the infinite-dimensional setting, the covariance (operator) estimation problem under sparsity-type constraints has received far less attention. In \cite{al2023covariance}, the authors consider i.i.d. draws of an infinite-dimensional Gaussian process defined over $D=[0,1]^d$ with covariance operator $C$ and corresponding covariance function $k: D \times D \to \R$, denoted $u_1,\dots, u_N \iid \text{GP}(0, C)$. They generalize \cite{bickel2008regularized} to infinite dimensions by considering Gaussian processes that are almost surely continuous with covariance operators $C \in \mcK_q,$ where 
\begin{align}\label{eq:SparseOperatorClass}
    \mcK_q := \mcK_q(R_q, M) = \inbraces{C \succ 0: ~ \sup_{x \in D} k(x,x) \le M, ~ \sup_{x \in D} \int_D |k(x,y)|^q dy \le R_q^q}.
\end{align}
This class naturally captures approximate sparsity of the covariance, which may arise, for example, from  
decay of correlations of the Gaussian process at different locations in the domain. It was then shown that for a universally thresholded covariance operator estimator, i.e. with covariance function $\hatk(x,y) \indicator \{ |\hatk(x,y)| \ge \rho_N\}$, the operator norm error was bounded above by $R_q^q ((\E[\sup_{x \in D} u(x)])^2/ N)^{(1-q)/2},$ which is a dimension-free quantity. Further, the authors demonstrate that if the covariance function is stationary and depends on a lengthscale parameter $\lambda$, then universally thresholded estimators enjoy an exponential improvement over the standard sample covariance estimator in the small $\lambda$ asymptotic. Notice that here and throughout this paper, $d$ represents the dimension of the physical domain $D = [0,1]^d$ and should not be confused with the dimension of the data points $\{u_n\}_{n=1}^N$, which here represent infinite-dimensional functions.

Motivated by applications in functional data analysis, \cite{fang2023adaptive} considers covariance estimation for a multi-valued process $\bfu: D \to \R^p$ given independent data 
\begin{align*}
    \bfu_n(\cdot) = \bigl(u_{n1}(\cdot), u_{n2}(\cdot), \dots, u_{n p}(\cdot)\bigr)^\top,
    \quad n=1,\dots,N.
\end{align*}
The covariance function now takes the form
\begin{align*}
    \mathbf{K}: D \times D \to \R^{p \times p}, 
    \qquad 
    \mathbf{K}(x, y) = \tcov \bigl(\bfu_n(x), \bfu_n(y)\bigr)= [k_{ij}(x,y)]_{i,j=1}^p,
\end{align*}
where $k_{ij}: D\times D \to \R $ is a component covariance function. 
Then, \cite{fang2023adaptive} studies the 
setting 
 in which the number of component is much larger than the sample size, i.e. 
$p \gg N$, and under the assumption that the true covariance function belongs to the class
\begin{align}\label{eq:SparseMatrixOperatorClassFang}
    \begin{split}
    \mcG_q(R_q, \varepsilon) = 
    \bigg \{
    &\mathbf{K} \succeq 0: 
     \max_{i \le p} \sum_{j=1}^p (\normn{k_{ii}}_\infty\normn{k_{jj}}_\infty)^{\frac{1-q}{2}} \normn{ k_{ij} }^q_{\text{HS}} \le R_q^q,
     \max_{i \le p} \normn{k_{ii}^{-1}}_\infty \normn{k_{ii}}_\infty \le \frac{1}{\varepsilon}
    \bigg \}.
    \end{split}
\end{align}
Here, we denote by $\normn{k}^2_{\text{HS}} = \iint k^2( x,  y ) \,  dx dy$ the Hilbert-Schmidt norm, and we denote $\normn{k}_{\infty} = \sup_{x,y \in D}|k(x,y)|.$ This class generalizes the class $\mcU^*_q$ and the authors obtain analogous upper bounds to those of \cite{cai2011adaptive} for the error of estimation under a functional version of the matrix $\ell_1$-norm. We provide further comparisons to this line of work in Remark~\ref{rem:globalSparsityAssumption}. Another popular approach in the functional data analysis literature is the partial observations framework, see \cite{yao2005functional, zhang2016sparse} and \cite[Section 4]{fang2023adaptive}. In this setting, observations are comprised of noisy evaluations of the infinite dimensional response function at a set of grid-points located randomly in the domain $D$. At a high level, much of this literature involves the study of nonparametric estimators (e.g. local polynomials) to recover estimates of the functions underlying the partial observations. 
These estimates are then used as inputs to the sample covariance estimator. Under general smoothness assumptions, which are necessary to control the bias of the nonparametric estimator, it can be shown that covariance estimators that use these estimated functions are asymptotically equivalent to  covariance estimators that use fully observed functional data. We discuss this approach further in Remark~\ref{rem:PartialObs}.

\subsection{Outline}
Section~\ref{sec:mainresults} contains the main results of this paper: Theorem \ref{thm:ThresholdOpNormBoundwSampleRho} shows an operator norm bound for adaptive threshold estimators, Theorem \ref{thm:UniversalThreshLowerBound} states a lower bound for universal thresholding, and Theorem
\ref{thm:CompareAdaptiveAndSampleCov} compares the sample covariance and adaptive threshold estimators. In addition, Section~\ref{sec:mainresults} also includes numerical simulations in physical dimension $d=1;$ similar results in dimension $d=2$ are deferred to an appendix.
The proof of Theorem \ref{thm:ThresholdOpNormBoundwSampleRho} can be found in  Section~\ref{sec:ErrorAnalysisAdaptiveEstimator}, and uses a recent result on empirical process theory discussed in Section~\ref{sec:ProdEmpiricalProcess}. Sections \ref{sec:LowerBound} and \ref{sec:NonstationaryWeightedCovAnalysis} contain the proofs of Theorems
\ref{thm:UniversalThreshLowerBound} and \ref{thm:CompareAdaptiveAndSampleCov}, respectively. The paper closes in Section~\ref{sec:Conclusions} with concluding remarks and suggestions for future work.

\paragraph{Notation}\label{sec:Notation}
Given two positive sequences $\{a_n\}$ and $\{b_n\}$, the relation $a_n \lesssim b_n$ denotes that $a_n \le c b_n$ for some constant $c>0$. If both $a_n \lesssim b_n$ and $b_n \lesssim a_n$ hold simultaneously, we write $a_n \asymp b_n.$ For an operator $\mcL$, we denote its operator norm by $\| \mcL \|$ and its trace by $\ttrace(\mcL).$ For a matrix $\Sigma \in \R^{p \times p}$ (resp. operator $C: L_2(D) \to L_2(D)$) we write $\Sigma \succ 0$ (resp. $C \succ 0$) to denote that $\Sigma$ (resp. $C$) is positive definite. 

\section{Main Results} \label{sec:mainresults}
This section introduces our framework, assumptions, and main results. In Section~\ref{ssec:SettingAndEstimators}, we discuss the data generating mechanism under consideration, and we define the various estimators that will be studied. In Section~\ref{ssec:AdaptiveThresholdBound}, we establish our main result, a high probability operator norm error bound for the adaptive threshold estimator. In Section~\ref{ssec:ComparingToOtherEstimators}, we provide both theoretical and empirical comparisons of adaptive threshold, universal threshold, and sample covariance estimators.

\subsection{Setting and Estimators}\label{ssec:SettingAndEstimators}
Let $D \subset \R^d$ and let $u_1,\dots, u_N$ be i.i.d. copies of a centered square-integrable process $u: D \to \R.$  
We are interested in estimating the covariance operator $C$ of $u$ from the data $\{u_n\}_{n=1}^N.$
Recall that the covariance function (kernel) $k:D\times D \to \R$ and covariance
operator $C: L_2(D) \to L_2(D)$ are defined by the requirement that, for any $x,y \in D$ and $\psi \in L_2(D)$, 
\begin{align*}
    k(x,y) := \E [u(x)u(y)], 
    \qquad
    (C\psi)(\cdot) := \int_D k(\cdot, y)\psi(y) \, dy.
\end{align*}
That is, $C$ is the integral operator with kernel $k.$ 
We will focus on (sub-)Gaussian data. Recall that a square-integrable process $u$ is called (sub-)Gaussian if, for any fixed $w \in L_2(D),$ the random variable $\langle u,w \rangle_{L_2(D)}$ is (sub-)Gaussian. We further recall that the process $u$ is called pre-Gaussian if there exists a centered Gaussian process, $v$, with the same covariance operator as that of $u.$ Following \cite[page 261]{ledoux2013probability}, we refer to $v$ as the Gaussian process \textit{associated to} $u$.
 
For simplicity, we take $D: = [0,1]^d$ to be the $d$-dimensional unit hypercube. In the applications that motivate this work, the ambient dimension $d$ is typically $1,2,$ or $3$, and so we treat $d$ as a constant throughout. We are interested in applications where the covariance function $k$ exhibits approximate sparsity (which may arise, for instance, due to spatial decay of correlations), and where the marginal variance function $\sigma^2(x):= k(x,x)$ has multiple scales in the domain $D.$ 
In this setting, the sample covariance function $\hatk$ and sample covariance operator $\hatC$  defined by 
\begin{align*}
    \hatk(x,y) := \frac{1}{N} \sum_{n=1}^N u_n(x) u_n(y),
    \qquad 
    (\hatC\psi)(\cdot) := \int_D \hatk(\cdot, y)\psi(y) \, dy
\end{align*}
perform poorly. To improve performance in regard to exploiting  approximate sparsity,
one can instead consider the universal threshold estimator defined by 
\begin{align*}
    \hatk_{\rho_N}^{\mathsf{U}}(x,y) := \hatk(x,y) \indicator \inbraces{|\hatk(x,y)| \ge \rho_N},
    \qquad 
    (\hatC^{\mathsf{U}}_{\rho_N}\psi)(\cdot) := \int_D \hatk^{\mathsf{U}}_{\rho_N}(\cdot, y)\psi(y)\, dy,
\end{align*}
where $\rho_N$ is a tunable thresholding parameter. However, this approach is not well suited if the marginal variance function takes a wide range of values on $D,$ where it becomes essential to consider a spatially varying thresholding parameter. To that end, we define the variance component $\theta: D \times D \to \R_{\ge 0}$ by  
\begin{align*}
    \theta(x,y) := \tvar \bigl(u(x)u(y)\bigr).
\end{align*}
To estimate the variance component, we consider the standard \emph{sample-based} estimator given by
\begin{equation*}
    \hattheta_{\mathsf{S}}(x,y) := \frac{1}{N} \sum_{n=1}^N u_n^2(x) u_n^2(y) - \hatk^2(x,y).
\end{equation*}  
In the Gaussian setting, we additionally consider the \emph{Wick's-based} estimator given by $$\hattheta_{\mathsf{W}}(x,y) := \hatk(x,x)\hatk(y,y) + \hatk^2(x,y),$$ which is motivated by the following derivation invoking Wick's theorem (also commonly referred to as Isserlis' theorem)
\begin{align*}
    \theta(x,y) 
    &= \E \bigl[u^2(x)u^2(y) \bigr]-\bigl(\E\bigl[u(x)u(y)\bigr]\bigr)^2\\
    &= \E \bigl[u^2(x)\bigr] \E\bigl[ u^2(y)\bigr] + 2\bigl(\E\bigl[u(x)u(y)\bigr] \bigr)^2 -\bigl(\E \bigl[u(x)u(y) \bigr]\bigr)^2\\
    &= k(x,x)k(y,y) + k^2(x,y).
\end{align*}
Given an estimator $\hattheta$ of $\theta$, we then define the adaptive threshold estimator
\begin{align*}
    \hatk^{\mathsf{A}}_{\rho_N} := \hatk(x,y) \indicator \inbraces{ \left|\frac{\hatk(x,y)}{\sqrt{\hattheta(x,y)}} \right| \ge \rho_N},
    \qquad 
    (
    \hatC_{\rho_N}^{\mathsf{A}}
    \psi)(\cdot) 
    := \int_D \hatk^{\mathsf{A}}_{\rho_N} (\cdot, y)\psi(y)\, dy,
\end{align*}
where we set $\mathsf{A} = \mathsf{S}$ when $\hattheta=\hattheta_{\mathsf{S}}$ and $\mathsf{A} = \mathsf{W}$ when $\hattheta=\hattheta_{\mathsf{W}}.$ We refer to this as adaptive thresholding (of the sample covariance) since the event in the indicator can be equivalently written as $\{|\hatk(x,y)| \ge \rho_N \hattheta^{1/2}(x,y)\}$, and so the level of thresholding
varies with the location $(x,y) \in D \times D.$ 
The goal of this paper is to demonstrate through rigorous theory and numerical examples the improved performance of the adaptive threshold estimator relative to the universal threshold and sample covariance estimators.

\subsection{Error Bound for Adaptive-threshold Estimator} \label{ssec:AdaptiveThresholdBound}
Our theory is developed under the following assumption: 
\begin{assumption}\label{ass:mainAssumption}
Let $u, u_1,\dots, u_N $ be i.i.d. centered sub-Gaussian and pre-Gaussian random functions on $D=[0,1]^d$ that are Lebesgue almost-everywhere continuous with probability one. It holds that:
    \begin{enumerate}[label=(\roman*)]
        \item $C \in \mcK_q^*$ where 
        \begin{align*}
             \mcK_q^* 
             :=  \mcK_q^*(R_q)
             =
             \inbraces{
             C \succ 0,
             \sup_{x \in D} \int_D \bigl(k(x,x)k(y,y) \bigr)^{\frac{1-q}{2}} |k(x,y)|^q \, dy \le R_q^q 
             }.
        \end{align*}
        \item There exists  a universal constant $\nu > 0$ such that, for any $x,y\in D,$
        $$
        \theta(x,y) \ge \nu k(x,x)k(y,y).
        $$
        \item For a sufficiently small constant $c$, the sample size satisfies 
        \begin{align*}
             \sqrt{N} \ge \frac{1}{c} \E \insquare{\sup_{x \in D} \frac{u(x)}{\sqrt{k(x,x)}}}.
        \end{align*}
    \end{enumerate}
\end{assumption}
In contrast to the  setting considered in \cite{al2023covariance}, Assumption~\ref{ass:mainAssumption} allows for sub-Gaussian data and admits covariance functions for which $\sup_{x \in D} k(x,x) \to \infty$. Furthermore, here we only require \emph{Lebesgue almost-everywhere} continuity of the data, whereas \cite{al2023covariance} requires continuous data. Assumption~\ref{ass:mainAssumption}~(i) specifies that the covariance function $k$ satisfies a weighted $L_q$-sparsity condition that generalizes the class of row-sparse matrices $\mcU_q^*(R_q)$ studied in \cite{cai2011adaptive} to our infinite-dimensional setting. Assumption~\ref{ass:mainAssumption}~(ii) ensures that consistent estimation of the variance component is possible, and is analogous to requirements in finite dimension \cite[Condition C1]{cai2011adaptive}. Assumption~\ref{ass:mainAssumption}~(iii) is imposed for purely cosmetic reasons and can be removed at the expense of a more cumbersome statement of the results and proofs that would need to account for the case in which the sample size is chosen to be insufficiently large. The sample size requirement can also be compared to \cite[Condition 4]{fang2023adaptive}, which requires that the pair $(N,p)$ satisfies $\log p /N^{1/4} \to 0$ as $N,p \to \infty$, where $p$ is the number of component random functions as described in Section~\ref{sec:LitReviewInfiniteDims}. In contrast, our assumption is nonasymptotic and stated only in terms of the dimension-free quantity  $\E [\sup_{x \in D} u(x)/\sqrt{k(x,x)}]$ as our proof techniques differ from theirs. 

\begin{remark}[Comparison to Global-type Sparse Class of \cite{fang2023adaptive}]\label{rem:globalSparsityAssumption}
    As noted in Section~\ref{sec:LitReviewInfiniteDims}, \cite{fang2023adaptive} recently studied a notion of sparse covariance functions $\mcG_q(R_q, \varepsilon)$  different to the one considered here. In particular, the class $\mcG_q(R_q, \varepsilon)$ specifies a \textit{global} notion of sparsity, in the sense that it imposes conditions on the relationship between the various coordinate covariance functions that make-up $\textbf{K}$. In contrast, the class $\mcK^*_q$ is \textit{local}, in that it imposes a sparse structure on the covariance function of a real-valued process. For example, in the single component case $p=1$, covariance functions $k_{11}$ belonging to $\mcG_q(R_q, \varepsilon)$ must satisfy $\|k_{11}\|_{\infty}^{1-q} \|k_{11}\|^q_{\text{HS}} \le R_q^q$ and $\| k_{11}^{-1}\|_{\infty} \| k_{11}\|_{\infty} \le 1/\varepsilon$. This is equivalent to requiring that the covariance function $k_{11}$ is bounded in Hilbert-Schmidt norm, and that $k_{11}$ and its inverse are bounded in supremum norm. Importantly, in contrast to the class $\mcK^*_q$ studied here, their assumption does not capture any decay of correlations of the process at two different points in the domain, nor does it permit $\sup_{x \in D} k_{11}(x,x) \to \infty.$
\end{remark}

We are now ready to state our main result, which establishes operator norm bounds for adaptive threshold estimators.
\begin{theorem}\label{thm:ThresholdOpNormBoundwSampleRho}
    Under Assumption~\ref{ass:mainAssumption}, let $v$ be the Gaussian process associated to $u$, and for a universal constant $c_0>0$, let 
     \begin{align*}
        \rho_N =  
        \frac{c_0}{{\sqrt{ N}}} \E\insquare{\sup_{x\in D} \frac{v(x)}{k^{1/2}(x,x)}}.
    \end{align*}
    Then, there exists universal positive constants $c_1, c_2$ such that, with probability at least $1-c_1e^{-N\rho_N^2}$,
    \begin{align*}
        \normn{\hatC_{\rho_N}^\mathsf{S} - C} \le c_2 R_q^q \rho_N^{1-q}.
    \end{align*}
    Moreover, if $u$ is a Gaussian process, then for 
    \begin{align*}
        \hat{\rho}_N =
        \frac{c_0}{{\sqrt{ N}}}\inparen{
        \frac{1}{N} \sum_{n=1}^N \sup_{x\in D} \frac{u_n(x)}{\hatk^{1/2}(x,x)}
        },
    \end{align*}
    there exists a universal positive constants $c_1, c_2$ such that, with probability at least $1-c_1e^{-N\rho_N^2}$,
    \begin{align*}
        \normn{\hatC_{\hat{\rho}_N}^{\mathsf{A}} - C} \le c_2 R_q^q \rho_N^{1-q},
    \end{align*}
    where $\mathsf{A} \in \{ \mathsf{S}, \mathsf{W} \}.$
\end{theorem}

\begin{remark}
    For adaptive covariance matrix estimation under weighted $\ell_q$-sparsity described in Section~\ref{sec:LitReviewFiniteDims}, \cite[Theorem 1]{cai2011adaptive} showed that if the (normalized) sample covariance matrix satisfies $\max_{1 \le i,j \le d_X} |\hatSigma_{ij}- \Sigma_{ij}|/\hatV_{ij}^{1/2} \lesssim \tilde{\rho}_N$, then the operator norm error of the adaptive thresholding covariance matrix estimator can be bounded by $\tilde{R}_q^q \tilde{\rho}_N^{1-q}$, where $\tilde{R}_q^q$ controls the row-wise weighted $\ell_q$-sparsity of $\Sigma$. The choice of thresholding parameter can be understood by appealing to the analogy that covariance matrix estimation may be interpreted as a heteroscedastic Gaussian sequence model (see \cite[Section 2]{cai2011adaptive}), so that roughly speaking, for large $N,$
    \begin{align*}
        \frac{1}{N} \sum_{n=1}^N X_{ni}X_{nj} \approx \Sigma_{ij} + \sqrt{\frac{V_{ij}}{N}}Z_{ij}, \qquad 1 \le i,j \le d_X,
    \end{align*}
    with $\{Z_{ij}\}$ i.i.d. standard normal. This explains the choice of the thresholding parameter in the finite-dimensional setting (after normalizing the data by $\hatV_{ij}$), since
    \begin{align*}
        \tilde\rho_N \asymp \sqrt{\frac{\log d_X}{N}} 
        \asymp 
        \frac{\E \insquare{\max_{i,j \le d_X} Z_{ij}}}{\sqrt{N}},
    \end{align*}
     provides element-wise control on the sample covariance. In the infinite-dimensional setting considered here, we require instead high probability sup-norm concentration bounds for the sample covariance function $\hatk(x,y)$ and the estimated
     variance component function $\hattheta(x,y)$. These bounds are obtained in Section~\ref{sec:ErrorAnalysisAdaptiveEstimator} utilizing tools adapted from recent advances in the study of multi-product empirical processes \cite{al2025sharp} via generic chaining. These techniques are discussed in Section~\ref{sec:empiricalnew}. Our results show that the correct thresholding radius $\rho_N$ must scale with the expected supremum of the normalized associated Gaussian process, as this is precisely the dimension-free quantity needed to control the sample quantities uniformly over the domain $D$. Since $\rho_N$ is a population level quantity, we establish in Lemma~\ref{lem:SampleRhoNConcentration} that it may be replaced by its empirical counterpart $\hat{\rho}_N$ in the Gaussian setting, yielding a computable estimator $\hatC_{\hat{\rho}_N}^{\mathsf{A}}.$ This is possible since the associated Gaussian process agrees with the observed process, i.e. $u=v$ (see also Remark~\ref{rem:genericChainingControl} for a technical discussion of this point). In the sub-Gaussian setting, Theorem \ref{thm:ThresholdOpNormBoundwSampleRho} shows that an adaptive threshold estimator with an appropriate choice of thresholding radius $\rho_N$ achieves the same estimation error as in the Gaussian case. In practice, we advocate choosing the thresholding radius $\rho_N$ by cross-validation in non-Gaussian settings. 
     
     Our theory for Gaussian data holds for both the sample-based and the Wick's-based estimators of the variance component. The Wick's-based estimator might be preferred in practice as it only requires estimating the second moment of the process, whereas the sample-based estimator requires estimating both the second and fourth moments, which is more computationally intensive. From a theoretical perspective, the Wick's-based estimator is also easier to analyze using results for quadratic empirical processes (see e.g. \cite{mendelson2016upper, mendelson2010empirical}), whereas the sample estimator relies on bounds for higher order multi-product empirical processes as described in Section~\ref{sec:empiricalnew}.
     We further remark that in contrast to finite-dimensional results in which the high probability guarantee improves as $d_X$ increases, the probability in Theorem~\ref{thm:ThresholdOpNormBoundwSampleRho} approaches $1$ as the expected supremum of the normalized process grows. It is straightforward but tedious to derive high probability bounds that are more general in that they depend additionally on a confidence parameter $t \ge 1$. We provide such bounds for the pre-requisite Lemmas~\ref{lem:SampleCovFnConcentration} and \ref{lem:SampleCovFnSquaredConcentration}. In Section~\ref{ssec:multiscale}, we study an explicit family of processes for which the expected supremum can be expressed in terms of parameters of the covariance kernel. Finally, we note that the pre-factor $c_0$ is unspecified. In the existing literature (e.g. \cite{cai2011adaptive, bickel2008regularized, al2023covariance}) it is common to choose $c_0$ manually or in a data-driven way, such as via cross-validation. For our purposes, we fix $c_0=5$ in our simulated experiments in Section~\ref{ssec:SimulationResults}.
\end{remark}

\begin{remark}[Comparison to the Partially Observed Framework] \label{rem:PartialObs}
    In this work, we assume access to the (infinite dimensional) Gaussian random functions $u_1,\dots, u_N$. While in practice we cannot work with such infinite-dimensional functional data, the theory is nonetheless illuminating for finite dimensional discretizations, as demonstrated by our 
     empirical study 
    in Section~\ref{ssec:SimulationResults}. An alternative approach, described in Section~\ref{sec:LitReviewInfiniteDims}, is the partial observations framework. While this approach is often considered more practical as real-world data is always discrete, we argue that the partial observations approach inadvertently masks the underlying structure of the problem, as the bounds in that literature necessarily rely on the smoothness exponents (e.g. the exponent of the Hölder condition when the underlying true functions are assumed to be Hölder smooth), and not on the expected supremum of the process, as in our theory. This dependence is an artifact of the smoothness assumption, as opposed to being a quantity that fundamentally characterizes the behavior of the underlying process. In effect, the discretization step is taken far too early in the partial observations framework to uncover the dependence on the expected supremum. We comment that our approach is more in line with the operator learning literature \cite{kovachki2024operator, de2023convergence,mueller2012linear} and  adheres to the philosophy put forward in \cite{dashti2017bayesian}, which states  ``\textit{...it is advantageous to design algorithms which, in principle, make sense in infinite dimensions; it is these methods which will perform well under refinement of finite dimensional approximations.''}
\end{remark}

\subsection{Comparison to Other Estimators}\label{ssec:ComparingToOtherEstimators} 
In this section, we extend our analysis of the adaptive covariance operator estimator by comparing to other candidate estimators, namely the universal thresholding and sample covariance estimators. In Section~\ref{ssec:UniversalThresholding}, we first demonstrate that universal thresholding is inferior to adaptive thresholding over the class $\mcK_q^*(R_q)$. Next, in Section~\ref{ssec:multiscale} we restrict attention to a class of highly nonstationary processes and show that adaptive thresholding significantly improves over the sample covariance estimator. Finally, in Section~\ref{ssec:SimulationResults} we compare all three estimators on simulated experiments.

\subsubsection{Inferiority of Universal Thresholding}\label{ssec:UniversalThresholding}
In this section, we show rigorously that universal thresholding over the class $\mcK^*_q$ can perform arbitrarily poorly relative to adaptive thresholding. The result extends \cite[Theorem 4]{cai2011adaptive}, which demonstrates in the finite-dimensional setting that universal thresholding behaves poorly over the class $\mcU^*_q.$ The result relies on a reduction of the infinite-dimensional problem to a finite-dimensional one, to which the existing aforementioned theory can be applied.

\begin{theorem}\label{thm:UniversalThreshLowerBound}
    Suppose that $R_q^q \ge 8$ and $\rho_N$ is defined as in Theorem~\ref{thm:ThresholdOpNormBoundwSampleRho}. Then, there exists a covariance operator $C_0 \in \mcK^*_q(R_q)$ such that, for sufficiently large $N,$
\begin{align*}
     \inf_{\gamma_N \ge 0} 
    \E_{\{u_n\}_{n=1}^N \iid \text{GP}(0, C_0)} 
    \normn{\hatC^{\mathsf{U}}_{\gamma_N} - C_0}
    \gtrsim
    (R_q^q)^{2-q} \rho_N^{1-q},
\end{align*}
    where $\hatC^{\mathsf{U}}_{\gamma_N}$ is the universal thresholding estimator with threshold $\gamma_N.$ 
\end{theorem}

\begin{remark}\label{rem:UniversalThresholdingGeneral}
Since $q\in (0,1)$, the lower bound for universal thresholding in Theorem~\ref{thm:UniversalThreshLowerBound} is larger than the upper bound for adaptive thresholding in Theorem~\ref{thm:ThresholdOpNormBoundwSampleRho}, and this discrepancy grows as $R_q$ increases. Therefore, Theorem~\ref{thm:UniversalThreshLowerBound} implies that over the class $\mcK^*_q$, universal thresholding estimators can perform significantly worse than their adaptive counterparts, regardless of how the universal threshold parameter $\gamma_N$ is chosen. This agrees with intuition, since if the scale of the marginal variance
of the process varies significantly over the domain $D$, universal thresholding will intuitively need to scale with the largest of these 
scales, 
and as a result the thresholding radius will be too large for all but a small portion of the domain. See also Remark~\ref{rem:UniversalThresholding}.
\end{remark}

\subsubsection{Nonstationary Weighted Covariance Models}\label{ssec:multiscale} 
In this section, we further demonstrate the utility of adaptive thresholding by applying Theorem \ref{thm:ThresholdOpNormBoundwSampleRho} to an explicit class of highly nonstationary covariance models. To that end, we restrict attention to a subset of $\mcK^*_q$ of operators with covariance functions of the form 
\begin{align}\label{eq:nonstationaryKernels}
    k(x,y) = \sigma(x)\sigma(y) \tildek(x,y),
\end{align}
where $\tilde{k}$ is chosen to be an isotropic \textit{base} covariance function and $\sigma$ will represent a \textit{marginal variance} function. It follows by standard facts on the construction of covariance functions (see e.g. \cite{genton2001classes}) that \eqref{eq:nonstationaryKernels} defines a valid covariance function. This class is particularly interesting as it can be thought of as a \textit{weighted} version of many standard isotropic covariance functions used in practice, such as the squared exponential $\tildek^{\text{SE}}$ and Matérn $\tildek^{\text{Ma}}$ classes, defined respectively in \eqref{eq:SEMaDefs}.
Further, it permits us to express the theoretical quantities of Theorem \ref{thm:ThresholdOpNormBoundwSampleRho} in terms of interpretable parameters of the covariance function, as we now describe rigorously.

 Throughout this section, we  
assume that the data $u_1, \ldots, u_N$ are Gaussian and that the base function $\tildek$ satisfies the following:
\begin{assumption}\label{ass:isotropicandlengthscale}
    $\tildek$ in \eqref{eq:nonstationaryKernels} is a covariance function satisfying:
    \begin{enumerate}[label=(\roman*)]
        \item $\tildek$ is isotropic and positive, so that for $r = \normn{x-y}$, $\tildek(x,y) = \tildek(r)>0$. Further, $\tildek(r)$ is differentiable, strictly decreasing on $[0,\infty)$, and satisfies $\lim_{r \to \infty} \tildek(r) = 0$.
        \item $\tildek= \tildek_\lambda$ depends on a correlation lengthscale parameter $\lambda > 0$ such that $\tildek_\lambda(\phi r) = \tildek_{\lambda \phi^{-1} }(r)$ for any $\phi>0$, and $\tildek_\lambda(0) = \tildek(0)$ is independent of $\lambda.$
    \end{enumerate}
\end{assumption}

The class described in Assumption~\ref{ass:isotropicandlengthscale} contains many popular examples of covariance functions, such as the squared exponential (Gaussian) and the Matérn models \cite{williams2006gaussian}. Importantly though, functions of the form \eqref{eq:nonstationaryKernels} are significantly more general since they are nonstationary (and therefore non-isotropic). 
 This nonstationarity is introduced by the marginal variance functions $\sigma$, which
 are termed as such since $k(x,x) = \sigma^2(x)$. 
The theoretical study of the small lengthscale regime was initiated in \cite{al2023covariance} in which the authors considered kernels satisfying Assumption~\ref{ass:isotropicandlengthscale}, or equivalently $\sigma(x) \equiv 1$. Our analysis here extends their results to the challenging non-isotropic and nonstationary setting. For concreteness, we focus on a specific class of marginal variance functions detailed in the following assumption, where we let $\sigma$ depend on the parameter $\lambda$ in Assumption \ref{ass:isotropicandlengthscale}.

\begin{assumption}\label{ass:WeightFunction}
    The marginal variance function in \eqref{eq:nonstationaryKernels} is taken to be either $\sigma_\lambda(x;\alpha) \equiv 1$ or
    $\sigma_\lambda(x; \alpha) = \exp \inparen{ \lambda^{-\alpha} \normn{x}^2}$ for $\alpha \in (0,1/2).$
\end{assumption}

Functions of the form \eqref{eq:nonstationaryKernels} and which additionally satisfy Assumptions~\ref{ass:isotropicandlengthscale} and \ref{ass:WeightFunction} are denoted by $k_\lambda(x,y)$. 
For small $\lambda,$ the case $\sigma_\lambda(x; \alpha) = \exp \inparen{ \lambda^{-\alpha} \normn{x}^2}$ results in a particularly challenging estimation problem due to the extreme nonstationarity induced by the wide range of the exponential function across the domain. The case $\sigma_\lambda(x;\alpha) \equiv 1$ allows us to include unweighted covariance functions, thus strictly generalizing the theory in \cite[Section 2.2]{al2023covariance}.
Our main interest in the exponential marginal variance function is further motivated by the following fact regarding the exponential dot-product covariance function, $k_\lambda(x,y) = \exp(x^\top y/\lambda^2)$, which is interesting as it may be viewed as the simplest example of a nonstationary covariance function. Note that we can write
\begin{align*}
        \exp\inparen{\frac{x^\top y}{\lambda^2}}
        &=
        \exp\inparen{\frac{\normn{x}^2}{2\lambda^2}}
        \exp\inparen{\frac{\normn{y}^2}{2\lambda^2}}
        \exp\inparen{-\frac{\normn{x-y}^2}{2 \lambda^2}}\\
        &=
        \exp\inparen{\frac{\normn{x}^2}{2\lambda^2}}
        \exp\inparen{\frac{\normn{y}^2}{2\lambda^2}}
        \tildek^{\text{SE}}_\lambda(x,y).
\end{align*}
Figure~\ref{fig:varyingAlphad1} shows random draws when $\tildek$ is chosen to be the squared exponential (SE) covariance function, with varying choices of $\lambda$ and $\alpha$. It is immediately clear that for smaller $\lambda$, the processes become more \textit{local}, whereas the role of $\alpha$ is to change the scale of the process across the domain, with this change being more pronounced for larger $\alpha.$ Analogous plots in the $d=2$ case are presented in Figure~\ref{fig:varyingAlphad2} in the appendix.

\begin{figure}
    \centering
    \includegraphics[width=0.72\textwidth]{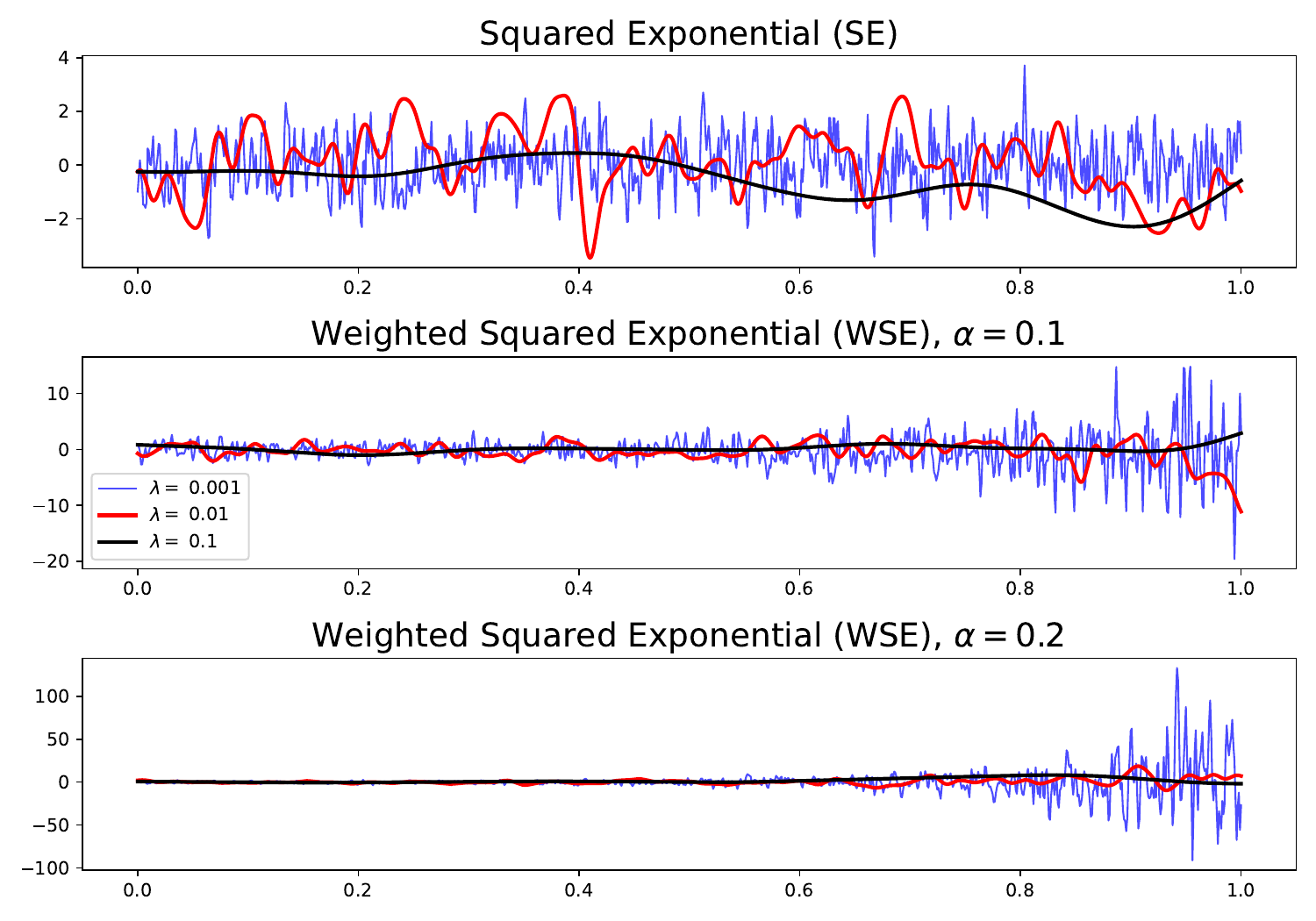}
    \caption{Draws from a centered Gaussian process on $D=[0,1]$ with weighted SE covariance function of the form \eqref{eq:nonstationaryKernels} with SE base kernel defined in \eqref{eq:SEMaDefs}. In the first plot, $\sigma =1$ (unweighted), and in the second and third plots, $\sigma$ is chosen according to Assumption~\ref{ass:WeightFunction} and with $\alpha=0.1, 0.2$ respectively. The scale parameter $\lambda$ is varied over $0.001$ (blue), $0.01$ (red) and $0.1$ (black)
    }
    \label{fig:varyingAlphad1}
\end{figure}
\begin{theorem}[Sample Covariance vs. Adaptive 
Thresholding]\label{thm:CompareAdaptiveAndSampleCov}
    Let $\hatC$ and $\hatC_{\hat{\rho}_N}^{\mathsf{A}}$ denote the sample covariance and adaptively thresholded estimator respectively. Then, there exists a universal constant $\lambda_0 > 0$ such that, for all $\lambda < \lambda_0$, it holds with probability at least $1-\lambda^d$
    that 
    \begin{align}
        \frac{\normn{\hatC- C}}{\normn{C}} 
        &\asymp
        \sqrt{\frac{\lambda^{-d}}{N}} \lor \frac{\lambda^{-d}}{N}, \label{eq:SampleCovBd}\\
    \frac{\normn{ \hatC_{\hat{\rho}_N}^{\mathsf{A}}} - C} {\|C\|}
    &\lesssim
    c(q) 
    \inparen{ \frac{ \log(\lambda^{-d})}{N} }^{(1-q)/2},\label{eq:AdaptiveCovBd}
\end{align}
where $\mathsf{A} \in \{ \mathsf{S}, \mathsf{W} \}$ and $c(q)$ is a constant depending only on $q$.
\end{theorem}

\begin{remark}
    In the proof of Theorem~\ref{thm:CompareAdaptiveAndSampleCov}, we provide an explicit expression for the constant $c(q)$. We remark that when  
     $\tildek :=\tildek^{\text{SE}},$ 
    straightforward calculations yield that $c(q) \asymp q^{-3d/2}$. We note once more that throughout this work, the dimension $d$ of the physical domain $D=[0,1]^d$ is treated as a constant. 
\end{remark}

Theorem~\ref{thm:CompareAdaptiveAndSampleCov} --- motivated by \cite{al2023covariance, koltchinskii2017concentration} --- considers the \textit{relative} as opposed to the \textit{absolute} errors commonly used in the sparse estimation literature \cite{bickel2008regularized, cai2011adaptive, fang2023adaptive}. The bound demonstrates that when $\lambda$ is sufficiently small, the adaptive thresholding estimator exhibits an exponential improvement in sample complexity over the sample covariance estimator. We remark that the bound is identical to \cite[Theorem 2.8]{al2023covariance} which considers the less general class of unweighted covariance functions, i.e. with $\sigma_\lambda := 1$ in \eqref{eq:nonstationaryKernels}. The sample covariance bound \eqref{eq:SampleCovBd} follows by an application of \cite[Theorem 9]{koltchinskii2017concentration}, which shows that with high probability $\normn{\hatC - C} \lesssim \|C\| (\sqrt{r(C)/N} \lor r(C)/N)$, where $r(C) = \ttrace(C)/\|C\|$ is the effective (intrinsic) dimension of $C$. To apply this result, it is therefore necessary to derive sharp characterizations for both $\ttrace(C)$ and $\|C\|$ in terms of the covariance function parameters $\alpha, \lambda$, which we provide in Lemmas~\ref{lem:TraceBound} and \ref{lem:CovOpBound}. The adaptive covariance bound \eqref{eq:AdaptiveCovBd} follows by an application of our main result, Theorem~\ref{thm:ThresholdOpNormBoundwSampleRho}, and therefore requires a sharp characterization of $R_q$ and $\rho_N$ in terms of the same parameters, provided in Lemma~\ref{lem:SparsityParameterCharacterization}.

\begin{remark}[Universal Thresholding] \label{rem:UniversalThresholding}
    Theorem~\ref{thm:CompareAdaptiveAndSampleCov} shows that for sufficiently small $\lambda$, adaptive thresholding with an appropriately chosen thresholding parameter will significantly outperform the sample covariance estimator. It is also instructive to consider the universal thresholding estimator in which the same threshold radius $\rho_N^{\mathsf{U}}$ is used at all points $(x,y)$ when estimating $k(x,y)$. This estimator was studied in \cite{al2023covariance} under the assumption that the covariance operator belonged to the class $\mcK_q$, with $M:= \sup_{x\in D}k(x,x)=1$. Removing the bounded marginal variance assumption, a careful analysis of their theory suggests that the universal threshold radius should be chosen to scale with $\sup_{x \in D} k(x,x)$. This is analogous to the finite $d_X$-dimensional covariance matrix estimation theory (see e.g. \cite{bickel2008regularized, cai2011adaptive}) in which the (universal) thresholding parameter must be chosen to scale with $\max_{i \le d_X} |\Sigma_{ii}|$. For processes with marginal variances that dramatically differ across the domain however, and as noted in Remark~\ref{rem:UniversalThresholdingGeneral}, such a scaling causes the estimator to fail as it will necessarily set the estimator to zero for a large proportion of the domain. Specifically in the setting of Assumption~\ref{ass:WeightFunction}, we have that $\sup_{x\in D} k_\lambda(x,x) = e^{2d/\lambda^\alpha}$ and $\inf_{x\in D} k_\lambda(x,x) = 1$. Therefore, as $\lambda$ decreases, the ratio of largest to smallest marginal variances of the process diverges, and the universal thresholding estimator is zero for larger portions of the domain. This behavior is also borne out in our simulation results (see Figure~\ref{fig:d1MultipleUniversalResults}) in which a grid of universal threshold parameters is considered and all fail dramatically relative to the adaptive estimator.
\end{remark}

\subsubsection{Simulation Results}\label{ssec:SimulationResults}
In this section, we study the behavior of the sample covariance, universal thresholding, and adaptive thresholding estimators. The results provide numerical evidence for our Theorem~\ref{thm:CompareAdaptiveAndSampleCov}, and also for the discussion around universal estimators in Remark~\ref{rem:UniversalThresholding}. Our experiments are carried out in physical dimension $d=1$ (we also provide results for the case $d=2$ in \ref{app:Dim2Simulations}). Although our theory works for any base kernel $\tildek$ satisfying Assumption~\ref{ass:isotropicandlengthscale}, we focus here on the squared exponential (SE) and Matérn (Ma) classes for simplicity, defined respectively by 
\begin{align}\label{eq:SEMaDefs}
\begin{split}
    \tildek_\lambda^{\text{SE}} (x,y) 
    &= \exp \inparen{-\frac{\|x-y\|^2}{2 \lambda ^2} } \, ,\\
    \tildek_\lambda^{\text{Ma}} (x,y) 
    &= 
    \frac{2^{1-\nu}}{\Gamma(\nu)}
    \inparen{ \frac{\sqrt{2\nu}}{\lambda} \|x-y\|}^\nu
    K_\nu \inparen{\frac{\sqrt{2\nu}}{\lambda} \|x-y\|}\, ,
    \qquad \nu > \frac{d-1}{2} \lor \frac{1}{2},
\end{split}
\end{align}
where $\Gamma$ denotes the Gamma function and $K_\nu$ is the modified Bessel function of the second kind. Both covariance functions can be shown to satisfy the assumptions in this work, see \cite[Section 2.2]{al2023covariance}. Our samples are generated by discretizing the domain $D=[0,1]$ with a uniform mesh of $L=1000$ points. We consider a total of 30 choices of $\lambda$ arranged uniformly in log-space and ranging from $10^{-2.5}$ to $10^{-0.1}$. For each $\lambda$, with corresponding covariance operator $C$, the discretized operators are given by the $L \times L$ covariance matrix $C^{ij} = \bigl(k(x_i, x_j)\bigr)_{1 \le i,j \le L}$. We sample $N=5 \log(\lambda^{-d})$ realizations of a Gaussian process on the mesh, denoted $u_1,\dots, u_N \sim N(0,C).$ We then compute the empirical and (adaptively) thresholded sample covariance matrices
\begin{align*}
    \hatC^{ij} = \frac{1}{N} \sum_{n=1}^N u_n(x_i)u_n(x_j),
    \qquad 
    \hatC^{A, ij}_{\hat{\rho}_N} = \hatC^{ij}  \indicator \bigl\{ |\hatC^{ij}| \ge \hat{\rho}_N (\hattheta^{ij})^{1/2} \bigr\}, 
    \qquad 1 \le i,j \le L,
\end{align*}
where $\hat{\rho}_N$ is defined as in Theorem~\ref{thm:ThresholdOpNormBoundwSampleRho},
and $\hattheta_{ij}$ are the estimated variance components defined by either
\begin{align*}
    (\hattheta_{\mathsf{S}})^{ij} = \frac{1}{N} \sum_{n=1}^N \inparen{ u_n(x_i) u_n(x_j) -\hatC^{ij}}^2, 
    \qquad 1 \le i,j \le L,
\end{align*}
or
\begin{align*}
    (\hattheta_{\mathsf{W}})^{ij} = \hatC^{ii}\hatC^{jj} + (\hatC^{ij})^2, \qquad 1 \le i,j \le L.
\end{align*}

\begin{figure}
    \centering
\includegraphics[scale=0.5]{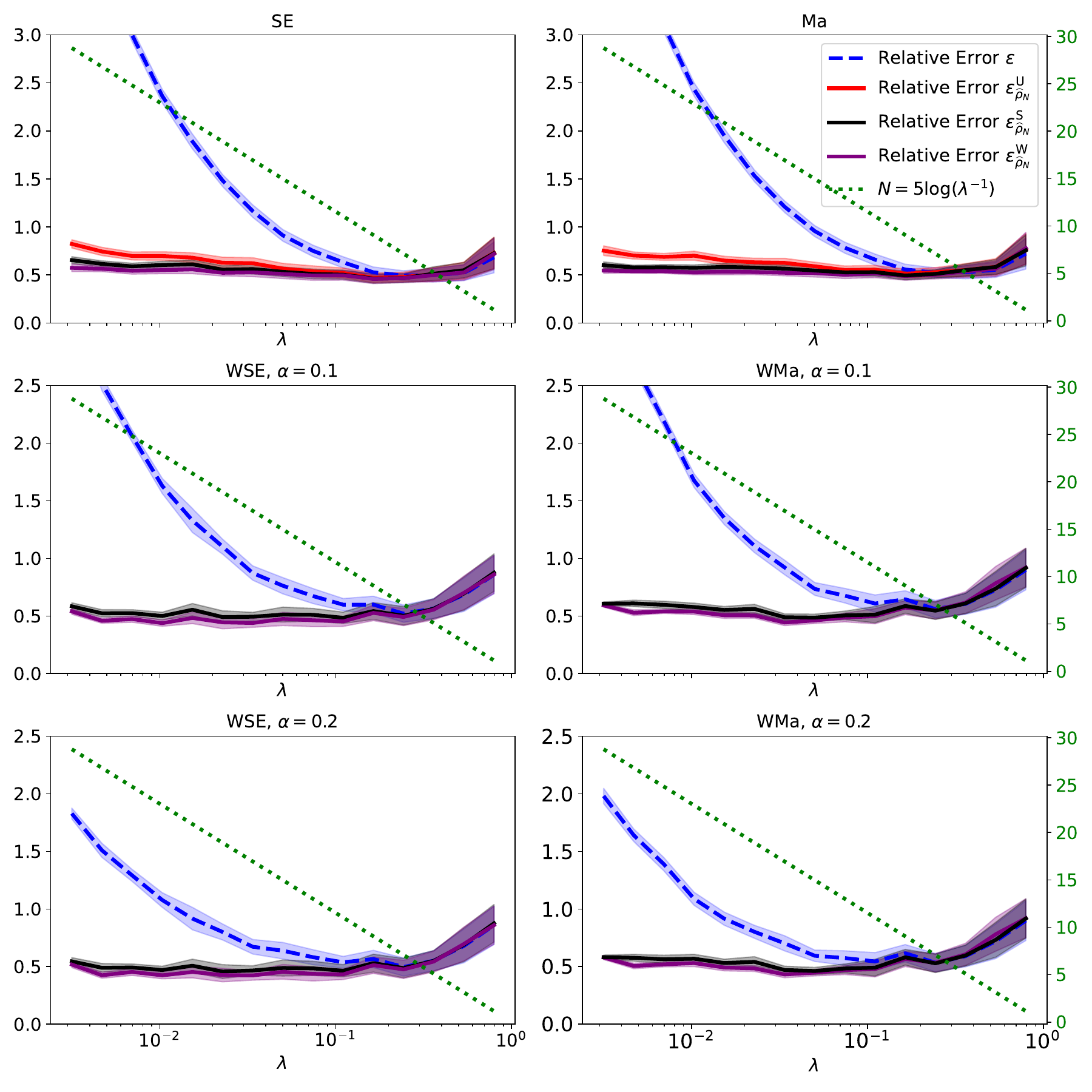}
    \caption{Plots of the average relative errors and 95\% confidence intervals achieved by the sample ($\varepsilon$, dashed blue), universal thresholding ($\varepsilon^{\mathsf{U}}_{\hat{\rho}_N}$, red), sample-based adaptive thresholding ($\varepsilon^{\mathsf{S}}_{\hat{\rho}_N}$, black) and Wick's adaptive thresholding ($\varepsilon^{\mathsf{W}}_{\hat{\rho}_N}$, purple) covariance estimators based on a sample size ($N$, dotted green) for the (weighted) squared exponential (left) and (weighted) Matérn (right) covariance functions in $d=1$ over 30 Monte-Carlo trials and 30 scale parameters $\lambda$ ranging from $10^{-2.5}$ to $10^{-0.1}$. The first row corresponds to the unweighted covariance functions and is the only case in which the universal thresholding estimator is considered; the second and third rows correspond to the weighted variants with $\alpha=0.1, 0.2$ respectively.}
    \label{fig:d1Results}
\end{figure}
To quantify performance, we consider the relative error of each of the estimators, i.e. $\varepsilon = \normn{\hatC - C}/\|C\|$ for the sample covariance, with analogous definitions for the other estimators considered. We repeat the experiment a total of 100 times for each lengthscale, and provide plots of the average relative errors as well as a 95\% confidence intervals over the trials. In Figure~\ref{fig:d1Results}, we consider in the first row the (unweighted) squared exponential and Matérn functions (with $\sigma_\lambda := 1$), in the second and third rows we choose $\sigma_\lambda$ according to Assumption~\ref{ass:WeightFunction} with $\alpha = 0.1$ and $\alpha=0.2$ respectively. In the unweighted case, we also consider the universally thresholded estimator where the threshold is taken to be $\hat{\rho}_N$, which is the correct choice by \cite[Theorem 2.2]{al2023covariance}. Note that in this case,  since $k(x,x)=1$ for all $x\in D$, all marginal variances are of the same scale, and so the universal and adaptive estimators have the same rate of convergence. While all thresholding estimators exhibit good performance as their relative errors are below $1$, it is clear that adaptive thresholding out-performs universal thresholding for the choice of pre-factor $5.$
Note that  all thresholding estimators significantly outperform the sample covariance estimator for small $\lambda$. For the second and third rows, the adaptive estimator continues to significantly outperform the sample covariance and is unaffected by the differences in scale introduced by $\sigma_\lambda$.  We observe that the sample-based and Wick's-based adaptive estimators perform similarly, with the Wick's-based estimator exhibiting slightly better performance in all experiments. The results clearly demonstrate our Theorem~\ref{thm:CompareAdaptiveAndSampleCov}, since taking only $N=5 \log(\lambda^{-1})$ samples, the relative error of the adaptive estimator remains constant as $\lambda$ decreases. 

\begin{figure}
    \centering
\includegraphics[scale=0.5]{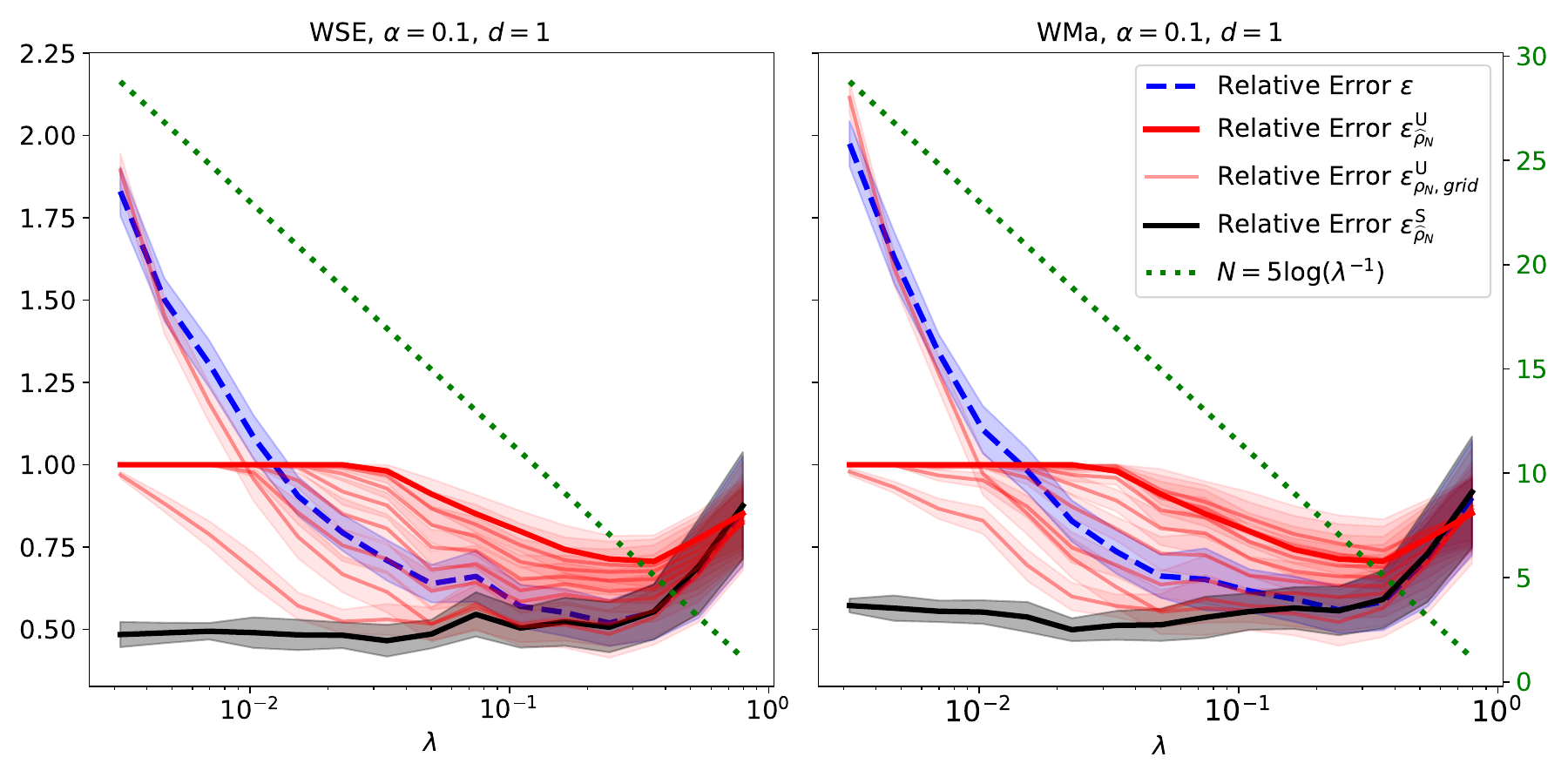}
    \caption{
    Plots of the average relative errors and 95\% confidence intervals achieved by the sample ($\varepsilon$, dashed blue), universal thresholding ($\varepsilon^{\mathsf{U}}_{\hat{\rho}_N}$, red), universal thresholding with data-driven radius ($\varepsilon^{\mathsf{U}}_{\rho_{N, \text{grid}}}$, pink) and sample-based adaptive thresholding ($\varepsilon^{\mathsf{S}}_{\hat{\rho}_N}$, black) covariance estimators based on a sample size ($N$, dotted green) for the (weighted) squared exponential (left) and (weighted) Matérn (right) covariance functions with $\alpha=0.1$ in $d=1$ over 30 Monte-Carlo trials and 30 scale parameters $\lambda$ ranging from $10^{-2.5}$ to $10^{-0.1}$.
    }
    \label{fig:d1MultipleUniversalResults}
\end{figure}
Next, in Figure~\ref{fig:d1MultipleUniversalResults} we study further the behavior of universal thresholding in the weighted setting with $\alpha=0.1$. As discussed in Remark~\ref{rem:UniversalThresholding}, the existing theory suggests to take the threshold radius to scale with $\sup_{x\in D}\sqrt{k_\lambda(x,x)} = e^{d/\lambda^\alpha}$, which becomes extremely large for small $\lambda$, and 
 causes the universal threshold estimator to behave 
 effectively like the zero estimator. 
 This choice is reflected by the error $\varepsilon^{\mathsf{U}}_{\hat{\rho}_N}$, which has relative error equal to 1 for small lengthscales. To further test the universal estimator, for each lengthscale we consider a grid of 10 
  thresholding
 radii ranging from 0 (corresponding to just using the sample covariance, with relative error $\varepsilon$) to the one suggested by the theory (corresponding to the theoretically suggested universal estimator, with relative error $\varepsilon^{\mathsf{U}}_{\hat{\rho}_N}$). The performance of these 10 estimators is represented in pink. It is clear from these results that regardless of the choice of thresholding radius, the universal estimator performs significantly worse than the adaptive estimator.

 The examples considered thus far possess a form of ordered sparsity in that the decay of the covariance function depends monotonically on the physical distance between its two arguments. Although this structure arises in many applications, it is not necessary for the success of thresholding-based estimators. In Figure~\ref{fig:isotropic}, we consider the performance of all estimators when the base kernel exhibits an unordered sparsity pattern. First, we study the periodic covariance function $\tildek^{\text{period}}$ given by
 \begin{align*}
    \tildek_\lambda^{\text{period}}(x,y) = \exp \inparen{-\frac{2 \sin^2(\pi \|x-y\|/\eta)}{\lambda^2}},
\end{align*}
where $\eta>0$ is the periodicity parameter. Intuitively, the periodic covariance function is composed of $\lfloor 1/\eta \rfloor$ bumps spaced uniformly over the domain, each behaving locally like $\tildek^{\text{SE}}_\lambda$. Consequently, this kernel is not monotonically decreasing, but it becomes sparser with smaller $\lambda.$ As another example, we consider the squared-exponential kernel applied to a random permutation of the underlying discretized grid. Shuffling the data breaks the spatial ordering while maintaining the same level of sparsity. For both periodic kernel and shuffled data examples, we choose $\sigma_\lambda$ according to Assumption~\ref{ass:WeightFunction} with $\alpha = 0.1$ and consider 30 scale parameters $\lambda$ ranging from $10^{-2.2}$ to $10^{-0.1}.$ The results demonstrate that adaptive thresholding is superior to both universal thresholding and sample covariance estimators. 
\begin{figure}
    \centering
    \includegraphics[scale=0.5]{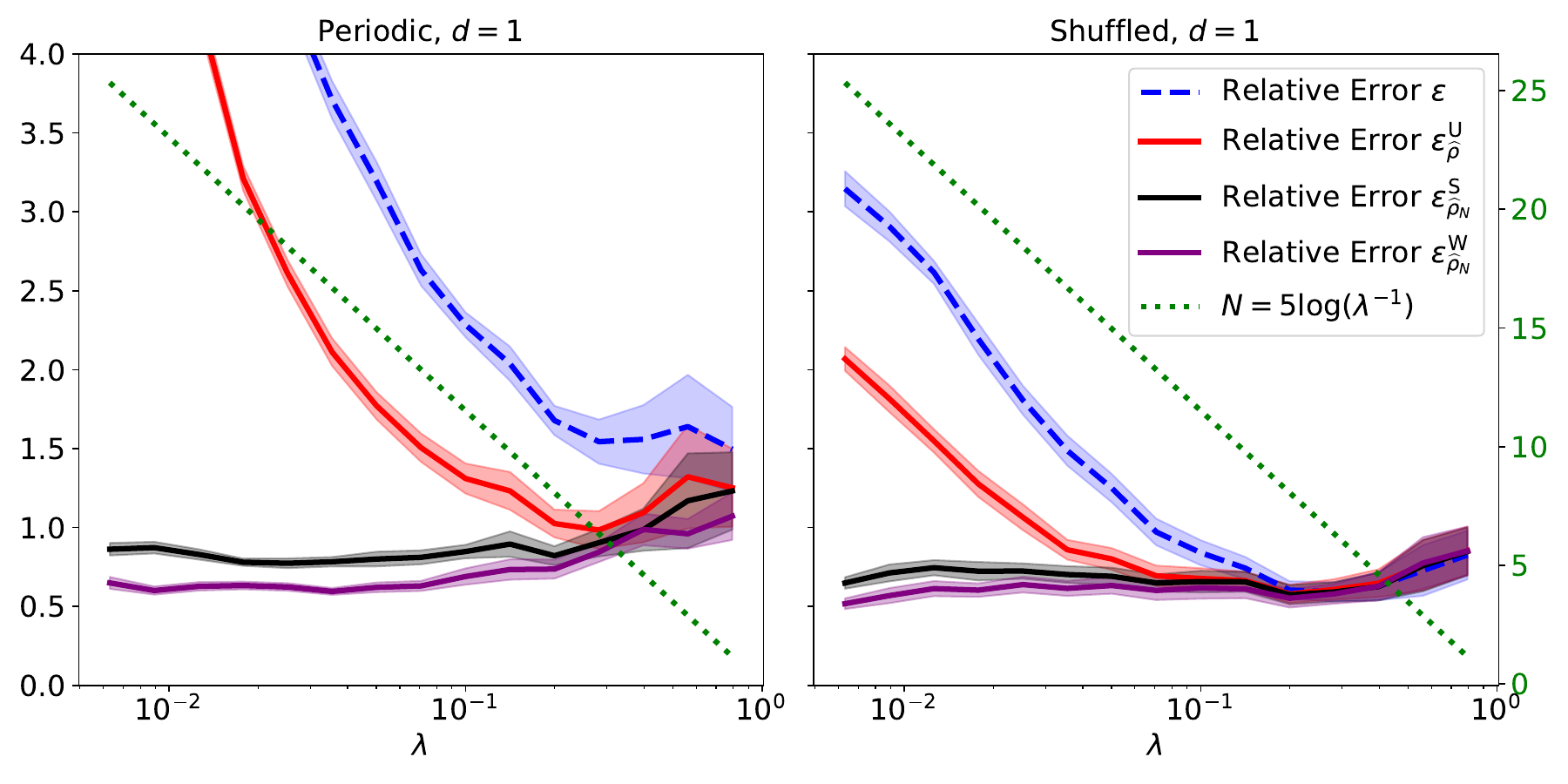}
    \caption{
    Plots of the average relative errors and 95\% confidence intervals achieved by the sample ($\varepsilon$, dashed blue), universal thresholding ($\varepsilon^{\mathsf{U}}_{\hat{\rho}_N}$, red), sample-based adaptive thresholding ($\varepsilon^{\mathsf{S}}_{\hat{\rho}_N}$, black) and Wick's adaptive thresholding ($\varepsilon^{\mathsf{W}}_{\hat{\rho}_N}$, purple) covariance estimators based on a sample size ($N$, dotted green) for the periodic kernel (left) and shuffled kernel (right) in $d=1$ over 30 Monte-Carlo trials and 30 scale parameters $\lambda$ ranging from $10^{-2.2}$ to $10^{-0.1}$.
    }
    \label{fig:isotropic}
\end{figure}

Although Theorem~\ref{thm:CompareAdaptiveAndSampleCov} holds for Gaussian data, we investigate numerically here the behavior of all estimators on sub-Gaussian data in the small lengthscale regime. Given two independent centered Gaussian processes $v^{(1)}, v^{(2)}$, with covariance functions $k^1 = k^2$ both satisfying Assumption~\ref{ass:isotropicandlengthscale}, we define $u^{(1)} := |v^{(1)}|- \E|v^{(1)}|$ and $u^{(2)} := (|v^{(1)}|-\E|v^{(1)}|) \sin(v^{(2)}).$ These transformations ensure the resulting processes are sub-Gaussian (technical details are deferred to \ref{app:subgCalcs}). The true covariance matrices are given respectively by
\begin{align*}
    C^{ij}_1 
    &=\frac{2 \sigma_\lambda(x_i) \sigma_\lambda(x_j)}{\pi} 
    \inparen{ \sqrt{1-\tildek_\lambda^2(x_i, x_j)}
    + \tildek_\lambda(x_i,x_j) \sin^{-1} (\tildek_\lambda(x_i,x_j))
    }, \\
    C^{ij}_2 &=
    C^{ij}_1 \times 
    e^{-\frac{1}{2}(\sigma_\lambda^2(x_i)+\sigma_\lambda^2(x_j))} \sinh(\sigma_\lambda(x_i) \sigma_\lambda(x_j) \tildek_\lambda(x_i,x_j)),\quad 1 \le i,j \le L.
\end{align*}
As in Figure~\ref{fig:d1MultipleUniversalResults}, we consider a grid of 10 threshold radii for each estimator, and for each lengthscale we choose the threshold that gives the smallest average relative error to generate the series in the figure. Throughout we choose $\sigma_\lambda$ according to Assumption~\ref{ass:WeightFunction} with $\alpha=0.1,$ and $\tildek = \tildek^{\text{Ma}}$. Similar results hold in the case $\tildek = \tildek^{\text{SE}}.$ The results are presented in Figure~\ref{fig:subgaussian}, with both adaptive estimators significantly beating out the sample covariance and universal threshold estimators. The empirical results suggest that the theoretical bound in Theorem~\ref{thm:CompareAdaptiveAndSampleCov} potentially continues to hold beyond the Gaussian setting. We leave a theoretical investigation of this extension to future work.

\begin{figure}[htp]
    \centering
    \includegraphics[scale=0.5]{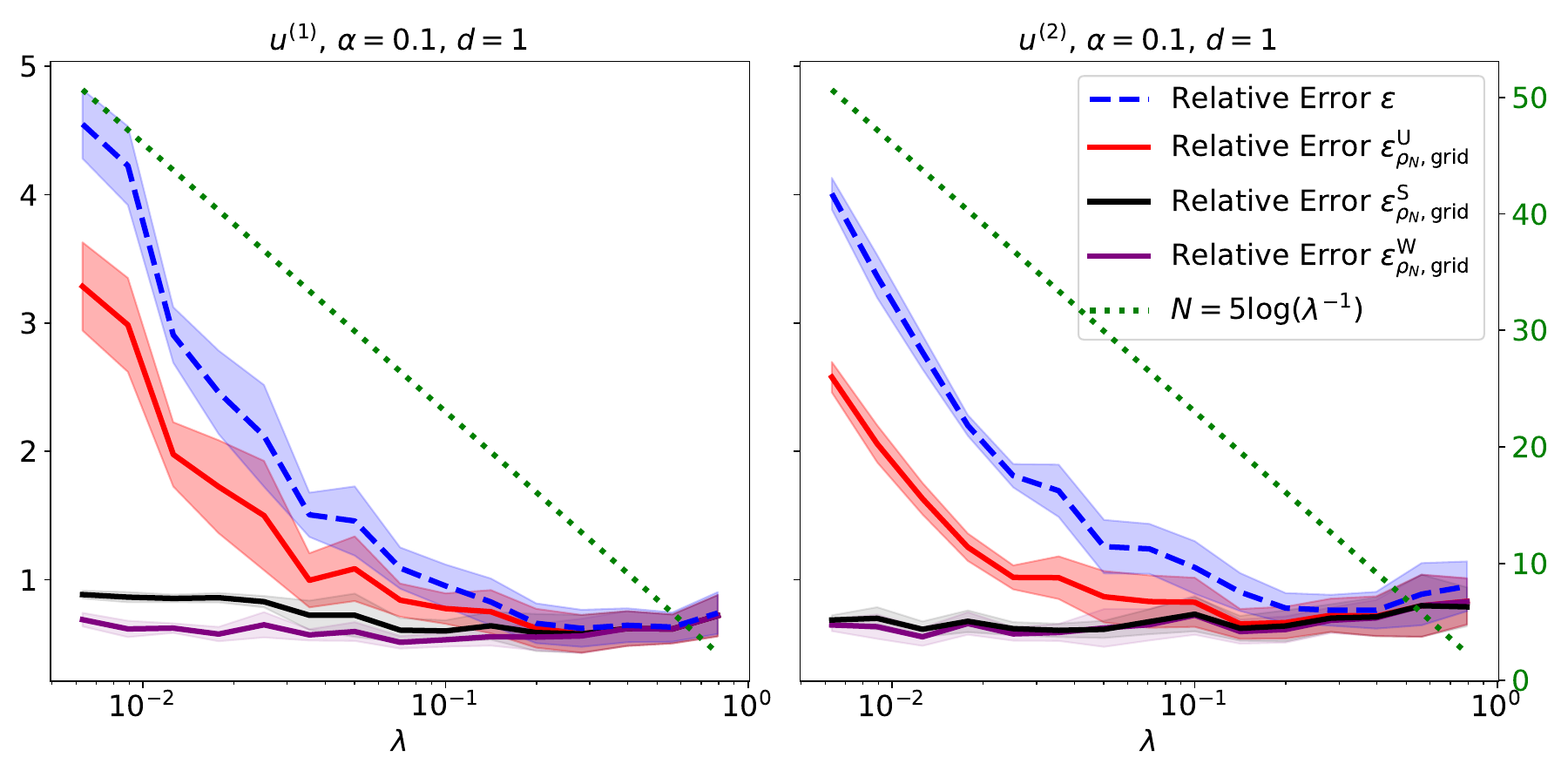}
    \caption{
    Plots of the average relative errors and 95\% confidence intervals achieved by the sample ($\varepsilon$, dashed blue), universal thresholding with data-driven radius ($\varepsilon^{\mathsf{U}}_{\rho_{N, \text{grid}}}$, pink), sample-based adaptive thresholding with data-driven radius ($\varepsilon^{\mathsf{S}}_{\rho_{N, \text{grid}}}$, black) and 
    Wick's-based adaptive thresholding with data-driven radius ($\varepsilon^{\mathsf{W}}_{\rho_{N, \text{grid}}}$, purple) covariance estimators based on a sample size ($N$, dotted green) for the sub-Gaussian processes $u^{(1)}$ (left) and $u^{(2)}$ (right). For each data-driven estimator and for each $\lambda$, $\rho_N$ is chosen as the error minimizing radius from a set of radii ranging from zero to the choice suggested by the theory in the Gaussian setting. The results are carried out in $d=1$ over 30 Monte-Carlo trials and 30 scale parameters $\lambda$ ranging from $10^{-2.2}$ to $10^{-0.1}$.
    }
    \label{fig:subgaussian}
\end{figure}

\begin{remark}\label{rem:WicksInSubGSetting}
    In all of our numerical experiments, the Wick's-based estimator exhibits strong performance at the level of and even superior to that of the sample-based estimator. While our theory suggests that the two estimators have the same rate of convergence, it does not preclude differences owing to the choice of pre-factor $c_0$ in the choice of sample size. The results therefore indicate that the Wick's-based estimator is more robust to smaller choices of this pre-factor. In the Gaussian setting, this is expected. Recall that Wick's theorem states that for a centered multivariate Gaussian vector $X:=(X_1,\dots,X_M),$ then $\E[X_1X_2\cdots X_M] = \sum_{\pi \in \Pi_M^2} \prod_{\{i,j\} \in \pi} \tcov(X_i, X_j),$ where $\Pi_M^2$ is the set of all partitions of $\{1,\dots, M\}$ of length 2. Therefore, in contrast to the sample-based estimator, the Wick's-based estimator avoids having to compute empirical higher order moments which plausibly leads to a more stable estimator for any given sample size. The situation in the sub-Gaussian setting depicted in Figure~\ref{fig:subgaussian} is somewhat more surprising, given that the Wick's-based estimator is only theoretically justified for Gaussian data. While a rigorous explanation of this phenomena is well beyond the scope of this work, we offer here some intuition as to why the Wick's-based estimator might be competitive even for sub-Gaussian data. A generalization of Wick's theorem to non-Gaussian data given in \cite{leonov1959method} states that, whenever the joint moment exists, $\E[X_1X_2\cdots X_M] = \sum_{\pi \in \Pi_M} \prod_{ a \in \pi} \kappa ((X_m)_{m \in a}),$ where $\Pi_M$ is the set of all partitions of $\{1,\dots, M\}$, and $\kappa ((X_m)_{m \in a})$ is the joint cumulant of the subset $(X_m)_{m \in a}.$ Therefore, one must estimate all higher-order cumulants as opposed to the Gaussian case in which second-order cumulants suffice. Recall that the cumulants are the coefficients in the Taylor series expansion of the cumulant (log-moment) generating function of $X$, $\psi(\gamma):=\log \E e^{\inp{\gamma, X}},$ which for sub-Gaussian $X$ is bounded above by $c \| \gamma \|^2_2$ for a positive universal constant $c.$ In order for this to be true, terms of cubic and higher-order cannot be too large.  Consequently, higher-order cumulants (third-order and above) cannot be too large. With this in mind, the Wick-based estimator can be interpreted as a type of penalized estimator that effectively treats these small higher-order cumulants as negligible by approximating them with zero.
\end{remark}

Notice that in Figures~\ref{fig:d1Results}, \ref{fig:d1MultipleUniversalResults}, \ref{fig:isotropic} and \ref{fig:subgaussian} thresholding seems to increase the relative error for large $\lambda$.  We note, however, that our theory holds only in the small $\lambda$ regime, and consequently the behavior for large $\lambda$ is not captured. We also note that the increase in error may be solely due to the very small sample size used for large values of $\lambda$.

\section{Error Analysis for Adaptive-threshold Estimator}\label{sec:ErrorAnalysisAdaptiveEstimator}
In this section, we prove our first main result, Theorem~\ref{thm:ThresholdOpNormBoundwSampleRho}. The proof structure is similar to that for the study of adaptive covariance matrix estimation in \cite{cai2012adaptive}, but our proof techniques differ in a number of important ways. Chiefly, our results are nonasymptotic and dimension free, owing to our use of recent theory on suprema of product empirical processes put forward in \cite{mendelson2016upper} and described in detail in Section~\ref{sec:ProdEmpiricalProcess}. This new approach allows us to prove Lemmas~\ref{lem:SampleCovFnConcentration} and \ref{lem:SampleCovFnSquaredConcentration}, which provide dimension-free control of the sample covariance and sample variance component. Building on these dimension-free bounds, we show five technical results, Lemmas~\ref{lem:AuxLem1}, \ref{lem:AuxLem2}, \ref{lem:AuxLem3}, \ref{lem:SampleRhoNConcentration}, and \ref{lem:AuxLem4} that are the key building blocks of the proof of the main result. Throughout, we denote the normalized versions of $u, u_1, \ldots, u_N$ by
\begin{equation}\label{eq:standardized}
    \tildeu(\cdot) := \frac{u(\cdot)}{\sqrt{k(\cdot, \cdot)}}, \qquad \tildeu_n(\cdot) := \frac{u_n(\cdot)}{\sqrt{k(\cdot, \cdot)}}, \qquad 1 \le n \le N.
\end{equation}
We further denote the Gaussian processes associated to $u, \tildeu$ by $v, \tildev$ respectively.

\begin{lemma} \label{lem:AuxLem1}
Under Assumption~\ref{ass:mainAssumption}, it holds with probability at least $1-2 e^{-(\E[\sup_{x \in D} \tildev(x)])^2}$ that
\begin{align*}
    \sup_{x,y \in D}
    \abs{
    \frac{\hattheta(x,y)-\theta(x,y)}{\theta(x,y)}
    } 
    \lesssim \frac{\E[\sup_{x \in D} \tildev(x)]}{ \nu \sqrt{N}},
\end{align*}
where $\hattheta \in \{\hattheta_{\mathsf{S}},\hattheta_{\mathsf{W}}\}$ in the Gaussian setting, and $\hattheta = \hattheta_{\mathsf{S}}$ otherwise.
\end{lemma}

\begin{proof}
    We consider first $\hattheta_{\mathsf{S}}$. Assumption~\ref{ass:mainAssumption} (ii) implies that, for any $x,y \in D,$
    \begin{align*}
        \abs{
        \frac{\hattheta_{\mathsf{S}}(x,y)-\theta(x,y)
        }{\theta(x,y)}
        }
        &=
        \abs{ 
        \frac{k^2(x,y) - \hatk^2(x,y)
        +
        \frac{1}{N} \sum_{n=1}^N 
        u_n^2(x)u_n^2(y) 
        -
        \E [u^2(x)u^2(y)]
        }{\theta(x,y)}}
        \\
        &\le 
        \frac{
        \abs {\hatk^2(x,y)- k^2(x,y)}}
        {\nu k(x,x)k(y,y)}
        +
        \frac{
        \abs{
        \frac{1}{N} \sum_{n=1}^N 
        u_n^2(x)u_n^2(y) 
        -
        \E [u^2(x)u^2(y)]
        }}{\nu k(x,x)k(y,y)}\\
        &=: I^{\mathsf{S}}_1 + I^{\mathsf{S}}_2.
    \end{align*}

    \textit{Controlling $I^{\mathsf{S}}_1$:} Note that for constants $a,b$, we have 
    \begin{align*}
        a^2 - b^2 = (a-b)(a+b) = (a-b)(a-b+2b) = (a-b)^2 + 2b(a-b).
    \end{align*}
    Therefore, 
    \begin{align*}
        I^{\mathsf{S}}_1
        =
        \frac{
        \abs {\hatk^2(x,y)- k^2(x,y)}}
        {\nu k(x,x)k(y,y)}
        &\le 
        \frac{1}{\nu}
        \abs{
        \frac{\hatk(x,y) - k(x,y)}{ \sqrt{k(x,x)k(y,y)}}
        }^2
        +2 |k(x,y)| \abs{
        \frac{\hatk(x,y) - k(x,y)}{\nu k(x,x)k(y,y)}
        }\\
        &\le 
        \frac{1}{\nu} 
        \abs{
        \frac{\hatk(x,y) - k(x,y)}{\sqrt{k(x,x)k(y,y)}}
        }^2
        +\frac{2}{\nu} \abs{
        \frac{\hatk(x,y) - k(x,y)}{\sqrt{k(x,x)k(y,y)}}
        },
    \end{align*}
    where we have used that $|k(x,y)| \le \sqrt{k(x,x)k(y,y)}$ by Cauchy-Schwarz. On the event $\Omega_{t}^{(1)}$ defined in Lemma~\ref{lem:SampleCovFnConcentration} it holds that, for all $x,y \in D,$
    \begin{align*}
        I^{\mathsf{S}}_1 
        \lesssim
        \frac{1}{\nu} 
        \inparen{\sqrt{\frac{t}{N}} \lor \frac{t^2}{N} \lor 
        \frac{\E[\sup_{x\in D} \tildev(x)]}{\sqrt{N}} \lor \frac{(\E[\sup_{x\in D} \tildev(x)])^4}{N}}.
    \end{align*}

    \textit{Controlling $I^{\mathsf{S}}_2$:} On the event $\Omega_{t}^{(2)}$ defined in Lemma~\ref{lem:SampleCovFnSquaredConcentration} it holds that, for all $x,y \in D$,
    \begin{align*}
        I^{\mathsf{S}}_2 
        \lesssim 
        \frac{1}{\nu} 
        \inparen{
        \sqrt{\frac{t}{N}} \lor \frac{t}{N} \lor \frac{\E[\sup_{x\in D} \tildev(x)]}{\sqrt{N}} \lor \frac{(\E[\sup_{x\in D} \tildev(x)])^2}{N}}.
    \end{align*}
    Therefore, by Assumption~\ref{ass:mainAssumption} (iii) and choosing $t = (\E[\sup_{x\in D} \tildev (x)])^2$, we have
    \begin{align*}
        I^{\mathsf{S}}_1 + I^{\mathsf{S}}_2 \lesssim 
        \frac{1}{\nu} 
        \inparen{
        \sqrt{\frac{t}{N}} \lor \frac{t^2}{N} \lor \frac{\E[\sup_{x\in D} \tildev(x)]}{\sqrt{N}} \lor \frac{(\E[\sup_{x\in D} \tildev(x)])^4}{N}
        }
        \lesssim
        \frac{1}{\nu} \frac{\E[\sup_{x\in D} \tildev(x)]}{\sqrt{N}}.
    \end{align*}
    The proof is completed by noting that the event $A_t := \Omega_{t}^{(1)} \cap \Omega_{t}^{(2)}$ has probability at least $1-2e^{-t}$ by Lemmas~\ref{lem:SampleCovFnConcentration} and \ref{lem:SampleCovFnSquaredConcentration}.

    Next, for $\hattheta_{\mathsf{W}},$ Assumption~\ref{ass:mainAssumption} (ii) implies that, for any $x,y \in D,$
    \begin{align*}
        \abs{
        \frac{\hattheta_{\mathsf{W}}(x,y)-\theta(x,y)
        }{\theta(x,y)}
        }
        &\le 
        \frac{|\hatk^2(x,y) - k^2(x,y)|}{\nu k(x,x)k(y,y)}
        +
        \abs{
        \frac{\hatk(x,x)\hatk(y,y) - k(x,x)k(y,y) 
        }{\theta(x,y)}} 
        =: I^{\mathsf{W}}_1 + I^{\mathsf{W}}_2.
    \end{align*}

    \textit{Controlling $I^{\mathsf{W}}_1$:} 
    Since $I^{\mathsf{W}}_1 = I^{\mathsf{S}}_1$, it follows that $I^{\mathsf{W}}_1 \lesssim \frac{\E[\sup_{x\in D} \tildev(x)]}{\nu \sqrt{N}}$ with probability at least $1 -e^{-(\E[\sup_{x \in D} \tildev(x)])^2}$.
    \textit{Controlling $I^{\mathsf{W}}_2$:} 
    Writing
    \begin{align*}
        \hatk(x,x)\hatk(y,y) - k(x,x)k(y,y)
        &= (\hatk(x,x)-k(x,x))(\hatk(y,y)-k(y,y))\\ 
        &+ (\hatk(x,x)-k(x,x))k(y,y) + (\hatk(y,y)-k(y,y))k(x,x),
    \end{align*}
    and by Assumption~\ref{ass:mainAssumption}, we have that, for any $x,y \in D,$
    \begin{align*}
        I^{\mathsf{W}}_2
        &\le         
        \abs{\frac{(\hatk(x,x) - k(x,x))(\hatk(y,y) - k(y,y)) }{\nu k(x,x)k(y,y)}} \\
        &+
        \abs{\frac{k(y,y) ( \hatk(x,x) - k(x,x)) }{\nu k(x,x)k(y,y)}} 
        + 
        \abs{\frac{k(x,x)(\hatk(y,y) - k(y,y))}{\nu k(x,x)k(y,y)}}\\
        &\le 
        \frac{1}{\nu} \sup_{x\in D} \abs{\frac{\hatk(x,x) - k(x,x) }{k(x,x)}}^2
        + \frac{2}{\nu} \sup_{x \in D}
        \abs{\frac{\hatk(x,x) - k(x,x) }{k(x,x)}} 
        =: \frac{1}{\nu} I^2_{21}+ \frac{2}{\nu} I^{\mathsf{W}}_{21}.
    \end{align*}
    Define the event $A :=\inbraces{ I^{\mathsf{W}}_{21} \lesssim \frac{\E[\sup_{x\in D} \tildev(x)]}{\sqrt{N} } },$ and note that on $A$,  
    \begin{align*}
        I^{\mathsf{W}}_2 
        & 
        \lesssim
        \frac{1}{\nu}
        \inparen{
         \frac{\E[\sup_{x\in D} \tildev(x)]}{\sqrt{N}}
        }^2 + \frac{2}{\nu} \inparen{\frac{\E[\sup_{x\in D} \tildev(x)]}{\sqrt{N}}} \lesssim
        \frac{\E[\sup_{x\in D} \tildev(x)]}{\nu \sqrt{N}}.
    \end{align*}
    By Lemma~\ref{lem:SampleCovFnConcentration} and Assumption~\ref{ass:mainAssumption} (iii), we have that $\P(A) \ge 1-e^{- (\E[\sup_{x \in D} \tildev(x)])^2}.$
\end{proof}

\begin{lemma}\label{lem:SampleCovFnConcentration}
    For any $t \ge 1$, define $\Omega^{(1)}_t$ to be the event on which 
    \begin{align*}
        \sup_{x,y\in D} 
        \abs{ \frac{\hatk(x,y) - k(x,y)}{\sqrt{k(x,x) k(y,y)}} }
        &\lesssim 
            \sqrt{\frac{t}{N}}
            \lor 
            \frac{t}{N}
            \lor 
            \frac{\E[\sup_{x \in D} \tildev(x)]}{\sqrt{N}}
            \lor 
            \frac{ (\E[\sup_{x \in D} \tildev(x)])^2}{N}.
    \end{align*}
    Then, it holds that $\P(\Omega^{(1)}_t) \ge 1-e^{-t}.$ 
\end{lemma}
\begin{proof}
    The result follows by invoking Lemma~\ref{lem:sub-GaussianStandardizedProductProcessGeneral} after noting that, for any $x,y \in D,$
    \begin{align*}
        \abs{ \frac{\hatk(x,y) - k(x,y)}{\sqrt{k(x,x) k(y,y)}} }
        &=
        \abs{ 
        \frac{1}{N} 
        \sum_{n=1}^n \frac{u_n(x)}{\sqrt{\E[u_n^2(x)]}}\frac{u_n(y)}{\sqrt{\E[u_n^2(y)]}}
        -
        \E \insquare{
        \frac{u(x)}{\sqrt{\E[u^2(x)]}}\frac{u(y)}{\sqrt{\E[u^2(y)]}}
        }
        }\\
        &= 
        \abs{
        \frac{1}{N} \sum_{n=1}^N
        \tildeu_n(x)\tildeu_n(y) - \E[\tildeu(x)\tildeu(y)] 
        }.
    \end{align*}
\end{proof}

\begin{lemma}\label{lem:SampleCovFnSquaredConcentration}
        For any $t \ge 1$, define $\Omega^{(2)}_t$ to be the event on which 
        \begin{align*}
            &\sup_{x,y \in D}
            \abs{\frac{
        \frac{1}{N} \sum_{n=1}^N u_n^2(x)u_n^2(y) 
        -
        \E [u^2(x)u^2(y)]
        }{k(x,x)k(y,y)} }\\
        &\hspace{4cm} \lesssim 
            \sqrt{\frac{t}{N}}
            \lor 
            \frac{t^2}{N}
            \lor 
            \frac{\E[\sup_{x \in D} \tildev(x)]}{\sqrt{N}}
            \lor 
            \frac{(\E[\sup_{x \in D} \tildev(x)])^4}{N}.
        \end{align*}
       Then, it holds that $\P(\Omega^{(2)}_t) \ge 1-e^{-t}.$ 
    \end{lemma}
    
    \begin{proof}
    The result follows by invoking Lemma~\ref{lem:sub-ExponentialStandardizedProductProcessGeneral} after noting that, for any $x,y \in D,$
    \begin{align*}
            \abs{\frac{
        \frac{1}{N} \sum_{n=1}^N 
        u_n^2(x)u_n^2(y) 
        -
        \E [u^2(x)u^2(y)]
        }{k(x,x)k(y,y)}}
        &=
            \abs{
        \frac{1}{N} \sum_{n=1}^N 
        \tildeu_n^2(x) \tildeu_n^2(y) 
        -
        \E [\tildeu^2(x)\tildeu^2(y)]}.
    \end{align*}
\end{proof}

\begin{lemma}\label{lem:AuxLem2} 
    Under Assumption~\ref{ass:mainAssumption}, it holds with probability at least $1-2 e^{-(\E[\sup_{x \in D} \tildev(x)])^2}$ that
    \begin{align*}
        \sup_{x,y \in D} 
        \abs{\frac{\hattheta^{1/2}(x,y) - \theta^{1/2}(x,y)}{ \hattheta^{1/2}(x,y)}} 
        \lesssim 
        \frac{\E[\sup_{x \in D} \tildev(x)]}{ \nu \sqrt{N}},
    \end{align*}
    where $\hattheta \in \{\hattheta_{\mathsf{S}},\hattheta_{\mathsf{W}}\}$ in the Gaussian setting, and $\hattheta = \hattheta_{\mathsf{S}}$ otherwise.
\end{lemma}

\begin{proof}
     Define the event 
    \begin{align*}
        A
        := \inbraces{
            \sup_{x,y \in D} \abs{
            \frac{\hattheta(x,y) - \theta(x,y)}{\theta(x,y)}
            } \le \frac{\E[\sup_{x \in D} \tildev(x)]}{ \nu \sqrt{N}}
        }.
    \end{align*}
    It holds that $\P(A) \ge 1-2e^{-(\E[\sup_{x \in D} \tildev(x)])^2}$ by Lemma~\ref{lem:AuxLem1}. 
    Further note that the universal constant in Assumption~\ref{ass:mainAssumption} (iii) can be taken sufficiently small to ensure that $\frac{\E[\sup_{x \in D} \tildev(x)]}{ \nu \sqrt{N}} \le \frac{1}{2}$. Then on $A$, for any $x,y \in D,$
    \begin{align*}
        \abs{\frac{\hattheta^{1/2}(x,y) - \theta^{1/2}(x,y)}{ \hattheta^{1/2}(x,y)}}
        &=
        \abs{\frac{\hattheta^{1/2}(x,y) - \theta^{1/2}(x,y)}{ \hattheta^{1/2}(x,y)}
        \frac{\theta^{1/2}(x,y) + \hattheta^{1/2}(x,y) }{\theta^{1/2}(x,y) + \hattheta^{1/2}(x,y) }
        } \\
        &=\abs{\frac{\hattheta(x,y) - \theta(x,y)}{ \hattheta(x,y) + \theta^{1/2}(x,y)\hattheta^{1/2}(x,y)}}\\
        &=\abs{
        \frac{\hattheta(x,y) - \theta(x,y)}{ \theta(x,y)}
        \frac{\theta(x,y)}{ \hattheta(x,y) + \theta^{1/2}(x,y)\hattheta^{1/2}(x,y)}
        }\\
        &\le \abs{
        \frac{\hattheta(x,y) - \theta(x,y)}{ \theta(x,y)}
        \frac{\theta(x,y)}{ \hattheta(x,y)}}
        \le \abs{
        \frac{\hattheta(x,y) - \theta(x,y)}{ \theta(x,y)}}
        \abs{\frac{\theta(x,y)}{ \hattheta(x,y)}}\\
        &\le 
        2\abs{\frac{\hattheta(x,y) - \theta(x,y)}{ \theta(x,y)}}
        \lesssim
        \frac{\E[\sup_{x \in D} \tildev(x)]}{ \nu \sqrt{N}},
    \end{align*}
    where the second to last inequality follows since on $A$ we have 
    \begin{align*}
    |\theta(x,y)| \le | \theta(x,y) - \hattheta(x,y)| + |\hattheta(x,y)| \le \frac{1}{2}|\theta(x,y)| + |\hattheta(x,y)| \implies |\theta(x,y)| \le 2 |\hattheta(x,y)|. 
    \end{align*}
\end{proof}

\begin{lemma}\label{lem:AuxLem3}
    Under Assumption~\ref{ass:mainAssumption}, it holds with probability at least $1-3 e^{-(\E[\sup_{x \in D} \tildev(x)])^2}$ that
    \begin{align*}
        \sup_{x,y \in D} \abs{
        \frac{\hatk(x,y) - k(x,y)}{\hattheta^{1/2}(x,y)}
        }
        \lesssim 
        \frac{\E[\sup_{x\in D} \tildev(x)]}{\nu \sqrt{ N}},
    \end{align*}
    where $\hattheta \in \{\hattheta_{\mathsf{S}},\hattheta_{\mathsf{W}}\}$ in the Gaussian setting, and $\hattheta = \hattheta_{\mathsf{S}}$ otherwise.
\end{lemma}
\begin{proof}
    Note that 
    \begin{align*}
        \abs{
        \frac{\hatk(x,y) - k(x,y)}{\hattheta^{1/2}(x,y)}
        }
        & \le 
        \abs{
        \frac{\hatk(x,y) - k(x,y)}{\theta^{1/2}(x,y)}
        }
        \abs{\frac{\theta^{1/2}(x,y)}{\hattheta^{1/2}(x,y)}}\\
        &\le 
        \abs{
        \frac{\hatk(x,y) - k(x,y)}{\theta^{1/2}(x,y)}
        }
        \inparen{
        \abs{\frac{\theta^{1/2}(x,y) - \hattheta^{1/2}(x,y)}{\hattheta^{1/2}(x,y)}}
        +1 
        }\\
        &\le 
        \abs{
        \frac{\hatk(x,y) - k(x,y)}{\sqrt{\nu k(x,x) k(y,y)}}
        }
        \inparen{
        \abs{\frac{\theta^{1/2}(x,y) - \hattheta^{1/2}(x,y)}{\hattheta^{1/2}(x,y)}}
        +1 
        }\\
        &=I_1 \times I_2.
    \end{align*}
    
    \textit{Controlling $I_1:$} 
    It holds on the event 
    $\Omega_{(\E[\sup_{x\in D} \tildev(x)])^2 }^{(1)}$ defined in Lemma~\ref{lem:SampleCovFnConcentration} that, for all $x,y \in D,$
    \begin{align*}
        I_1 
        \lesssim
        \frac{1}{\sqrt{\nu}}
        \frac{\E[\sup_{x\in D} \tildev(x)]}{\sqrt{N}},
    \end{align*}
    and $\P(\Omega_{(\E[\sup_{x\in D} \tildev(x)])^2 }^{(1)}) \ge 1-e^{-(\E[\sup_{x\in D} \tildev(x)])^2}.$

    \textit{Controlling $I_2$:} Let $B$ be the event on which the bound in Lemma~\ref{lem:AuxLem2} holds. Then, $\P(B) \ge 1-2e^{-(\E[\sup_{x\in D} \tildev(x)])^2}$, and on $B$
    \begin{align*}
        I_2 \lesssim \frac{\E[\sup_{x\in D} \tildev(x)]}{\nu\sqrt{N}} + 1.
    \end{align*}
    Then, on the event $E = \Omega_{(\E[\sup_{x\in D} \tildev(x)])^2 }^{(1)}\cap B$, we have
    \begin{align*}
        I_1 \times I_2 \lesssim 
        \frac{1}{\nu^{3/2}}
        \frac{(\E[\sup_{x\in D} \tildev(x)])^2}{N}
        \lor 
        \frac{1}{\sqrt{\nu}}
        \frac{\E[\sup_{x\in D} \tildev(x)]}{\sqrt{N}}
        =
        \frac{1}{\sqrt{\nu}}
        \frac{\E[\sup_{x\in D} \tildev(x)]}{\sqrt{N}}. 
    \end{align*}
\end{proof}

\begin{lemma}\label{lem:SampleRhoNConcentration}
    Let $v, v_1,\dots, v_N$ denote the Gaussian processes associated to $u, u_1,\dots, u_N$, which satisfy Assumption~\ref{ass:mainAssumption}. Define
    \begin{align*}
        \rho_N = 
        \frac{1}{{\nu \sqrt{ N}}} \E\insquare{\sup_{x\in D} \frac{v(x)}{k^{1/2}(x,x)}},
        \qquad 
        \hat{\rho}_N = 
        \frac{1}{{\nu \sqrt{ N}}}\inparen{
        \frac{1}{N} \sum_{n=1}^N \sup_{x\in D} \frac{v_n(x)}{\hatk^{1/2}(x,x)}
        }.
    \end{align*}
    Then, it holds with probability at least $1-4 e^{-(\E[\sup_{x \in D} \tildev(x)])^2}$ that $|\hat{\rho}_N - \rho_N| \lesssim \rho_N.$
\end{lemma}
\begin{proof}
    \begin{align*}
        &\nu \sqrt{ N} \abs{\hat{\rho}_N - \rho_N}
        = 
        \abs{
        \frac{1}{N} \sum_{n=1}^N \sup_{x\in D} \frac{v_n(x)}{\hatk^{1/2}(x,x)}
        -
        \E\insquare{\sup_{x\in D} \frac{v(x)}{k^{1/2}(x,x)}}
        }\\
        &=
        \abs{
        \frac{1}{N} \sum_{n=1}^N \sup_{x\in D} 
        \inparen{ 
        \frac{v_n(x)}{\hatk^{1/2}(x,x)}
        -
        \frac{v_n(x)}{k^{1/2}(x,x)}
        +
        \frac{v_n(x)}{k^{1/2}(x,x)}
        }
        -
        \E\insquare{\sup_{x\in D} \frac{v(x)}{k^{1/2}(x,x)}}
        }\\
        &\le 
        \abs{
        \frac{1}{N} \sum_{n=1}^N \sup_{x\in D} 
        \inparen{ 
        \frac{v_n(x)}{\hatk^{1/2}(x,x)}
        -
        \frac{v_n(x)}{k^{1/2}(x,x)}
        }
        }
        + 
        \abs{
        \frac{1}{N} \sum_{n=1}^N \sup_{x\in D} 
        \frac{v_n(x)}{k^{1/2}(x,x)}
        -
        \E\insquare{\sup_{x\in D} \frac{v(x)}{k^{1/2}(x,x)}}
        }\\
        &= I_1 + I_2.
    \end{align*}
    
    \textit{Controlling $I_1$:} We write
    \begin{align*}
        \frac{v_n(x)}{\hatk^{1/2}(x,x)}
        -
        \frac{v_n(x)}{k^{1/2}(x,x)}
        &=
        \frac{v_n(x)}{k^{1/2}(x,x)} \frac{\hatk^{1/2}(x,x)- k^{1/2}(x,x)}{\hatk^{1/2}(x,x)}.
    \end{align*}
    Define the event 
    \begin{align*}
        A
        := \inbraces{
            \sup_{x \in D} \abs{
            \frac{\hatk(x,x) - k(x,x)}{k(x,x)}
            } \le \frac{\E[\sup_{x \in D} \tildev(x)]}{ \nu \sqrt{N}}
        }.
    \end{align*}
    By Lemma~\ref{lem:SampleCovFnConcentration} and Assumption~\ref{ass:mainAssumption} (iii), we have that $\P(A) \ge 1-2e^{- (\E[\sup_{x \in D} \tildev(x)])^2}.$ Further note that the universal constant in Assumption~\ref{ass:mainAssumption}~(iii) can be taken sufficiently small to ensure that $\frac{\E[\sup_{x \in D} \tildev(x)]}{ \nu \sqrt{N}} \le \frac{1}{2}$. By a similar argument to the one used in Lemma~\ref{lem:AuxLem2}, conditional on $A$ and for any $x \in D,$
    \begin{align*}
        \abs{\frac{\hatk^{1/2}(x,x) - k^{1/2}(x,x)}{ \hatk^{1/2}(x,x)}}
        \lesssim
        \frac{\E[\sup_{x \in D} \tildev(x)]}{ \nu \sqrt{N}} \le \frac{1}{2}.
    \end{align*}
    Therefore, on $A$ it holds that
    \begin{align*}
        I_1 
        &=
        \abs{
        \frac{1}{N} \sum_{n=1}^N \sup_{x\in D} 
        \inparen{ 
        \frac{v_n(x)}{\hatk^{1/2}(x,x)}
        -
        \frac{v_n(x)}{k^{1/2}(x,x)}
        }
        }\\
        &\le 
        \frac{1}{2}
        \abs{
        \frac{1}{N} \sum_{n=1}^N \sup_{x\in D} 
        \frac{v_n(x)}{k^{1/2}(x,x)}
        }\\
        &\lesssim
        \abs{
        \frac{1}{N} \sum_{n=1}^N 
        \sup_{x\in D} \frac{v_n(x)}{k^{1/2}(x,x)}
        -
         \E \insquare{\sup_{x\in D} \frac{v(x)}{k^{1/2}(x,x)}}
        }
        +
        \E \insquare{\sup_{x\in D}\frac{v(x)}{k^{1/2}(x,x)}}\\
        &=
        I_2 + \nu\sqrt{N} \rho_N.
    \end{align*}
    
    \textit{Controlling $I_2$:} By \cite[Lemma 2.4.7]{talagrand2022upper}, $\sup_{x \in D} \tildev_n(x)$ is  $\sup_{x \in D} \tvar \bigl(\tildev_n(x)\bigr)$-sub-Gaussian. Since $\tvar \bigl(\tildev_n(x)\bigr) = 1$, it follows by sub-Gaussian concentration that with probability at least $1-2e^{-t}$, $I_2 \le \sqrt{2t/N}.$ Choosing $t=(\E[\sup_{x\in D} \tildev(x)])^2/2$, we have $I_2 \le \E[\sup_{x\in D} \tildev(x)]/\sqrt{N}$. Putting the bounds together, we have shown that
    \begin{align*}
         \abs{\hat{\rho}_N - \rho_N} \le \frac{1}{\nu \sqrt{ N}} (I_1 + I_2) \lesssim \frac{1}{\nu \sqrt{ N}} \frac{\E[\sup_{x\in D} \tildev(x)]}{\sqrt{N}} + \rho_N 
         \lesssim \rho_N.
    \end{align*}
\end{proof}

\begin{lemma}\label{lem:AuxLem4}
    Under the setting of Lemma~\ref{lem:SampleRhoNConcentration}, it holds with probability at least $1-7 e^{-(\E[\sup_{x \in D} \tildev(x)])^2}$ that
    \begin{align*}
            \sup_{x,y \in D} \abs{
            \frac{\hatk(x,y) - k(x,y)}{\hattheta^{1/2}(x,y)}
            } \le \frac{\hat{\rho}_N}{2},
    \end{align*}
    where $\hattheta \in \{\hattheta_{\mathsf{S}},\hattheta_{\mathsf{W}}\}$ in the Gaussian setting, and $\hattheta = \hattheta_{\mathsf{S}}$ otherwise.
\end{lemma}
\begin{proof}
    Let $E$ be the event on which the bound in the statement of the theorem holds. Then
    \begin{align*}
        \P(E^c) 
        &= 
        \P \inparen{\sup_{x,y \in D} \abs{
            \frac{\hatk(x,y) - k(x,y)}{\hattheta^{1/2}(x,y)}
            }  + \frac{1}{2}(\rho_N - \hat{\rho}_n) \ge \frac{1}{2}  \rho_N}\\
        &\le  
        \P \inparen{\sup_{x,y \in D} \abs{
            \frac{\hatk(x,y) - k(x,y)}{\hattheta^{1/2}(x,y)}
            } \ge \frac{1}{4}\rho_N }
            + 
            \P \inparen{\rho_N - \hat{\rho}_n \ge \frac{1}{4}\rho_N}\\
        & \le 7 e^{-(\E [\sup_{x \in D} \tildev (x)])^2},
    \end{align*}
    where the last line follows by Lemmas~\ref{lem:AuxLem3} and \ref{lem:SampleRhoNConcentration}.
\end{proof}

\begin{proof}[Proof of Theorem \ref{thm:ThresholdOpNormBoundwSampleRho}]
     We consider first the Gaussian case. Let $\hattheta \in \{\hattheta_{\mathsf{S}},\hattheta_{\mathsf{W}} \}$.
    Define the three events: 
    \begin{align*}
        E_1 &:= \inbraces{
            \sup_{x,y \in D} \abs{
            \frac{\hatk(x,y) - k(x,y)}{\hattheta^{1/2}(x,y)}
            } 
            \lesssim 
            \frac{\hat{\rho}_N}{2}
        },\\
        E_2 &:= 
        \inbraces{
            \sup_{x,y \in D} \abs{
            \frac{\hattheta(x,y) - \theta(x,y)}{\theta(x,y)}
            } 
            \lesssim
            \frac{1}{2}
        },\\
        E_3 &:= 
        \inbraces{
        \abs{\hat{\rho}_N - \rho_N} 
        \lesssim
        \rho_N
        },
    \end{align*}
    and $E= E_1 \cap E_2 \cap E_3$. The final result holds on $E$ as will be shown below, and so the proof is completed by noting that from Lemmas~\ref{lem:AuxLem3}, \ref{lem:SampleRhoNConcentration} and \ref{lem:AuxLem4}, $\P(E) \ge 1-c_1e^{- (\E[\sup_{x \in D} \tildeu(x)])^2}.$ Further note that on the event $E_2$, for any $x,y$ we have the following relation:
    \begin{align}\label{eq:E2Fact}
        \frac{1}{2}|\theta(x,y)| \le |\hattheta(x,y)|\le 2|\theta(x,y)|.
    \end{align}
    Now, defining the set
    \begin{align*}
        \Omega_x := \inbraces{ y \in D : \abs{\frac{k(x,y)}{\hattheta^{1/2}(x,y)}} \ge \frac{\hat{\rho}_N}{2} }, 
    \end{align*}
    we have 
    \begin{align*}
        &\normn{\hatC_{\hat{\rho}_N} - C} 
        \le \sup_{x\in D} \int_D |\hatk_{\hat{\rho}_N} (x,y) - k(x,y)| dy\\
        &=
        \sup_{x\in D} \int_{\Omega_x} 
        \abs{ 
        \frac{\hatk_{\hat{\rho}_N} (x,y) - k(x,y)}{\hattheta^{1/2}(x,y)}
        } |\hattheta^{1/2}(x,y)| dy
        +
        \sup_{x\in D} \int_{\Omega_x^c} 
        \abs{ 
        \frac{\hatk_{\hat{\rho}_N}(x,y) - k(x,y)}{\hattheta^{1/2}(x,y)}
        } |\hattheta^{1/2}(x,y)| dy\\
        &=
        \sup_{x\in D} \int_{\Omega_x} 
        \abs{ 
        \frac{\hatk_{\hat{\rho}_N} (x,y) - \hatk(x,y)}{\hattheta^{1/2}(x,y)}
        } |\hattheta^{1/2}(x,y)| dy
        +
        \sup_{x\in D} \int_{\Omega_x} 
        \abs{ 
        \frac{\hatk(x,y) - k(x,y)}{\hattheta^{1/2}(x,y)}
        } |\hattheta^{1/2}(x,y)| dy\\
        &+
        \sup_{x\in D} \int_{\Omega_x^c} 
        \abs{ 
        \frac{\hatk_{\hat{\rho}_N}(x,y) - k(x,y)}{\hattheta^{1/2}(x,y)}
        } |\hattheta^{1/2}(x,y)| dy =: I_1 + I_2 + I_3.
    \end{align*}

    \textit{Controlling $I_1$}: For any $x,y \in D,$
    \begin{align*}
        \abs{ 
        \frac{\hatk_{\hat{\rho}_N} (x,y) - \hatk(x,y)}{\hattheta^{1/2}(x,y)}
        } 
       & = 0 \times \indicator  
        \inbraces{
        \abs{\frac{\hatk(x,y)}{\hattheta^{1/2}(x,y)}} \ge \hat{\rho}_N
        }
        + 
        \abs{\frac{\hatk(x,y)}{\hattheta^{1/2}(x,y)}}
          \times
          \indicator  
        \inbraces{
        \abs{\frac{\hatk(x,y)}{\hattheta^{1/2}(x,y)}} < \hat{\rho}_N
        } \\
          &\le \hat{\rho}_N.
    \end{align*}
    Therefore, 
    \begin{align*}
        I_1 
        \le 
        \hat{\rho}_N \sup_{x\in D} \int_{\Omega_x} |\hattheta^{1/2}(x,y)| dy.
    \end{align*}
    By Assumption~\ref{ass:mainAssumption}, we have that 
    \begin{align*}
        R^q_q 
        &\ge 
        \sup_{x \in D} \int_D \inparen{k(x,x)k(y,y)}^{(1-q)/2} |k(x,y)|^q dy\\
        &\ge 
        \sup_{x \in D} \int_{\Omega_x} \inparen{k(x,x)k(y,y)}^{(1-q)/2} |k(x,y)|^q dy\\
        &\gtrsim
        \sup_{x \in D} \int_{\Omega_x} \inparen{k(x,x)k(y,y)}^{(1-q)/2} \hat{\rho}_N^q |\hattheta^{q/2}(x,y)| dy\\
        &\gtrsim \sup_{x \in D} \int_{\Omega_x} \inparen{k(x,x)k(y,y)}^{(1-q)/2} \hat{\rho}_N^q |\theta^{q/2}(x,y)| dy,
    \end{align*}
    where the third inequality follows by definition of $\Omega_x$, and the final inequality holds by \eqref{eq:E2Fact}. Further, we have
    \begin{align*}
     \theta(x,y) = \tvar \bigl(u(x)u(y)\bigr) \le \sqrt{\E [u^4(x)]\E [u^4(y)]}   \lesssim \E [u^2(x)] \E [u^2(y)] = k(x,x) k(y,y),
    \end{align*}
    where the first inequality follows by Cauchy-Schwarz,
    and the second inequality follows by the $L_4$-$L_2$ equivalence property of sub-Gaussian random variables. Therefore, it follows that 
    \begin{align*}
        R^q_q 
        &\gtrsim 
        \hat{\rho}_N^q 
        \sup_{x \in D} \int_{\Omega_x} |\theta^{1/2}(x,y)| dy 
        \ge 
        \hat{\rho}_N^q \frac{I_1}{\hat{\rho}_N}.
    \end{align*}
    We have therefore shown that $I_1 \lesssim R_q^q \hat{\rho}_N^{1-q}$, and by definition of $E_3$, it follows immediately that $I_1 \lesssim R_q^q \rho_N^{1-q}$.

    \textit{Controlling $I_2$}: On $E$, we have 
    \begin{align*}
        I_2 \lesssim \rho_N  
        \sup_{x\in D} \int_{\Omega_x} 
        |\hattheta^{1/2}(x,y)| dy
        \lesssim 
        R_q^q \rho_N^{1-q},
    \end{align*}
    which can be bounded with an identical argument to the one used to bound $I_1$.

    \textit{Controlling $I_3$}: On $E \cap \Omega_x^c$, we have 

    \begin{align*}
        \abs{\frac{\hatk(x,y)}{\hattheta^{1/2}(x,y)}}
        \le 
        \abs{\frac{\hatk(x,y) - k(x,y)}{\hattheta^{1/2}(x,y)}}
        + \abs{\frac{k(x,y)}{\hattheta^{1/2}(x,y)}}
        \le 
        \frac{\hat{\rho}_N}{2}+\frac{\hat{\rho}_N}{2}
        = \hat{\rho}_N.
    \end{align*}
    Therefore, $\hatk_{\hat{\rho}_N}(x,y) = \hatk(x,y) \indicator \inbraces{ \abs{\frac{\hatk(x,y)}{\hattheta^{1/2}(x,y)}} \ge \hat{\rho}_N} = 0$. Now, for any $q \in [0,1),$
    \begin{align*}
        I_3
        &\le 
        \sup_{x\in D} \int_{D} 
        \abs{\frac{k(x,y)}{\hattheta^{1/2}(x,y) }} |\hattheta^{1/2}(x,y)|  \indicator \inbraces{ \abs{
        \frac{k(x,y)}{\hattheta^{1/2}(x,y)}
        }
        \le \frac{\hat{\rho}_N}{2} } 
        dy\\
         &\le 
        \sup_{x\in D} \int_{D} 
        \abs{\frac{k(x,y)}{\hattheta^{1/2}(x,y) }} |\hattheta^{1/2}(x,y)|  \indicator \inbraces{ \abs{
        \frac{k(x,y)}{\hattheta^{1/2}(x,y)}
        }
        \le \rho_N }
        dy\\
        &\le 
        \rho_N
        \sup_{x\in D} \int_{D} 
        \inparen{\abs{\frac{k(x,y)}{\hattheta^{1/2}(x,y)}}/\rho_N}^q  |\hattheta^{1/2}(x,y)|  
        \indicator \inbraces{ \abs{
        \frac{k(x,y)}{\hattheta^{1/2}(x,y)}
        }
        \le \rho_N } 
        dy\\
        &\lesssim 
        \rho_N^{1-q}
        \sup_{x\in D} \int_{D} 
        \abs{k(x,y)}^q \hattheta(x,y)^{(1-q)/2}  dy.
    \end{align*}
    The second inequality holds since on $E_3$, $\hat{\rho}_N \lesssim 2\rho_N$. The third inequality holds since the quantity being taken to the $q$-th power is smaller than 1 and $q \in [0,1)$. Combining \eqref{eq:E2Fact} with Assumption~\ref{ass:mainAssumption} (ii) gives that $\hattheta(x,y)^{(1-q)/2} \le \theta(x,y)^{(1-q)/2} \le \bigl(k(x,x)k(y,y)\bigr)^{(1-q)/2}$
    and so $I_3 \lesssim \rho_N^{1-q} R_q^q$. 
     This completes the proof of the result in the Gaussian case. The proof in the sub-Gaussian setting follows identically except that the events $E_1,E_2$ are defined with respect to $\hattheta_{\mathsf{S}}$ only, and $\rho_N$ is used in place of $\hat{\rho}_N.$
\end{proof}

\section{Product Empirical Processes}\label{sec:ProdEmpiricalProcess}

This section contains the proofs of Lemmas~\ref{lem:sub-GaussianStandardizedProductProcessGeneral} and \ref{lem:sub-ExponentialStandardizedProductProcessGeneral}, which were used to establish Lemmas~\ref{lem:SampleCovFnConcentration} and \ref{lem:SampleCovFnSquaredConcentration}. The proofs  rely on the recent work \cite{al2025sharp}, which provides sharp bounds for suprema of multi-product empirical processes. We begin in Section \ref{sec:backgroundempiricalprocess} by introducing 
technical definitions as well as the main result regarding multi-product empirical processes from \cite{al2023covariance}. 
We then prove in Section \ref{sec:empiricalnew} our main results of this section, Lemmas~\ref{lem:sub-GaussianStandardizedProductProcessGeneral} and \ref{lem:sub-ExponentialStandardizedProductProcessGeneral}. 
Our proofs have been inspired by the techniques introduced in \cite{koltchinskii2017concentration} as well as \cite{ghattas2022non} and \cite{al2023covariance}. These works deal with product empirical processes in which the product is taken over two sub-Gaussian classes. In contrast, the results here pertain to product empirical processes over a special category of sub-Exponential classes that arise in the nonasymptotic analysis of the variance component $\theta(x,y)$.

\subsection{Background}\label{sec:backgroundempiricalprocess}
    Let $X, X_1,\dots,X_N \iid \P$ be a sequence of random variables on a probability space $(\Omega, \mathbb{P}).$ The empirical process indexed by a class $\mcF$ of functions on $(\Omega, \P)$ is given by
    \begin{align*}
        f \mapsto \frac{1}{N} \sum_{n=1}^N f(X_n) - \E f(X), \qquad f \in \mcF.
    \end{align*}
    For $s \ge 2,$ the order-$s$ multi-product empirical process indexed by $\mcF$  is given by
    \begin{align*}
        f \mapsto \frac{1}{N} \sum_{n=1}^N f^s(X_n)- \E f^s(X), \qquad f \in \mcF.
    \end{align*}
    For any function $f$ on $(\Omega, \P)$ and $\alpha \ge 1$, the Orlicz $\psi_\alpha$-norm of $f$ is defined as 
    \begin{align*}
        \normn{f}_{\psi_\alpha(\P)}
        = \inf \Bigl\{c>0:  \E_{X \sim \P} \bigl[\exp(|f(X)/c|^\alpha)\bigr] \le 2 \Bigr\}
        = \sup_{q \ge 1} \frac{\normn{f}_{L_q(\P)}}{q^{1/\alpha}}.
    \end{align*}
    The base measure will be clear from the context, and so we write $\normn{f}_{\psi_\alpha(\P)} = \normn{f}_{\psi_\alpha}$ and similarly for the $L_q$-norms. The corresponding Orlicz space $L_{\psi_\alpha}$ contains functions with finite Orlicz $\psi_\alpha$-norm. A class of functions $\mcG$ is $L$-sub-Gaussian if, for every $f,h \in \mcG \cup \{0\}$, 
    \begin{align*}
        \normn{f-h}_{\psi_2} \le L \normn{f-h}_{L_2}.
    \end{align*}
    For a sub-Gaussian class $\G$ it holds that, for every $f,h \in \mcG \cup \{0\}$ and $q \ge 1,$
    \begin{align*}
        \normn{f-h}_{L_q} 
        \le 
        c \sqrt{q}\normn{f-h}_{\psi_2} 
        \le c L\sqrt{q} \normn{f-h}_{L_2}.
    \end{align*}
    A class of functions $\mcE$ is $L$-sub-Exponential if, for every $f,h \in \mcE \cup \{0\}$, 
    \begin{align*}
        \normn{f-h}_{\psi_1} \le L \normn{f-h}_{L_2}.
    \end{align*}
  For a sub-Exponential class $\mcE$ it holds that, for every $f,h \in \mcE \cup \{0\}$ and $q \ge 1,$
    \begin{align*}
        \normn{f-h}_{L_q} 
        \le 
        c q \normn{f-h}_{\psi_1}
        \le c L q \normn{f-h}_{L_2}.
\end{align*}
Our results depend on Talagrand's $\gamma$-functional, whose definition we now recall.
\begin{definition}[{Talagrand's $\gamma$ functional, \cite{talagrand2022upper}}]
    Let $(\mcF, \mathsf{d})$ be a metric space. An admissible sequence of $\mcF$ is a collection of subsets $\mcF_s \subset \mcF$ whose cardinality satisfies $|\mcF_s| \le 2^{2^s}$ for $s \ge 1$, and $|\mcF_0| = 1$. Set 
    \begin{align*}
        \gamma_{2}(\mcF,\mathsf{d}) 
        = \inf \sup_{f \in \mcF} \sum_{ s \ge 0} 2^{s/2} \mathsf{d}(f,\mcF_s),
    \end{align*}
    where the infimum is taken over all admissible sequences, and $\sfd(f, \mcF_s) = \inf_{g \in \mcF_s} \sfd(f,g).$ We write $\gamma_2(\mcF, \psi_2)$ when the distance on $\mcF$ is induced by the $\psi_2$-norm.
\end{definition}
We now introduce a technical result that will be used in the subsequent proofs.

\begin{lemma}\label{lem:subadditivegamma}
    Let $\mcG, \mcH$ be arbitrary subsets of a normed space endowed with the norm $\|\cdot\|.$ Define $\mcF = \mcG + \mcH$, which inherits this norm. Then 
    \begin{align*}
        \gamma_2(\mcF, \sfd) \le 
        2(\sup_{g\in\mcG} \|g\|+\sup_{h\in\mcH} \|h\|) + \sqrt{2} (\gamma_2(\mcG, \sfd)+\gamma_2(\mcH, \sfd)),
    \end{align*}
    where $\sfd(a,b)=\|a-b\|.$ Moreover, if $\mcG$ and $\mcH$ both either contain $0$ or are symmetric,
    \begin{align*}
        \gamma_2(\mcF, \sfd) \lesssim \gamma_2(\mcG, \sfd)+\gamma_2(\mcH, \sfd).
    \end{align*}
\end{lemma}
\begin{proof}
    Let $(\mcG_s)_s, (\mcH_s)_s$ be admissible sequences for $\mcG$ and $\mcH$ respectively. We can construct an admissible sequence for $\mcF$ as follows. Let $\mcF_0$ be an arbitrary element of $\mcF$, and for $s \ge 1$, set $\mcF_s = \mcG_{s-1}+ \mcH_{s-1} = \{ g+h: g \in \mcG_{s-1}, h \in \mcH_{s-1}\}$. This ensures admissibility since $|\mcF_s| \le |\mcG_{s-1}||\mcH_{s-1}| \le 2^{2^{s-1}}2^{2^{s-1}} = 2^{2^{s}}.$ Note then that 
    \begin{align*}
        \gamma_2(\mcF, \sfd)
        &\le \sup_{f \in \mcF} \mathsf{d}(f,\mcF_0)
        + \sup_{f \in \mcF} \sum_{ s \ge 1} 2^{s/2} \mathsf{d}(f,\mcF_s)
        \\
        &=\sup_{f \in \mcF} \mathsf{d}(f,\mcF_0) + 
        \sup_{g \in \mcG, h \in \mcH} \sum_{ s \ge 1} 2^{s/2} \mathsf{d}(g+h,\mcG_{s-1}+\mcH_{s-1})\\
        &\le \sup_{f \in \mcF} \mathsf{d}(f,\mcF_0) + \sup_{g \in \mcG} \sum_{ s \ge 1} 2^{s/2} \mathsf{d}(g,\mcG_{s-1})
        + \sup_{h \in \mcH} \sum_{ s \ge 1} 2^{s/2} \mathsf{d}(h,\mcH_{s-1})\\
        &=
        \sup_{f \in \mcF} \mathsf{d}(f,\mcF_0)+
        \sqrt{2}\sup_{g \in \mcG} \sum_{ s \ge 0} 2^{s/2} \mathsf{d}(g,\mcG_{s})
        + \sqrt{2} \sup_{h \in \mcH} \sum_{ s \ge 0} 2^{s/2} \mathsf{d}(h,\mcH_{s}).
        \end{align*}
        Noting that 
        \begin{align*}
            \sup_{f \in \mcF} \mathsf{d}(f,\mcF_0) \le \tdiam(\mcF) \le \tdiam(\mcG) + \tdiam(\mcH) \le 2(\sup_{g\in\mcG} \|g\|+\sup_{h\in\mcH} \|h\|), 
        \end{align*}
        and taking the infimum with respect to $(\mcG_s)_s$ and $(\mcH_s)_s$ on both sides yields the first result. In the case that $\mcG$ and $\mcH$ both either contain $0$ or are symmetric, we have that $\sup_{g\in\mcG} \|g\| \lesssim \gamma_2(\mcG, \sfd)$ by \cite[Lemma 4.6]{al2025sharp}, and similarly for $\mcH.$
\end{proof}

The next result provides optimal high probability bounds on order-$s$ multi-product empirical processes.

\begin{theorem}[{\cite[Theorem 2.2]{al2025sharp}}] \label{thm:multiproductemp}
    Assume that $0 \in \mcF$ or that $\mcF$ is symmetric (i.e., $f \in \mcF \implies -f \in \mcF$). For any $s \ge 2$ and $t \ge 1$, it holds with probability at least $1-e^{-t}$ that, for any $f \in \mcF,$
    \begin{align*}
        \abs{
        \frac{1}{N} \sum_{n=1}^n f^s(X_n) - \E f^s(X)
        } \lesssim_s 
        \frac{\gamma_2(\mcF, \psi_2) d_{\psi_2}^{s-1}(\mcF)}{\sqrt{N}}
        \lor 
        \frac{\gamma_2^s(\mcF, \psi_2)}{N}
        \lor 
        d_{\psi_2}^s(\mcF) \inparen{ \sqrt{\frac{t}{N}}
        \lor 
        \frac{t^{s/2}}{N}},
    \end{align*}
    where $\lesssim_s$ indicates that the inequality holds up to a universal positive constant depending only on $s$, and $d_{\psi_2}(\mcF) = \sup_{f \in \mcF} \|f\|_{\psi_2}.$
\end{theorem}

\subsection{Product Sub-Gaussian and Sub-Exponential Classes}\label{sec:empiricalnew}
The goal of this section is to apply Theorem~\ref{thm:multiproductemp} to the problem of bounding product empirical processes indexed by a function class $\mcF$, given by
    \begin{align*}
        f,g \mapsto \frac{1}{N} \sum_{n=1}^N f(X_n)g(X_n)- \E[f(X)g(X)], \qquad f,g \in \mcF.
    \end{align*}
    Bounding the suprema of such processes arises in two important ways in this work. First, in establishing uniform bounds on the deviation of the sample covariance function $\hatk$ from its expectation, in which case the indexing class $\mcF$ is sub-Gaussian and we refer to it as a product sub-Gaussian process. Second, in establishing uniform bounds on the deviation of the sample variance component $\hattheta$ from its expectation, in which case $\mcF$ is sub-Exponential and we refer to it as a sub-Exponential product process.

We now present our two main results of this section. The first bounds the suprema of the product process indexed by two sub-Gaussian classes,  and the second bounds the suprema of the product process indexed by two sub-Exponential classes. 

We recall here that $u, u_1,\dots, u_N $ are i.i.d. centered sub-Gaussian and pre-Gaussian random functions on $D=[0,1]^d$ taking values on the real line and with covariance function $k$. We assume that these functions are Lebesgue almost-everywhere continuous with probability one.  Denote by $\tildeu, \tildeu_1, \ldots, \tildeu_N$ their normalized versions as defined in \eqref{eq:standardized}. Further, recall that a pre-Gaussian process $u$ is one for which there exists a centered Gaussian process, $v$, that has the same covariance structure as $u$. Following \cite[page 261]{ledoux2013probability}, we refer to $v$ as the Gaussian process \textit{associated to} $u$.

\begin{lemma} \label{lem:sub-GaussianStandardizedProductProcessGeneral}
        It holds with probability at least $1-e^{-t}$ that, for any $x,y \in D,$
        \begin{align*}
            \abs{
            \frac{1}{N} \sum_{n=1}^N 
                \tildeu_n(x)
                \tildeu_n(y)
                -
                \E [\tildeu(x)\tildeu(y)]} 
        \lesssim 
            \sqrt{\frac{t}{N}}
            \lor 
            \frac{t}{N}
            \lor 
            \frac{\E[\sup_{x \in D} \tildev(x)]
            }{\sqrt{N}}
            \lor 
            \frac{(\E[\sup_{x \in D} \tildev(x)])^2
            }{N},
        \end{align*}
        where $\tildev$ is the Gaussian process associated with $\tildeu.$
        \end{lemma}
    \begin{proof}
    For $x \in D$, let $\ell_x: v \mapsto \ell_x(v) = v(x)$ be the evaluation functional at $x \in D$. We then have
        \begin{align*}
            \sup_{x,y \in D}
            \abs{
            \frac{1}{N} \sum_{n=1}^N 
                \tildeu_n(x)
                \tildeu_n(y)
                -
                \E [\tildeu(x)\tildeu(y)]} 
            &=
            \sup_{x,y \in D}
            \abs{
            \frac{1}{N} \sum_{n=1}^N 
                \ell_x(\tildeu_n)
                \ell_y(\tildeu_n)
                -
                \E [\ell_x(\tildeu)\ell_y(\tildeu)]
                } \\
            &\le 
            \sup_{x \in D}
            \abs{
            \frac{1}{N} \sum_{n=1}^N 
                \ell^2_x(\tildeu_n)
                -
                \E [\ell^2_x(\tildeu)]
                } \\
            &+\frac{1}{2}
            \sup_{x,y \in D}
            \abs{
            \frac{1}{N} \sum_{n=1}^N 
                (\ell_x-\ell_y)^2(\tildeu_n)
                -
                \E [(\ell_x-\ell_y)^2(\tildeu)]
                } \\
            &\lesssim
            \sup_{f \in \mcF}
            \abs{
            \frac{1}{N} \sum_{n=1}^N 
                f^2(\tildeu_n)
                -
                \E f^2(\tildeu)
                },
        \end{align*}
        where the first inequality follows by the fact that for two constants $a,b$, $ab = \frac{1}{2}(a^2 + b^2 - (a-b)^2),$ and the second inequality follows for $\mcF := \{ c_1\ell_x - c_2\ell_y: x,y\in D, c_1,c_2 \in \{0,1\} \}.$ Note that $0 \in \mcF$ since we can take $c_1=c_2=0.$ We then have
        \begin{align*}
            d_{\psi_2}(\mcF)
            = \sup_{f \in \mcF} \|f\|_{\psi_2}
            \le 
             \sup_{x \in D} \| \ell_x(\tildeu_n)\|_{\psi_2}
            \lor 
            \sup_{x,y \in D} \| (\ell_x-\ell_y)(\tildeu_n)\|_{\psi_2}
            \lesssim 1.
        \end{align*}
        Define $\mcG := \{c \ell_x: x\in D, c\in \{0,1\}\},$ and note that $\mcF \subset \mcG - \mcG,$ from which we have 
        \begin{align*}
            \gamma_2(\mcF,\psi_2)
            \le 
            \gamma_2(\mcG-\mcG, \psi_2)
            \lesssim
            \gamma_2(\mcG,\psi_2)
            \lesssim
            \gamma_2(\mcG,L_2),
        \end{align*}
        where the second inequality holds by Lemma~\ref{lem:subadditivegamma} and the third inequality holds by the equivalence of $L_2$ and $\psi_2$ norms for linear functionals. Next, let $\tildev$ be the Gaussian process associated to $\tildeu$ and define
        \begin{align*}
            \mathsf{d}_{\tildev}(x,y) 
            := \sqrt{\E \Bigl[\bigl(\tildev(x)-\tildev(y)\bigr)^2 \Bigr]} 
            = \normn{\ell_x(\cdot) - \ell_y(\cdot)}_{L_2}, \qquad x,y\in D.
    \end{align*}
    Then,
    \begin{align*}
        \gamma_2(\mcG, L_2)
        \lesssim
        \gamma_2(\{\ell_x: x\in D\}, L_2)
        = 
        \gamma_2(D, \mathsf{d}_{\tildev})
        \asymp
        \E \insquare{ \sup_{x \in D} \tildev(x) },
    \end{align*}
    where the first inequality holds by the fact that for any constant $c \in \R,$ function class $\mcF$ and metric $\sfd$, $\gamma_2(c \mcF, \sfd) \le |c|\gamma_2(\mcF, \sfd),$ and the second inequality holds by the definition of $\sfd_{\tildev}$, (see also \cite[Theorem 4]{koltchinskii2017concentration}, \cite[Proposition 3.1]{al2023covariance}). The final result therefore follows by invoking Theorem~\ref{thm:multiproductemp} with $s=2.$
    \end{proof}

    \begin{lemma} \label{lem:sub-ExponentialStandardizedProductProcessGeneral}
       It holds with probability at least $1-e^{-t}$ that, for any $x,y \in D,$
        \begin{align*}
            \abs{
            \frac{1}{N} \sum_{n=1}^N 
                \tildeu_n^2(x)
                \tildeu_n^2(y)
                -
                \E [\tildeu^2(x)\tildeu^2(y)]} 
        \lesssim 
            \sqrt{\frac{t}{N}}
            \lor 
            \frac{t^2}{N}
            \lor 
            \frac{\E[\sup_{x \in D} \tildev(x)]
            }{\sqrt{N}}
            \lor 
            \frac{(\E[\sup_{x \in D} \tildev(x)])^4
            }{N},
        \end{align*}
        where $\tildev$ is the Gaussian process associated with $\tildeu.$
        \end{lemma}
    \begin{proof}
    For $x \in D$, let $\ell_x: v \mapsto \ell_x(v) = v(x)$ be the evaluation functional at $x \in D$. We then have
        \begin{align*}
            \sup_{x,y \in D}
            \abs{
            \frac{1}{N} \sum_{n=1}^N 
                \tildeu_n^2(x)
                \tildeu_n^2(y)
                -
                \E [\tildeu^2(x)\tildeu^2(y)]} 
            &=
            \sup_{x,y \in D}
            \abs{
            \frac{1}{N} \sum_{n=1}^N 
                \ell_x^2(\tildeu_n)
                \ell_y^2(\tildeu_n)
                -
                \E [\ell_x^2(\tildeu)\ell_y^2(\tildeu)]
                } \\
            & \hspace{-1cm} \le 
            \frac{1}{3}
            \sup_{x \in D}
            \abs{
            \frac{1}{N} \sum_{n=1}^N 
                \ell^4_x(\tildeu_n)
                -
                \E [\ell^4_x(\tildeu)]
                } \\
            & \hspace{-1cm} +\frac{1}{12}
            \sup_{x,y \in D}
            \abs{
            \frac{1}{N} \sum_{n=1}^N 
                (\ell_x-\ell_y)^4(\tildeu_n)
                -
                \E [(\ell_x-\ell_y)^4(\tildeu)]
                } \\
                & \hspace{-1cm} +\frac{1}{12}
            \sup_{x,y \in D}
            \abs{
            \frac{1}{N} \sum_{n=1}^N 
                (\ell_x+\ell_y)^4(\tildeu_n)
                -
                \E [(\ell_x+\ell_y)^4(\tildeu)]
                } \\
            & \hspace{-1cm} \lesssim
            \sup_{f \in \mcF}
            \abs{
            \frac{1}{N} \sum_{n=1}^N 
                f^4(\tildeu_n)
                -
                \E f^4(\tildeu)
                },
        \end{align*}
        where the first inequality follows by the fact that for two constants $a,b$, $a^2b^2 = \frac{1}{12}((a+b)^4 + (a-b)^4 - 2a^4-2b^4),$ and the second inequality follows for $\mcF := \{ c_1\ell_x - c_2\ell_y: x,y\in D, c_1,c_2 \in \{-1,0,1\} \}.$ Note that $0 \in \mcF$ since we can take $c_1=c_2=0.$ We then have 
        \begin{align*}
            d_{\psi_2}(\mcF)
            = \sup_{f \in \mcF} \|f\|_{\psi_2}
            \le 
             \sup_{x \in D} \| \ell_x(\tildeu_n)\|_{\psi_2}
            \lor 
            \sup_{x,y \in D} \| (\ell_x-\ell_y)(\tildeu_n)\|_{\psi_2}
            \lesssim 1.
        \end{align*}
        Note that for $\mcG = \{ c \ell_x: x \in D, c \in \{-1,0,1\}\},$ we have $\mcF \subset \mcG- \mcG.$ Using once more the fact that for any constant $c \in \R,$ function class $\mcF$ and metric $\sfd$, $\gamma_2(c \mcF, \sfd) \le |c|\gamma_2(\mcF, \sfd),$ we have that $\gamma_2(\mcG) \lesssim \gamma_2(\{\ell_x: x\in D\}).$ By an identical argument to the one used in the proof of Lemma~\ref{lem:sub-GaussianStandardizedProductProcessGeneral}, we have $\gamma_2(\mcF, \psi_2)\lesssim \E[\sup_{x\in D} \tildeu(x)].$ The final result therefore follows by invoking Theorem~\ref{thm:multiproductemp} with $s=4.$
        \end{proof}

\begin{remark}\label{rem:genericChainingControl}
Our results are expressed in terms of the supremum of the Gaussian process $v$ associated with the observed process $u$ rather than directly in terms of the supremum of $u$. A key step in proving Lemmas~\ref{lem:sub-GaussianStandardizedProductProcessGeneral} and \ref{lem:sub-ExponentialStandardizedProductProcessGeneral} is bounding Talagrand's $\gamma$-functional, which arises from Theorem~\ref{thm:multiproductemp}. In general, the task of controlling the $\gamma$-functional efficiently is extremely difficult. One remarkable exception is Talagrand's majorizing measures theorem, which relates $\gamma_2(\mcF, \sfd)$ to the expected supremum of a Gaussian process. We leverage the pre-Gaussianity of $u$ to establish equivalence between $\gamma_2$ functionals defined with respect to $L_2(\P)$ (where $\P$ is the law of $u$) and the natural metric $\sfd_{\tildev}$ of $\tildev$. Applying Talagrand’s theorem, we obtain an upper bound in terms of $\E[\sup_{x \in D} \tildev(x)]$. Extending this approach to instead bound $\E[\sup_{x \in D} \tildeu(x)]$ from below in terms of $\gamma_2$ or finding an alternative approach that allows a bound in terms of $u$ requires further investigation which we leave to future work.
\end{remark}

\section{Lower Bound for Universal Thresholding}\label{sec:LowerBound}
This section contains the proof of Theorem~\ref{thm:UniversalThreshLowerBound}. The idea is to first reduce the covariance operator estimation problem to a finite-dimensional covariance matrix estimation problem, then apply Theorem 4 in \cite{cai2011adaptive} which proves a lower bound for universal thresholding in the finite-dimensional covariance matrix estimation problem.  The reduction is based on the recent technique developed in \cite[Proposition 2.6]{al2024optimal}.

\begin{proof}[Proof of Theorem~\ref{thm:UniversalThreshLowerBound}]
    For $m \in \N$ to be chosen later, let $\{ I_i \}_{i=1}^m$ be a uniform partition of $D$ with $\tvol(I_i) = m^{-1}.$ For any positive definite matrix $H = (h_{ij}) \in \R^{m \times m}$, define the covariance operator $C_H$ with corresponding covariance function
\begin{align*}
    k_H(x,y) = \sum_{i,j=1}^m h_{ij} \indicator_i(x)\indicator_j(y),
\end{align*}
where $\indicator_i(x) := \indicator \{ x \in I_i\}$. Then, $C_H$ is a positive definite operator since, for any $\psi \in L_2(D),$ 
\begin{align*}
    \int_{ D \times D} k_H(x,y) \psi(x) \psi(y) dxdy 
    =  \inp{H \bar{\psi},\bar{\psi}} > 0,
\end{align*}
where $\bar{\psi}= (\bar{\psi}_1,\dots, \bar{\psi}_m)^\top$ and $\bar{\psi}_i = \int_{I_i} \psi(x)dx.$ Note further that for $u_n \sim \text{GP}(0, C_H)$, $u_n$ is almost surely a piecewise constant function that can be written as $u_n(x) = \sum_{i=1}^m z_n^{(i)} \indicator_i(x),$ for $(z_n^{(1)},\dots, z_n^{(m)}) \sim N(0,H).$ Consider next the sparse covariance matrix class $\mcU^*_q(m, R_q)$ studied in \cite{cai2011adaptive} and defined in \eqref{eq:SparseWeightedMatrixClass}. For any $H \in \mcU^*_q(m, m^{1/q}R_q)$ it holds that $C_H \in \mcK_q^*(R_q)$. 
\begin{align*}
    &\sup_{x\in D} \int_D (k_H(x,x)k_H(y,y))^{(1-q)/2} |k_H(x,y)|^q dy\\
    &=
    \sup_{x\in D} \int_D \sum_{i,j=1}^m  (h_{ii} h_{jj})^{(1-q)/2}|h_{ij}|^q \indicator_i(x) \indicator_j(y) dy\\
    &= 
    \max_{i \le m } \int_D \sum_{j=1}^m  (h_{ii} h_{jj})^{(1-q)/2}|h_{ij}|^q \indicator_j(y) dy\\
    &= 
    \max_{i \le m } \sum_{j=1}^m  (h_{ii} h_{jj})^{(1-q)/2}|h_{ij}|^q \int_D  \indicator_j(y) dy\\
    &= 
    \frac{1}{m}
    \max_{i \le m } \sum_{j=1}^m  (h_{ii} h_{jj})^{(1-q)/2}|h_{ij}|^q \le  R_q^q.
\end{align*}
    Further, for $H \in \mcU^*_q(m, R_q)$, we have that $m H \in \mcU^*_q(m,m^{1/q} R_q).$ Now, we will choose $\tildeH_0 = (\tildeh_{0, ij})\in \mcU^*_q(m, R_q)$ to be the covariance matrix constructed in the proof of \cite[Theorem 4]{cai2011adaptive}. Namely, let $s_1 = \lceil (R_q^q-1)^{1-q} (\log m/N)^{-q/2}\rceil +1$, and set 
    \begin{align*}
        \tildeh_{0,ij} = 
        \begin{cases}
            1 \qquad &\text{if}\qquad  1 \le i=j \le s_1,\\
            R_q^q \qquad &\text{if}\qquad  s_1+1 \le i=j \le m,\\
            4^{-1} R_q^q\sqrt{\log m/N}&\text{if}\qquad  1 \le i\neq j \le s_1,\\
            0 &\text{otherwise}.
        \end{cases}
    \end{align*}
    Then, $H_0 := m \tildeH_0 \in \mcU_q^*(m,m^{1/q} R_q)$ and $C_0 := C_{H_0} \in \mcK_q^*(R_q)$. Next, let $\hatC^{\mathsf{U}}_{\gamma_N}$ have covariance function $\hatt_{\gamma_N}(x,y) = \hatk(x,y) \indicator \{|\hatk(x,y)| \ge \gamma_N\},$ i.e. $\hatC^{\mathsf{U}}_{\gamma_N}$ is the (universal) thresholding covariance estimator with threshold $\gamma_N$. By definition of the operator norm,
    \begin{align*}
        \normn{\hatC^{\mathsf{U}}_{\gamma_N} - C_{H_0}}
        &=
        \sup_{\|f\|_{L_2(D)}=\|g\|_{L_2(D)}=1} 
        \int f(x) \inparen{
        \int (\hatt_{\gamma_N}(x,y) - k_{H_0}(x,y)) g(y)dy
        }dx\\
        &\ge  
        \sup_{a, b \in \mcS_{m-1}} 
        \int f_a(x) \inparen{
        \int (\hatt_{\gamma_N}(x,y) - k_{H_0}(x,y)) f_b(y)dy
        }dx\\
        &=
        \sup_{a, b \in \mcS_{m-1}} 
        m
        \sum_{i,j=1}^m
        a_i b_j
        \iint_{D \times D} 
        1_i(x) 1_j(y) (\hatt_{\gamma_N}(x,y) - k_{H_0}(x,y)) dy
        dx\\
        &=
        \sup_{a, b \in \mcS_{m-1}} 
        m
        \sum_{i,j=1}^m
        a_i b_j
        \inparen{
        \iint_{I_i \times I_j} \hatt_{\gamma_N}(x,y) dxdy - m^{-2} h_{0,ij}}\\
         &=
        m
        \sup_{a, b \in \mcS_{m-1}} 
        \inp{
        a,
        \inparen{\hatT_{m, \gamma_N}- m^{-2} H_0 }
        b
        }\\
         &=
        \normn{
            m  \hatT_{m, \gamma_N} -  m^{-1}H_0 
        } =
        \normn{
            m  \hatT_{m, \gamma_N} - \tildeH_0 
        },
    \end{align*}
where $f_a(x)  := \sqrt{m} \sum_{i=1}^m a_i \indicator_i(x)$ and the lower bound holds since $a$ is a unit vector and therefore $\|f_a\|_{L_2(D)} = 1$. $f_b$ is defined analogously. Note further that we have defined the $m \times m$ matrix $\hatT_{m, \gamma_N} $ with $(i,j)$-th element
\begin{align*}
    \iint_{I_i \times I_j} \hatt_{\gamma_N}(x,y) dxdy 
    &= 
    \iint_{I_i \times I_j} 
    \frac{1}{N} \sum_{n=1}^N u_n(x) u_n(y) \indicator \{|\hatk(x,y) \ge \gamma_N|\}
    dxdy,
\end{align*}
where $u_1,\dots, u_n \iid \text{GP}(0, C_{H_0})$. Define $v_n = m^{-1/2} u_n$ for $n=1,\dots, N$, and so $v_1,\dots, v_n \iid \text{GP}(0, C_{\tildeH_0}).$ Then, we have 
\begin{align*}
    \inf_{\gamma_N \ge 0}  
    \E_{\{u_n\}_{n=1}^N \iid \text{GP}(0, C_{H_0})} 
    \normn{\hatC^{\mathsf{U}}_{\gamma_N} - C_{H_0}}
    & \ge 
    \inf_{\gamma_N \ge 0} 
    \E_{\{u_n\}_{n=1}^N \iid \text{GP}(0, C_{H_0})} 
    \normn{
        m  \hatT_{m, \gamma_N} - \tildeH_0 
    }\\
    &=
    \inf_{\gamma_N \ge 0} 
    \E_{\{v_n\}_{n=1}^N \iid \text{GP}(0, C_{\tildeH_0})} 
    \normn{
        \tildeT_{m, \gamma_N} - \tildeH_0
    },
\end{align*}
where $\tildeT_{m, \gamma_N} := m T_{m, \gamma_N}$. Since the samples $u_n$ are piecewise constant functions, we have for any block $I_i \times I_j$  for which the indicator is equal to 1 that
\begin{align*}
    (\tildeT_{m, \gamma_N})_{ij}
    &=
    m^2 
    \iint_{I_i \times I_j} 
    \frac{1}{N} \sum_{n=1}^N v_n(x) v_n(y) \indicator \{|\hatk(x,y)| \ge \gamma_N\}
    dxdy\\
    &=
    m^2 
    \iint_{I_i \times I_j} 
    \frac{1}{N} \sum_{n=1}^N v_n(x) v_n(y) 
    dxdy\\
    &=
    m^2 
    \iint_{I_i \times I_j} 
    \frac{1}{N} \sum_{n=1}^N w_n^{(i)} w_n^{(j)}
    dxdy =
    \frac{1}{N} \sum_{n=1}^N w_n^{(i)} w_n^{(j)},
\end{align*}
for $w_1,\dots, w_N \iid N(0,\tildeH_0)$ with $w_n = \bigl(w_n^{(1)},\dots, w_n^{(m)}\bigr)$ for $1 \le n \le N$. Therefore, $T_{m, \gamma_N}$ is a universally thresholded sample covariance matrix estimator. (Note that the scaling inside the indicator is not an issue, as the infimum is over all positive $\gamma_N$.) It follows immediately by \cite[Theorem 4]{cai2011adaptive} that if $m$ is chosen to satisfy $N^{5q} \le m \le e^{o(N^{1/3})}$ and $8 \le R_q^q \le \min \Bigl\{m^{1/4}, 4 \sqrt{\frac{N}{\log m}} \Bigr\}$, then, for sufficiently large $N,$ 
\begin{align*}
    \inf_{\gamma_N \ge 0} 
    \E_{\{v_n\}_{n=1}^N \iid \text{GP}(0, C_{\tilde{H}_0})} 
    \normn{
        \tildeT_{m, \gamma_N} - \tildeH_0
    }
    \gtrsim
    (R_q^q)^{2-q} \inparen{\frac{\log m}{N}}^{(1-q)/2}.
\end{align*}

     Note then that for $u \sim \text{GP}(0, C_{H_0})$, we have that $u(x) = \sum_{i=1}^m z^{(i)} \indicator_i(x),$ for $z= \bigl(z^{(1)}, \dots, z^{(m)}\bigr) \sim N(0,H_0).$ Further, the normalized process $\tildeu(x)$ is of the form $\tildeu(x) = u(x)/\sqrt{k_{H_0}(x,x)} = \sum_{i=1}^m g^{(i)} \indicator_i(x),$ for $g= \bigl(g^{(1)}, \dots, g^{(m)}\bigr) \sim N(0,D_0)$ where $D_0$ has elements $d_{0, ij} = m\tildeh_{0, ij} / \sqrt{\tvar(z^{(i)}z^{(j)})}$. By \cite[Lemma 2.3]{van2017spectral},
     \begin{align*}
         \E \insquare{\sup_{x \in D} \tildeu(x)}
         = \E \insquare{\max_{i \le m} g^{(i)}}
         \lesssim 
         \max_{i \le m} \sqrt{d_{0, ii} \log (i+1)}.
     \end{align*}
     Moreover, for $1 \le i \le m,$ using that  $\E[(z^{(i)})^\zeta] = (\zeta-1)!! (\tvar(z^{(i)}))^{\zeta/2}$ for any non-negative even integer $\zeta$, we have 
     \begin{align*}
         \tvar \bigl((z^{(i)})^2\bigr)
         &= \E\insquare{(z^{(i)})^4} -\Bigl(\E\insquare{(z^{(i)})^2}\Bigr)^2
         = 3 m^2\tildeh^2_{0,ii} - m^2\tildeh^2_{0,ii}
         = m^2\tildeh^2_{0,ii}.
     \end{align*}
     Therefore, 
     \begin{align*}
         d_{0, ii} = 
         \frac{m \tildeh_{0,ii}}{ \sqrt{\tvar((z^{(i)})^2)}}
         = \frac{1}{\sqrt{2}}.
     \end{align*}
     Plugging this in yields $ \E [\sup_{x \in D} \tildeu(x)] \lesssim \sqrt{\log m}$ and so the bound becomes
     \begin{align*}
         (R_q^q)^{2-q} \inparen{\frac{\log m}{N}}^{(1-q)/2}
         \gtrsim
         (R_q^q)^{2-q} \inparen{\frac{ (\E [\sup_{x \in D} \tildeu(x)] )^2 }{N}}^{(1-q)/2},
     \end{align*}
     as desired.
\end{proof}

\section{Error Analysis for Nonstationary Weighted Covariance Models}\label{sec:NonstationaryWeightedCovAnalysis}
Throughout this section, $u$ denotes a centered Gaussian process on $D=[0,1]^d$ with covariance function $k_\lambda(x,y)$ satisfying Assumptions~\ref{ass:isotropicandlengthscale} and \ref{ass:WeightFunction}. We make repeated use of the following easily verifiable facts:
\begin{enumerate}
    \item $\tildek_\lambda(r) = \tildek_1(\lambda^{-1}r).$
    \item $\sigma_\lambda(\lambda x; \alpha) = \sigma_1(\lambda^{1-\alpha/2}x; \alpha).$
    \item For any $\lambda>0$, $\alpha \in (0,1/2)$, $1 \le \sigma_{\lambda}(x; \alpha) \le \exp(d /\lambda^\alpha )$. 
\end{enumerate}
In this section, we use the notation ``$(E), ~\lambda \to 0^+$'' to mean that there is a universal constant $\lambda_0>0$ such that if $\lambda < \lambda_0$, then $(E)$ holds. We interchangeably use the term ``\textit{for sufficiently small $\lambda$.}''

\begin{lemma} \label{lem:TraceBound}
    It holds that
    \begin{align*}
        \ttrace(C) 
        =  2^{-d/2} \tildek(0)\lambda^{\alpha d/2} \inparen{\int_0^{\sqrt{2/\lambda^{\alpha}}} e^{t^2} dt}^d.
    \end{align*}
    Therefore, for sufficiently small $\lambda$, it holds that 
    \begin{align*}
        \ttrace(C) 
        \asymp 
        \tildek(0)  \lambda^{\alpha d} e^{2d /\lambda^\alpha}.
    \end{align*}
\end{lemma}
\begin{proof}
    By definition, 
     \begin{align*}
        \ttrace(C) 
        = \int_D k_\lambda(x,x) dx 
        = \tildek(0) \int_D \sigma^2_\lambda(x)  dx 
        = \tildek(0) \normn{\sigma_\lambda}^2_{L_2(D)}.
    \end{align*}
    Then, note that 
    \begin{align*}
        \normn{\sigma_\lambda}^2_{L_2(D)} 
        &= \int_D \exp(2\lambda^{-\alpha} \normn{x}^2) dx 
        =\int_D \exp(2\lambda^{-\alpha} \sum_{j=1}^d x_j^2) dx \\
        &=\prod_{j=1}^d \int_0^1 \exp(2 \lambda^{-\alpha} x_j^2) dx 
        = 2^{-d/2} \lambda^{\alpha d/2} \inparen{
        \int_0^{\sqrt{2/\lambda^{\alpha}}} e^{t^2} dt
        }^d.
    \end{align*}
    In the last line of the above working, we have used the fact that $\int_0^1 e^{cz^2}dz= \frac{1}{\sqrt{c}} \int_0^{\sqrt{c}} e^{t^2}dt.$ 
    Then, we have
        \begin{align*}
            \int_0^{\sqrt{2/\lambda^{\alpha}}} e^{t^2} dt
            = e^{2/\lambda^\alpha} 
            \mcD (\sqrt{2 / \lambda^\alpha}),
        \end{align*}
        where $\mathcal{D}(\cdot)$ is the Dawson function. By \cite[Section 7.1]{abramowitz1968handbook}, it holds for $\lambda$ sufficiently small
        \begin{align*}
            \mathcal{D} (\sqrt{2 / \lambda^\alpha})
            = 
            \frac{1}{2\sqrt{2}} \lambda^{\alpha/2}
            +
            \frac{1}{8 \sqrt{2}} \lambda^{3\alpha/2}
            +
            \frac{3}{32 \sqrt{2}} \lambda^{5\alpha/2} + \cdots
            \asymp
            \lambda^{\alpha/2}.
        \end{align*}
        Therefore, for $\lambda$ sufficiently small, 
        \begin{align*}
             \normn{\sigma_\lambda}^2_{L_2(D)} 
             \asymp
             2^{-d/2} \lambda^{\alpha d /2} e^{2d/\lambda^\alpha} \lambda^{\alpha d/2}
             =
             \lambda^{\alpha d} e^{2d/\lambda^\alpha}. 
        \end{align*}
        
\end{proof}

\begin{lemma}\label{lem:supklambda_positiveexponential} 
It holds that
    \begin{align*}
        \sup_{x\in D} \int_D |k_\lambda(x,y)|^q dy  
        \asymp 
        q^{-d}
        \lambda^{d(\alpha+1)} 
        e^{2d q /\lambda^{\alpha}}
     \int_0^\infty r^{d-1}\tildek_1(r)^q dr, 
    \qquad \lambda \to 0^+.
    \end{align*}
\end{lemma}
\begin{proof}
    Starting with the upper bound, we have 
\begin{align*}
    \sup_{x\in D} \int_D |k_\lambda(x,y)|^q dy
    &=
    \sup_{x\in D} 
    \int_D 
    e^{q \|x\|^2/\lambda^\alpha}e^{q \|y\|^2/\lambda^\alpha} 
    |\tildek_\lambda(\normn{x-y})|^q dy\\
    &\le 
    \sup_{x\in D} e^{q \|x\|^2/\lambda^\alpha}
    \sup_{x\in D}
    \int_D 
    e^{q \|y\|^2/\lambda^\alpha} 
    |\tildek_\lambda(\normn{x-y})|^q dy\\
    &=
    e^{d q/\lambda^\alpha}
    \int_D 
    e^{q \|y\|^2/\lambda^\alpha} 
    |\tildek_\lambda(\normn{y})|^q dy\\
    &=
    e^{d q/\lambda^\alpha}
    \int_D 
    e^{q \|y\|^2/\lambda^\alpha} 
    |\tildek_1(\normn{ \lambda^{-1} y})|^q dy\\
    &=
    \lambda^{d}  
    e^{d q/\lambda^\alpha}
    \int_{[0, \lambda^{-1}]^d }
    e^{q \lambda^{2-\alpha} \| y\|^2} 
    |\tildek_1(\normn{ y})|^q dy,
    \end{align*}
    where the last line follows by substituting $y \mapsto \lambda^{-1}y$, and noting that the transformation has Jacobian $\lambda^d$. Therefore, treating $y = (y_1,\dots, y_d)$ as a random vector with $y_1,\dots, y_d \iid \text{unif}([0, \lambda^{-1}])$
    \begin{align*}
     \sup_{x\in D} \int_D |k_\lambda(x,y)|^q dy &\le
    e^{d q/\lambda^\alpha}
    \inparen{\lambda^{d} 
    \int_{[0, \lambda^{-1}]^d }
    e^{q \lambda^{2-\alpha} \| y\|^2} 
    |\tildek_1(\normn{ y})|^q dy}\\
    &=
    e^{d q/\lambda^\alpha}
    \E_y [
    e^{q \lambda^{2-\alpha} \| y\|^2} 
    |\tildek_1(\normn{ y})|^q
    ]\\
    &\le  
    e^{d q/\lambda^\alpha}
    \E_y [
    e^{q \lambda^{2-\alpha} \| y\|^2}
    ]
    \E_y [
    |\tildek_1(\normn{ y})|^q
    ]\\
    &=
    e^{d q/\lambda^\alpha}
    \inparen{
        \lambda^d
        \int_{[0,\lambda^{-1}]^d}
        e^{q \lambda^{2-\alpha} \| y\|^2}
        dy
    }
    \inparen{
    \lambda^d
        \int_{[0,\lambda^{-1}]^d}
        |\tildek_1(\|y\|)|^q
        dy
    }\\
    &=
    \lambda^{2d}  
    e^{d q/\lambda^\alpha}
    \inparen{
        \int_{[0,\lambda^{-1}]^d}
        e^{q \lambda^{2-\alpha} \| y\|^2}
        dy
    }
    \inparen{
        \int_{[0,\lambda^{-1}]^d}
        |\tildek_1(\|y\|)|^q
        dy
    }\\
    &=: \lambda^{2d}  
    e^{d q/\lambda^\alpha}
    I_1 \times I_2,
\end{align*}
where we have used the covariance inequality \cite[Theorem 4.7.9]{berger2001statistical}, which states that for non-decreasing $g$ and non-increasing $h$, $\E_y [g(y)h(y)] \le (\E_y [g(y)]) (\E_y [h(y)])$. For the first integral, using the substitution $v_j^2 = q \lambda^{2-\alpha} y_j^2$ yields
\begin{align*}
    I_1
    &=  
    \prod_{j=1}^d 
    \int_{0}^{\lambda^{-1}}
    e^{q \lambda^{2-\alpha} y_j^2} dy_j
    =  
    \prod_{j=1}^d 
    \frac{1}{ \sqrt{q \lambda^{2-\alpha}}}
    \int_{0}^{\lambda^{-1}\sqrt{q \lambda^{2-\alpha}}}
    e^{v_j^2} dv_j\\
    &=  
    \inparen{
    \frac{
    \exp (\lambda^{-2} q \lambda^{2-\alpha} )
    }{ \sqrt{q \lambda^{2-\alpha}}}
    \mcD(\lambda^{-1}\sqrt{q \lambda^{2-\alpha}})
    }^d
    \asymp 
    \inparen{
    \frac{
    \exp (\lambda^{-2} q \lambda^{2-\alpha} )
    }{ \sqrt{q \lambda^{2-\alpha}}}
    \frac{1}{\lambda^{-1}\sqrt{q \lambda^{2-\alpha}}} }^d\\
    &=
    \inparen{
    \frac{
    \exp (\lambda^{-2} q \lambda^{2-\alpha} )
    }{ 
    \lambda^{-1} q \lambda^{2-\alpha}}}^d
    =
    \inparen{
    \frac{
    \exp (q \lambda^{-\alpha} )
    }{ 
    q \lambda^{1-\alpha}}}^d
    \asymp 
    \frac{
    e^{d q/\lambda^\alpha}
    }{ 
    q^d \lambda^{d(1-\alpha)}},
\end{align*}
where $\mcD(x)$ is the Dawson function, and we have used the fact that $ \mcD(x) \asymp x^{-1}$ for $x \to \infty$ 
(see the proof of Lemma~\ref{lem:TraceBound}). For the second integral, by switching to polar coordinates, we have 
\begin{align*}
    I_2 
    &\le 
    \int_{\R^d} |\tildek_1(\|y\|)|^q dy
    =
    \int_{0}^\infty r^{d-1} \int_{\mcS_{d-1}} |\tildek_1(r\|u\|)|^q d\mfs_{d-1}(u) dr\\
    &=
    \int_{0}^\infty r^{d-1} \tildek_1(r)^q
    \int_{\mcS_{d-1}}  d\mfs_{d-1}(u) dr
    = A(d) \int_{0}^\infty r^{d-1} \tildek_1(r)^q dr,
\end{align*}
where we have used the fact that $\normn{u}=1$ for any $u \in \mcS_{d-1}$, and $A(d)$ denotes the surface area of the unit sphere in $\R^d$. Therefore, we have that
\begin{align*}
    \sup_{x\in D} \int_D |k_\lambda(x,y)|^q dy
    &\le  \lambda^{2d}  
    e^{d q/\lambda^\alpha}
    I_1 \times I_2\\
    &\le  \lambda^{2d}  
    e^{d q/\lambda^\alpha}
    \frac{
    e^{d q/\lambda^\alpha}
    }{ 
    q^d\lambda^{d(1-\alpha)}}
    A(d) \int_{0}^\infty r^{d-1} \tildek_1(r)^q dr\\
    &=  q^{-d}\lambda^{d(1+\alpha)}  
    e^{2d q/\lambda^\alpha}
    A(d) \int_{0}^\infty r^{d-1} \tildek_1(r)^q dr.
\end{align*}

For the lower bound, note that 
\begin{align*}
    \sup_{x\in D} \int_D |k_\lambda(x,y)|^q dy 
    &\ge 
    \int_{D\times D}|k_\lambda(x,y)|^q dx dy
    =
    \int_{[0,1]^d \times [0,1]^d} \sigma_{\lambda}^q(x)\sigma_{\lambda}^q(y)|\tildek_\lambda( \normn{x-y} )|^q dx dy \\
    &=
    \int_{[0,1]^d \times [0,1]^d} \sigma_{\lambda}^q(x)\sigma_{\lambda}^q(y)|\tildek_1( \lambda^{-1} \normn{x-y} )|^q dx dy\\
    &=
    \lambda^{2d} 
    \int_{[0,\lambda^{-1}]^d \times [0,\lambda^{-1} ]^d} \sigma_{\lambda}^q(\lambda x')\sigma_{\lambda}^q(\lambda y')| \tildek_1(\normn{x'-y'} )|^q dx' dy',
\end{align*}
where the last line is due to the substitution $x'=\lambda^{-1}x$ and $y' = \lambda^{-1}y$. Now, let $w = x' - y'$ and $z = x' + y'$ and note that the Jacobian of this transformation is $2^{-d}$, and also that $x' = (z+w)/2$ and $y' = (z-w)/2$. Continuing from the last line of the above display, the substitution and the fact that $\sigma_\lambda(\lambda x') = \sigma_1(\lambda^{1-\alpha/2}x') = \exp(\lambda^{2-\alpha} \|x'\|^2)$ give that 
\begin{align*}
    &=
    \frac{\lambda^{2d}}{2^{d}  }
    \int_{[-\lambda^{-1},0]^d} 
    \inparen{
    \int_{ R_1 }
    \exp \inparen{\frac{q\lambda^{2-\alpha}}{4} \normn{z+w}^2 } 
    \exp \inparen{\frac{q\lambda^{2-\alpha}}{4} \normn{z-w}^2 }
     dz }
    | \tildek_1(\normn{w} )|^q dw\\
    & + 
    \frac{\lambda^{2d}}{2^{d}  }
    \int_{[0,\lambda^{-1}]^d} 
    \inparen{
    \int_{ R_2 }    
    \exp \inparen{\frac{q\lambda^{2-\alpha}}{4} \normn{z+w}^2 } 
    \exp \inparen{\frac{q\lambda^{2-\alpha}}{4} \normn{z-w}^2 }
     dz}
    | \tildek_1(\normn{w} )|^q dw,
\end{align*}
where we have defined
\begin{align*}
    R_1 &= \{z: -w_j \le z_j \le 2\lambda^{-1} + w_j, ~ 1 \le j \le d \},\\
    R_2 &= \{z: w_j \le z_j \le 2\lambda^{-1} - w_j, ~ 1 \le j \le d \}.
\end{align*}
For the integral over $R_1$, 
\begin{align*}
    &\int_{ R_1 }
    \exp \inparen{\frac{q\lambda^{2-\alpha}}{4} \normn{z+w}^2 } 
    \exp \inparen{\frac{q\lambda^{2-\alpha}}{4} \normn{z-w}^2 }
     dz \\
     &\ge 
     \int_{ R_1 }
    \exp \inparen{\frac{q\lambda^{2-\alpha}}{2} \normn{z+w}^2 }
     dz \\
     &=
     \int_{ R_3 }
    \exp \inparen{\frac{q\lambda^{2-\alpha}}{2} \normn{u}^2 }
     du\\
     &=
     \prod_{j =1}^d
     \int_{ 0}^{2(\lambda^{-1} + w_j) }
    \exp \inparen{\frac{q\lambda^{2-\alpha}}{2} u_j^2 } 
     du_j,\\
\end{align*}
where the inequality holds by the fact that $w \in [-\lambda^{-1},0]^d$,
and the second line holds by making the substitution $u=z+w$ and defining $R_3= \{u: 0 \le u_j \le 2(\lambda^{-1} + w_j), ~ 1 \le j \le d \}$. Now, letting $v_j^2 := \frac{q\lambda^{2-\alpha}}{2} u_j^2$ gives 
\begin{align*}
    \prod_{j =1}^d
   &  \int_{ 0}^{2(\lambda^{-1} + w_j) }
    \exp \inparen{\frac{q\lambda^{2-\alpha}}{2} u_j^2 } 
     du_j
     =
     \prod_{j =1}^d
     \sqrt{\frac{2}{q\lambda^{2-\alpha}}}
     \int_{ 0}^{2(\lambda^{-1} + w_j) \sqrt{\frac{q\lambda^{2-\alpha}}{2}  }}
    e^{ v_j^2 } 
     dv_j\\
     &=
     \prod_{j =1}^d
     \sqrt{\frac{2}{q\lambda^{2-\alpha}}}
     \exp \inparen{
     4(\lambda^{-1} + w_j)^2 \frac{q\lambda^{2-\alpha}}{2}
     }
     \mcD \inparen{ 2(\lambda^{-1} + w_j) \sqrt{\frac{q\lambda^{2-\alpha}}{2}  } }\\
     &\asymp
     \prod_{j =1}^d
     \sqrt{\frac{2}{q\lambda^{2-\alpha}}}
     \exp \inparen{
     4(\lambda^{-1} + w_j)^2 \frac{q\lambda^{2-\alpha}}{2}
     }
     \frac{1}{ (\lambda^{-1} + w_j) \sqrt{\frac{q\lambda^{2-\alpha}}{2}  } }\\
     &\asymp
     \prod_{j=1}^d
     \exp \inparen{
     2 q (1 + \lambda w_j)^2 /\lambda^{\alpha}
     }
     \frac{1}{ q(\lambda^{-1} + w_j) \lambda^{2-\alpha}  }\\
     &=
     q^{-d} \lambda^{d(\alpha-1)} \prod_{j=1}^d \exp \inparen{
     2 q (1 + \lambda w_j)^2 /\lambda^{\alpha}
     } \frac{1}{ (1 + \lambda w_j)  },
\end{align*}
where we have used the same properties of the Dawson function as in the upper bound. Therefore, we have so far shown that 
\begin{align*}
    &\int_{ R_1 }
    \exp \inparen{\frac{q\lambda^{2-\alpha}}{4} \normn{z+w}^2 } 
    \exp \inparen{\frac{q\lambda^{2-\alpha}}{4} \normn{z-w}^2 }
     dz \\
     & \hspace{4cm}
     \gtrsim 
     q^{-d}
    \lambda^{d(\alpha-1)} \prod_{j=1}^d \exp \inparen{
     2 q (1 + \lambda w_j)^2 /\lambda^{\alpha}
     } \frac{1}{ (1 + \lambda w_j)  }.
\end{align*}
For the integral over $R_2$ a similar argument shows that 
\begin{align*}
    &\int_{ R_2 }
    \exp \inparen{\frac{q\lambda^{2-\alpha}}{4} \normn{z+w}^2 } 
    \exp \inparen{\frac{q\lambda^{2-\alpha}}{4} \normn{z-w}^2 }
     dz\\ 
     & \hspace{4cm}
     \gtrsim 
     q^{-d}
    \lambda^{d(\alpha-1)} \prod_{j=1}^d \exp \inparen{
     2  q (1 - \lambda w_j)^2 /\lambda^{\alpha}
     } \frac{1}{ (1 - \lambda w_j)  }.
\end{align*}

Putting the two bounds together, we have that
\begin{align*}
    &\sup_{x\in D} \int_D |k_\lambda(x,y)|^q dy \\
    &\ge 
    q^{-d}\lambda^{2d}  
    \int_{[-\lambda^{-1},\lambda^{-1}]^d} 
    \lambda^{d(\alpha-1)} 
    \prod_{j=1}^d \exp \inparen{
     2 q (1 - \lambda |w_j|)^2 /\lambda^{\alpha}
     } \frac{1}{ (1 - \lambda |w_j|)  }
     |\tildek_1(\|w\|)|^q
     dw
     \\
     &\asymp 
     q^{-d}
      \lambda^{d(1+\alpha)}  
    \int_{\R^d} 
    e^{2 d q /\lambda^{\alpha}} |\tildek_1(\|w\|)|^q dw\\
    &=
    q^{-d}
     \lambda^{d(1 + \alpha)} e^{2qd/\lambda^\alpha}   A(d) \int_0^\infty r^{d-1} \tildek_1(r) dr,
\end{align*}
where the second line holds by the Dominated Convergence Theorem as $\lambda \to 0^+$. The final line follows as in the proof of the upper bound. 

\end{proof}

\begin{lemma}\label{lem:SparsityParameterCharacterization} 
The kernel $k_\lambda$ satisfies Assumption \ref{ass:mainAssumption} with
    \begin{align*}
        R_q^q
        \asymp 
        q^{-d}
        \lambda^{d(\alpha + 1)} 
    e^{2 d \lambda^{-\alpha}}
     \int_0^\infty r^{d-1} \tildek_1(r)^q dr,
    \qquad \lambda \to 0^+.
    \end{align*}
\end{lemma}
\begin{proof}
The result follows by noting that 
\begin{align*}
    \sup_{x\in D} \int_D (k_\lambda(x,x)k_\lambda(y,y))^{(1-q)/2}|k_\lambda(x,y)|^q dy = 
    \tildek(0)
    \sup_{x\in D} \int_D  \sigma_\lambda(x) \sigma_\lambda(y) |\tildek_\lambda(x,y)|^q dy,
\end{align*}
and using an identical approach to the one in Lemma~\ref{lem:supklambda_positiveexponential}.
\end{proof}

\begin{lemma}\label{lem:CovOpBound}
It holds that
    \begin{align*}
        \normn{C}
        \asymp 
        \lambda^{d(\alpha + 1)} 
    e^{2 d \lambda^{-\alpha}}
     \int_0^\infty r^{d-1} \tildek_1(r) dr,
    \qquad \lambda \to 0^+.
    \end{align*}
\end{lemma}
\begin{proof}
    The proof follows in an identical way to that of \cite[Lemma 4.2]{al2023covariance}, but invokes our novel characterization in  Lemma~\ref{lem:supklambda_positiveexponential} instead of \cite[Lemma 4.1]{al2023covariance}.
\end{proof}

\begin{proof}[Proof of Theorem \ref{thm:CompareAdaptiveAndSampleCov}]
    For the sample covariance, note that by \cite[Theorem 9]{koltchinskii2017concentration}, for any $t \ge 1$ it holds with probability at least $1-e^{-t}$ that 
    \begin{align*}
        \frac{\normn{\hatC- C}}{\normn{C}} 
        \asymp 
        \sqrt{\frac{r_2(C)}{N}} \lor \frac{r_2(C)}{N} \lor \sqrt{\frac{t}{N}} \lor \frac{t}{N}, 
        \qquad r_2(C) := \frac{\ttrace(C)}{\normn{C}}.
    \end{align*}
    By Lemmas~\ref{lem:TraceBound} and \ref{lem:CovOpBound}, we have that for sufficiently small $\lambda$, 
    \begin{align*}
        r_2(C) \asymp
        \frac{\tildek_1(0) \lambda^{\alpha d} 
        e^{2d/\lambda^\alpha}}{ \lambda^{d(\alpha+1)} e^{2d/\lambda^\alpha}}
        \asymp 
        \lambda^{-d}.
    \end{align*}
     Therefore, with probability at least $1-e^{-\log(\lambda^{-d})} = 1-\lambda^d$ 
    \begin{align*}
        \frac{\normn{\hatC- C}}{\normn{C}} 
        &\asymp 
        \sqrt{\frac{\lambda^{-d}}{N}} \lor \frac{\lambda^{-d}}{N} 
        \lor \sqrt{\frac{\log(\lambda^{-d})}{N}} \lor \frac{\log(\lambda^{-d})}{N}
         =
        \sqrt{\frac{\lambda^{-d}}{N}} \lor \frac{\lambda^{-d}}{N}.
    \end{align*}
    For the adaptive estimator, note first that the normalized process $\tildeu(x) = u(x)/k^{1/2}(x,x)$ is isotropic with covariance function $\tildek_\lambda(\|x-y\|).$ Then by \cite[Lemma 4.3]{al2023covariance}, we have that $\E[\sup_{x\in D} \tildeu(x)] \asymp \sqrt{\log(\lambda^{-d})}$ for sufficiently small $\lambda.$ By Theorem~\ref{thm:ThresholdOpNormBoundwSampleRho}, Lemma~\ref{lem:CovOpBound}, and Lemma~\ref{lem:SparsityParameterCharacterization}, we then have with probability at least $1-e^{-\log(\lambda^{-d})} = 1-\lambda^d$
    \begin{align*}
        \normn{\hatC_{\hat{\rho}_N}- C} 
        &\lesssim 
        R_q^q \inparen{\frac{\log(\lambda^{-d})}{{\sqrt{N}}} }^{1-q}\\
        &\asymp 
        q^{-d}
        \lambda^{d(\alpha + 1)} 
    e^{2 d \lambda^{-\alpha}}
     \int_0^\infty r^{d-1} \tildek_1(r)^q dr
        \inparen{\frac{\log(\lambda^{-d})}{{\sqrt{N}}} }^{1-q}\\
        &\asymp 
        q^{-d} \|C\|
     \frac{\int_0^\infty r^{d-1} \tildek_1(r)^q dr}{\int_0^\infty r^{d-1} \tildek_1(r) dr}
        \inparen{\frac{\log(\lambda^{-d})}{{\sqrt{N}}} }^{1-q}, \qquad \lambda\to 0^+.
    \end{align*}
    Rearranging yields the result with $c(q) := q^{-d} \frac{\int_0^\infty r^{d-1} \tildek_1(r)^q dr}{\int_0^\infty r^{d-1} \tildek_1(r) dr}.$
\end{proof}

\section{Conclusions and Future Work}\label{sec:Conclusions}
In this paper, we have studied covariance operator estimation under a novel sparsity assumption, developing a new nonasymptotic and dimension-free theory. Our model assumptions capture a particularly challenging class of nonstationary covariance models, where the marginal variance may vary significantly over the spatial domain. Adaptive threshold estimators are then shown to perform well over this class from both a theoretical and experimental perspective. The theory developed in this work as well as the connections made to both the (finite) high-dimensional literature and the functional data analysis literature open the door to many interesting avenues for future work, which we outline in the following:
\begin{itemize}
    \item \textit{Inference for covariance operators}: The techniques developed in this work, in particular the dimension-free bounds for the sample covariance and variance component functions open the door to lifting the finite high-dimensional inference theory to the infinite dimensional setting. The review paper \cite{cai2017global} provides a detailed overview of the covariance inference literature in finite dimensions. Our analysis can likely be used to extend this theory to the covariance operator setting. To the best of our knowledge, the existing literature on inference for covariance operators (see for example \cite{panaretos2010second, kashlak2019inference}) does not account for potential sparse structure in the underlying operators, though such structure can naturally be assumed in many cases of interest.

    \item \textit{Estimating the covariance operator of a multi-valued Gaussian process}: As described in Section~\ref{sec:LitReviewInfiniteDims}, the paper \cite{fang2023adaptive} considers covariance estimation for multi-valued Gaussian processes under a sparsity assumption on the dependence between the individual component of the process. In contrast to our approach (see also Remark~\ref{rem:globalSparsityAssumption}) their assumption does not impose any sparsity on each component  covariance function. It would be interesting therefore to extend our analysis to the multi-valued functional data analysis setting where each component satisfies a local-type sparsity constraint, such as belonging to $\mcK^*_q$.
\end{itemize}

\section*{Acknowledgments}
The authors are grateful for the support of the NSF CAREER award DMS-2237628, DOE DE-SC0022232, and the BBVA Foundation. The authors are also thankful to Jiaheng Chen and Subhodh Kotekal for inspiring discussions.

\begin{appendix}
\section{Additional Numerical Simulations}\label{app:Dim2Simulations}
This appendix shows random draws and results of operator estimation in $d=2$. The experimental set-up is identical to the $d=1$ setting described in Section~\ref{sec:empiricalnew}, however due to computational constraints in the higher dimensional setting, we use different settings for the experimental parameters. Specifically, our samples are generated by discretizing the domain $D=[0,1]^2$ with a uniform mesh of $L=10,000$ points. We consider a total of 10 choices of $\lambda$ arranged uniformly in log-space and ranging from $10^{-2}$ to $10^{-0.1}$. As in the $d=1$ case, the plots indicate that adaptive thresholding is a significant improvement over both the sample covariance and universal thresholding estimators in both the unweighted and weighted settings.

\begin{figure}
    \centering
\includegraphics[width=0.85\textwidth]{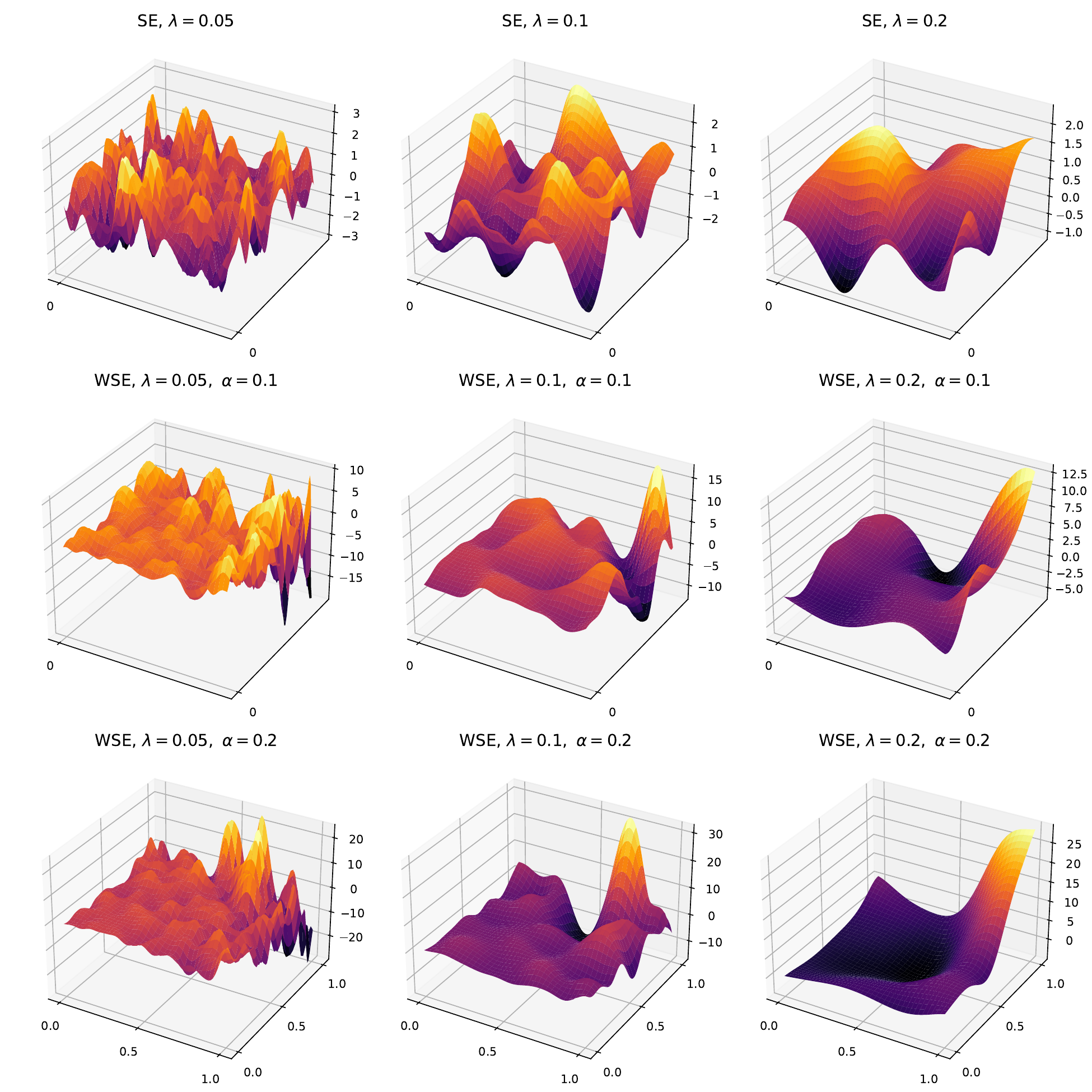}
    \caption{Draws from a centered Gaussian process on $D=[0,1]^2$ with covariance function SE in the first row, WSE($\alpha=0.1$) in the second and WSE($\alpha=0.2$) in the third, with varying $\lambda$ parameter.
    }
    \label{fig:varyingAlphad2}
\end{figure}

\begin{figure}
    \centering
\includegraphics[scale=0.5]{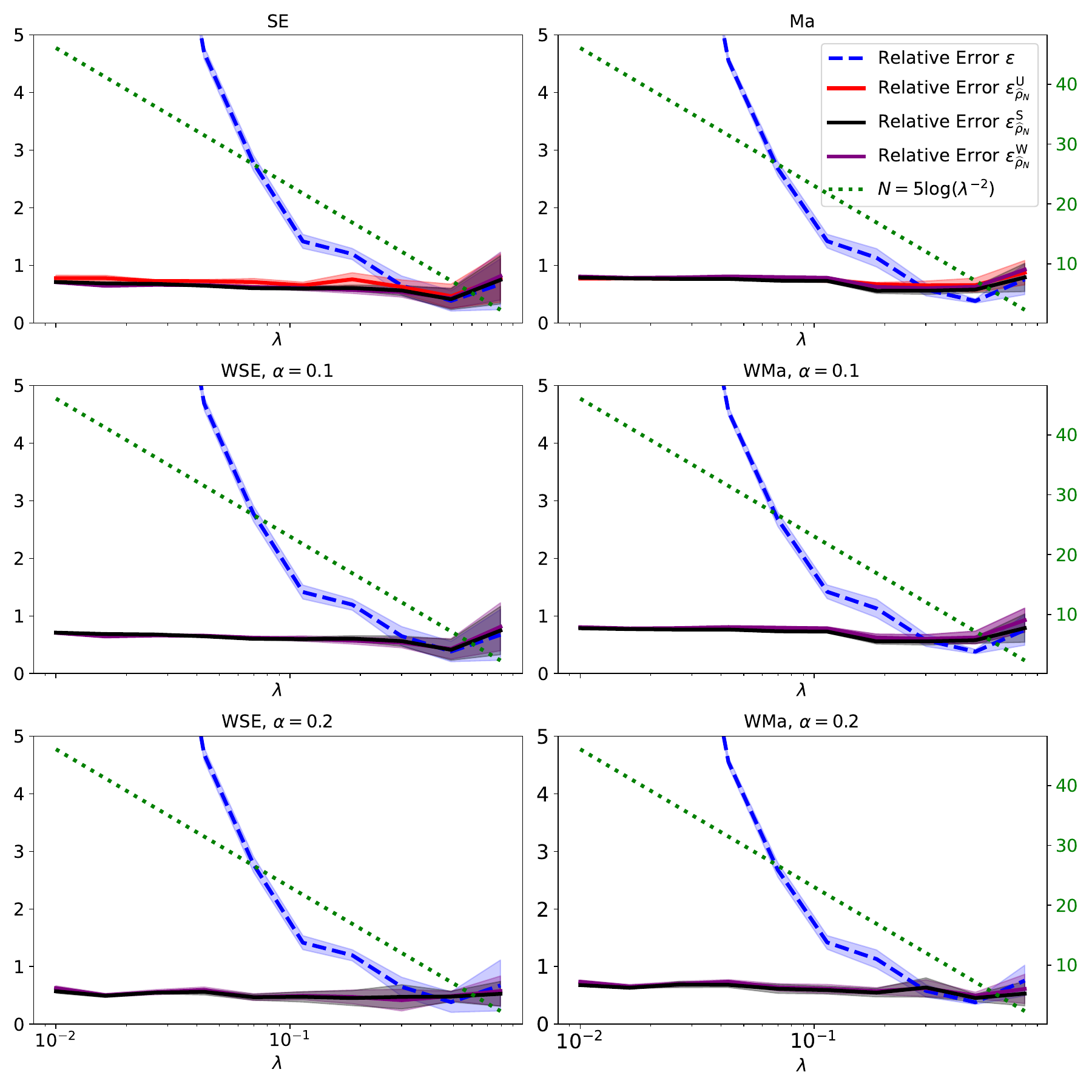}
    \caption{Plots of the average relative errors and 95\% confidence intervals achieved by the sample ($\varepsilon$, dashed blue), universal thresholding ($\varepsilon^{\mathsf{U}}_{\hat{\rho}_N}$, red), sample-based adaptive thresholding ($\varepsilon^{\mathsf{S}}_{\hat{\rho}_N}$, black) and Wick's adaptive thresholding ($\varepsilon^{\mathsf{W}}_{\hat{\rho}_N}$, purple) covariance estimators based on a sample size ($N$, dotted green) for the (weighted) squared exponential (left) and (weighted) Matérn (right) covariance functions in $d=2$ over 10 Monte-Carlo trials and 10 scale parameters $\lambda$ ranging from $10^{-2}$ to $10^{-0.1}$. The first row corresponds to the unweighted covariance functions and is the only case in which the universal thresholding estimator is considered; the second and third rows correspond to the weighted variants with $\alpha=0.1, 0.2$ respectively.}
    \label{fig:d2Results}
\end{figure}

\section{Sub-Gaussian Process Calculations}\label{app:subgCalcs}
Let $v^{(1)}, v^{(2)}$ denote independent centered Gaussian processes both with covariance function $k^v$ of the form \eqref{eq:nonstationaryKernels}, such that $k^v(x,y) = \sigma_\lambda(x)\sigma_\lambda(y)\tildek_\lambda(x,y).$ Consider the process $u$ obtained by transforming $v^{(1)}, v^{(2)}$ and let $m^u, k^u$ denote the mean and covariance functions of $u,$ respectively. In the following sections, we provide explicit expressions for $m^u, k^u$ for various transformations. These results are utilized in the sub-Gaussian portion of Section~\ref{ssec:SimulationResults}.

\subsection{Sine Function}\label{ssec:Sine}
  Let $u := \sin(v^{(1)})$. As $v^{(1)}$ is centered, and $\sin(\cdot)$ is an odd function, $\E[\sin(v^{(1)}(x))]=0$ for any $x \in D$, implying $m^u =0.$ Further, since $\sin(\cdot)$ is bounded, $u$ is a sub-Gaussian process.  Note then that
\begin{align*}
    k^u(x,y)
    &= \E[\sin(v^{(1)}(x))\sin(v^{(1)}(y))]\\
    &= \frac{1}{2} \E[\cos(v^{(1)}(x)-v^{(1)}(y))]
    -
    \frac{1}{2} \E[\cos(v^{(1)}(x)+v^{(1)}(y))],
\end{align*}
where we have made use of the identity $\sin(a)\sin(b) = \frac{1}{2}(\cos(a-b) - \cos(a+b)).$ Recall that $v^{(1)}(x)-v^{(1)}(y) \sim N(0,k^v(x,x)+k^v(y,y) - 2k^v(x,y))$ and $v^{(1)}(x)+v^{(1)}(y) \sim N(0,k^v(x,x)+k^v(y,y) + 2k^v(x,y))$. Further note that for $Z \sim N(0, \tau^2),$ $\E[\cos(Z)] = e^{-\tau^2/2}.$ Therefore, 
\begin{align*}
     k^u(x,y)
     &= \frac{1}{2} e^{-\frac{1}{2}(k^v(x,x)+k^v(y,y) - 2k^v(x,y)))}
    -
    \frac{1}{2} e^{-\frac{1}{2}(k^v(x,x)+k^v(y,y) + 2k^v(x,y)))}\\
    &=
    e^{-\frac{1}{2}(k^v(x,x)+k^v(y,y))} \sinh(k^v(x,y))\\
    &= e^{-\frac{1}{2}(\sigma_\lambda^2(x)+\sigma_\lambda^2(y))} \sinh(\sigma_\lambda(x) \sigma_\lambda(y) \tildek^v_\lambda(x,y)).
\end{align*}

\subsection{Absolute Value Function}\label{ssec:absval}
Let $u := |v^{(1)}|$. Since $|\cdot|$ is Lipschitz, $u-m^u$ is a sub-Gaussian process. Direct calculation yields $m^u(x) = \E|v^{(1)}(x)|= \sqrt{\frac{2}{\pi} k^v(x,x)} = \sigma_\lambda(x) \sqrt{\frac{2}{\pi}}$ for any $x \in D.$ Recall that for any $x,y \in D$, $(v^{(1)}(x),v^{(1)}(y))$ is a centered bi-variate Gaussian vector with $\E [v^{(1)}(x)v^{(1)}(y)] = \sigma_\lambda(x) \sigma_\lambda(y) \tildek_\lambda(x,y),$ by {\cite[Corollary 3.1]{li2009gaussian}} 
\begin{align*}
    \E [|v^{(1)}(x)||v^{(1)}(y)|]
    = \frac{2 \sigma_\lambda(x)\sigma_\lambda(y)}{\pi} 
    \inparen{ \sqrt{1-\tildek_\lambda^2(x,y)}
    + \tildek_\lambda(x,y) \sin^{-1} (\tildek_\lambda(x,y))
    }.
\end{align*}
Therefore 
\begin{align*}
    k^u(x,y) 
    &=
    \frac{2 \sigma_\lambda(x)\sigma_\lambda(y)}{\pi}
    \inparen{
    \sqrt{1-\tildek_\lambda^2(x,y)}
    + \tildek_\lambda(x,y) \sin^{-1} (\tildek_\lambda(x,y))
    -1}.
\end{align*}

\subsection{Absolute Value $\times$ Sine Function}\label{ssec:absvaluetimessine}
Let $u:=(|v^{(1)}|-\E|v^{(1)}|) \sin(v^{(2)}).$ By \ref{ssec:absval} and the fact that $\sin(v^{(2)})$ is bounded, $u$ is the product of a sub-Gaussian process and a bounded process, and so is itself sub-Gaussian.

Next, by independence of $v^{(1)}$ and $v^{(2)}$, $m^u =0.$ Recall that the product of two independent stochastic processes with covariance functions $k^1,k^2$ has covariance function $k^1 k^2$. Therefore, by the derivations in \ref{ssec:Sine} and \ref{ssec:absval}, for any $x,y \in D,$
\begin{align*}
     k^u(x,y) 
     &= \frac{2 \sigma_\lambda(x)\sigma_\lambda(y)}{\pi}
    \inparen{
    \sqrt{1-\tildek_\lambda^2(x,y)}
    + \tildek_\lambda(x,y) \sin^{-1} (\tildek_\lambda(x,y))}
    \\&
    \times 
     e^{-\frac{1}{2}(\sigma_\lambda^2(x)+\sigma_\lambda^2(y))} \sinh(\sigma_\lambda(x) \sigma_\lambda(y) \tildek^v_\lambda(x,y)).
\end{align*}

\end{appendix}

\bibliographystyle{abbrvnat}  
\bibliography{references}

\end{document}